\documentclass[reqno]{amsart}

\usepackage[T1]{fontenc}
\usepackage{lmodern}
\usepackage{bold-extra}

\usepackage[foot]{amsaddr}
\usepackage{amsmath,amssymb,amsfonts,mathrsfs}
\usepackage{comment}
\usepackage{verbatim}
\usepackage{bbold}
\usepackage{tikz-cd}
\usepackage{calligra}
\usepackage[english]{babel}
\usepackage[utf8]{inputenc}
\usepackage{mathtools}
\usepackage{comment}
\usepackage{cancel}
\usepackage{esint}

\makeatletter
\renewcommand{\part}{\@startsection{part}{1}%
  \z@{.7\linespacing\@plus\linespacing}{.5\linespacing}%
  {\large\bfseries\centering}} 
\makeatother

\usepackage[
top=3cm,
bottom=3cm,
left=3cm,
right=3cm,
marginparwidth=2.5cm
]{geometry}

\usepackage{todonotes}
\usepackage[hypertexnames=false]{hyperref}
\hypersetup{
    colorlinks,
    linkcolor={red!50!black},
    citecolor={blue!50!black},
    urlcolor={blue!80!black},
}

\usepackage[shortlabels]{enumitem}
\setlist[1]{itemsep=2pt}
\usepackage[capitalise]{cleveref}

\usepackage{autonum} 

\newtheorem{theorem}{Theorem}[section]
\newtheorem{prop}[theorem]{Proposition}
\newtheorem{corollary}[theorem]{Corollary}
\newtheorem{lemma}[theorem]{Lemma}
\newtheorem{defi}[theorem]{Definition}

\theoremstyle{remark} 
\newtheorem{remark}[theorem]{Remark}
\newtheorem{notation}[theorem]{Notation}
\newtheorem{example}[theorem]{Example}

\AddToHook{env/prop/begin}{\crefalias{theorem}{prop}}
\AddToHook{env/corollary/begin}{\crefalias{theorem}{corollary}}
\AddToHook{env/lemma/begin}{\crefalias{theorem}{lemma}}
\AddToHook{env/defi/begin}{\crefalias{theorem}{defi}}
\AddToHook{env/remark/begin}{\crefalias{theorem}{remark}}
\AddToHook{env/notation/begin}{\crefalias{theorem}{notation}}
\AddToHook{env/example/begin}{\crefalias{theorem}{example}}

\crefname{theorem}{Theorem}{Theorems}
\crefname{prop}{Proposition}{Propositions}
\crefname{corollary}{Corollary}{Corollaries}
\crefname{lemma}{Lemma}{Lemmas}
\crefname{defi}{Definition}{Definitions}
\crefname{remark}{Remark}{Remarks}
\crefname{notation}{Notation}{Notations}
\crefname{example}{Example}{Examples}

\newcommand{\R}{\mathbb{R}}
\newcommand{\N}{\mathbb{N}}
\newcommand{\dis}{\mathcal D}
\newcommand{\eps}{\ve}
\newcommand{\hei}{\mathbb{H}^1}
\newcommand{\ve}{\varepsilon}

\newcommand{\foc}{\mathrm{foc}}

\newcommand{\cut}{\mathrm{cut}}
\newcommand{\de}{{\mathrm d}}
\DeclareMathOperator{\vol}{\mathrm{vol}}

\DeclareMathOperator{\Si}{Si}
\newcommand{\owedge}{\mathbin{\mathpalette\owedge@{}}}
\newcommand{\Leb}{\mathscr L}
\newcommand{\sr}{sub-Riemannian }
\newcommand{\Max}{{\rm Max}}
\newcommand{\proj}{{\mathfrak p}}

\newcommand{\supp}{\operatorname{supp}}
\newcommand{\D}{\mathcal{D}}

\newcommand{\Fp}{\operatorname{Fp}}

\title[]
{Curvature distributions and the Steiner formula \\   for characteristic surfaces in the Heisenberg group}
\author[D. Barilari]{Davide Barilari}
\address[Davide Barilari, Eugenio Bellini]{Dipartimento di Matematica ``Tullio Levi-Civita'', Universit\`a degli Studi di Padova,\ 
\textbf{{davide.barilari@unipd.it,\, eugenio.bellini@unipd.it}}}

\author[E. Bellini]{Eugenio Bellini}

\author[T.~Rossi]{Tommaso Rossi}
\address[Tommaso Rossi]{Universit\`a degli Studi dell'Aquila, \ 
\textbf{tommaso.rossi1@univaq.it}}

\begin{document}
 \begin{abstract}

We establish a new Steiner-type formula for smooth surfaces in the sub-Riemannian Heisenberg group.  The main novelty is that we admit characteristic points, whose presence produces nontrivial singularities both of the distance from the surface and of its tubular neighborhood. 

 Such singularities are reflected in the Steiner formula: the expansion up to a given order $N\geq 4$ of the volume of the $\eps$-tubular neighborhood $\mathcal T_S(\ve)$ of a surface $S$ is no longer polynomial, in general, taking the form
\begin{equation}
        \mathrm{vol}(\mathcal T_S(\ve)) = \sum_{j=1}^3\alpha_j\ve^j+\sum_{n=4}^N \left(\alpha_n+\beta_n\log\ve\right)\ve^n+o(\ve^N),\qquad\text{as }\ve\to0^+ .
    \end{equation}  
More precisely, under a \emph{geometric stability} condition on characteristic points, we prove that the first three coefficients are integrals over $S$ of locally integrable curvature densities, coinciding with those known in the non-characteristic case \cite{BaloghSteiner,BarBoss2024}. 

A new phenomenon appears at order $4$, equal to the homogeneous dimension of the Heisenberg group, where logarithmic terms may occur.    
We show that the higher-order coefficients are naturally described by  a distributional analogue of Federer’s curvature measures, whose logarithmic components are supported on the characteristic set.  We compute the fourth-order coefficients explicitly.

The proof is based on different ingredients of independent interest: (i) an explicit formula for the  cut time of homogeneous surfaces, (ii) a fine analysis of the normal exponential map through blow-ups, (iii) the study of the stability of the surface cut locus under perturbations. 
\end{abstract}

\maketitle

\tableofcontents

\section{Introduction}
\noindent A classical result in integral geometry is Steiner's formula \cite{Steiner,WeylTubes,GrayTubes}. In its Euclidean form, it describes the volume of a small neighborhood of a regular domain as a polynomial in the radius. If $\Omega\subset \mathbb R^n$ is a bounded domain with smooth boundary $S=\partial\Omega$, and $\Omega_\ve$ denotes its $\ve$-neighborhood, then there is $\ve_0>0$ such that
\begin{equation}
	\mathrm{vol}(\Omega_\ve)
	=
	\mathrm{vol}(\Omega)
	+
	\sum_{k=1}^n a_k(S)\ve^k,\qquad\forall\,\ve\in[0,\ve_0],
\end{equation}
The coefficients $a_k(S)$, also known as Minkowski quermassintegrals, can be expressed as integrals over $S$ of symmetric homogeneous polynomials of its principal curvatures.

A closely related statement is the so-called half-tube formula. If $S\subset \mathbb R^n$ is a compact embedded hypersurface, boundary of a regular domain $S=\partial\Omega$, one may consider the \emph{half-tube} of radius $\ve>0$
\begin{equation}
	\mathcal T_S(\ve)
	=
	\Omega_\ve\setminus \Omega,
\end{equation}
which is the set of points whose distance from $S$ is at most $\ve$, lying outside of $\Omega$. In this case the  half-tube formula reduces to the non-constant part of Steiner's formula:
\begin{equation}\label{eq:steiner_riem}
	\mathrm{vol}(\mathcal T_S(\ve))	=	\sum_{k=1}^n a_k(S)\ve^k,\qquad\forall\,\ve\in[0,\ve_0].
\end{equation}
In the  simplest case of a surface $S$ embedded in $\mathbb R^3$, the half-tube formula for small radii is written as
\begin{equation}\label{eq:steiner_riem2}
	\mathrm{vol}(\mathcal{T}_S(\ve)) = \ve\int _SdA-\ve^2\int _SHdA+\frac{\ve^3}{3}\int_SKdA,
\end{equation}
where $H$ and $K$ are the mean and Gaussian curvature of $S$ and $dA$ is the Riemannian area measure. 

If the ambient space is an $n$-dimensional Riemannian manifold, the Riemannian volume of $\mathcal T_S(\ve)$ is no longer, in general, a polynomial in $\ve$. Nevertheless, it admits an asymptotic expansion in powers of $\ve$, as $\ve\to 0$. The coefficients are expressed in terms of the Riemannian curvature of the ambient manifold and of the extrinsic geometry of
the submanifold; see, for instance, \cite{GrayTubes}.

Tube formulae have been studied beyond the smooth setting. Federer’s theory of sets of positive reach provides a powerful framework in which Steiner-type formulae continue to hold \cite{FedererCM} and coefficients can be interpreted as the total mass of the so-called \emph{curvature measures}.

\subsection{The sub-Riemannian setting} In the sub-Riemannian setting, the study of embedded surfaces and of the volumes of their tubular neighborhoods has recently begun to attract new attention. For general results on embedded surfaces in contact sub-Riemannian manifolds see e.g.\ \cite{Daniele,BellBosc,ABBR2025,Lucia,Bellini2026}, while for volumes of tubes around non-characteristic submanifolds see \cite{BaloghSteiner,RitoreTubes,BarBoss2024,BossioRizziRossi2024}. One of the main motivations is to understand whether curvature-like invariants of submanifolds can be recovered from the coefficients of the corresponding asymptotic expansions.

This theme is particularly interesting even in the simplest case of a surface $S$ embedded in the sub-Riemannian Heisenberg group $\mathbb{H}^1$, which is $\mathbb R^3$ endowed with the contact distribution 
\begin{equation} \label{eq:framei}
	\mathcal D=\mathrm{span}\{X_1,X_2\},\qquad X_1:=\partial_x-\frac{y}{2}\partial_z,\qquad X_2:=\partial_y+\frac{x}{2}\partial_z,
\end{equation}
and the sub-Riemannian metric $g$ is defined by promoting the pair $X_1,X_2$ to an orthonormal frame.

The main source of difficulty in extending Steiner formula to the sub-Riemannian setting is given by the presence of
\emph{characteristic points}, namely points at which the horizontal distribution
is tangent to the surface. These points define the characteristic set
\begin{equation}
	\Sigma(S):=\{p\in S:\ T_pS=\mathcal D_p\}.
\end{equation}
As we observe in \cref{sec:distance_from_S}, in any neighborhood of a characteristic point there is at least one point reached by more than one minimizing geodesic. Thus, the presence of characteristic points forces the surface $S$ to have zero (sub-Riemannian) reach in the sense of Federer, falling outside the positive reach framework.

For a compact surface $S$ embedded in the Heisenberg group $\mathbb{H}^1$, under the assumption that $S$ is \emph{non-characteristic}, i.e., $\Sigma(S)=\emptyset$, a formula analogous to \eqref{eq:steiner_riem2} was obtained in \cite{BaloghSteiner}. Their formula takes the form
\begin{equation}\label{eq:balogh_formula}
	\mathrm{vol}(\mathcal{T}_S(\ve)) = \sum_{n=1}^\infty  a_n(S)\, \ve^n, 
\end{equation}
with coefficients expressed as integrals of the so called \emph{curvature forms}, introduced below.

In this paper we obtain a Steiner type formula for smooth compact surfaces with non-degenerate, geometrically stable characteristic points, with these conditions specified below. The essential issue in computing the half-tube volume integral is that characteristic points produce non trivial singularities both of the distance from the surface and of its tubular neighborhood, as we now explain.

\subsection{The tube integral and the curvature forms}

Let $S=\partial\Omega\subset\hei$ be a smooth embedded surface oriented towards the complement of $\Omega$. We denote by $\mathrm{dist}_S$ the signed sub-Riemannian distance from $\Omega$, positive outside of it, and for $\ve> 0$, consider the positive half-tube
\begin{equation}
	\mathcal T_S(\ve)=\{q\in\hei:0<\mathrm{dist}_S(q)\leq\ve\}.
\end{equation}
At every $p\in S\setminus\Sigma(S)$ there is a unique unit-speed normal
geodesic $\gamma_p$ starting from $p$, which for sufficiently short times is valued in $\hei\setminus\Omega$ and realizes the distance from
$S$. The normal exponential map is then defined as 
\begin{equation}\label{eq:intro-normal-exp}
	\exp^\perp_S:\mathbb R^+\times S\setminus\Sigma(S)\to \hei,\qquad  \exp_S^\perp(t,p):=\gamma_p(t).
\end{equation}
Although smooth on its domain, the normal exponential map does not admit a smooth extension across the characteristic set (see \cref{sec:normal_exp}). This is consistent
with the fact that no unit-speed geodesic realizing the distance
from $S$ starts at a characteristic point (see \cref{lem:no_abn}).
The maximal interval of length-minimality of a normal geodesic defines the cut time
\begin{equation}
	t_\cut(p)
	:=\sup\{t>0:\mathrm{dist}_S(\gamma_p(t))=t\},\qquad p\in S\setminus\Sigma(S).
\end{equation}
Since every point in $\hei\setminus \Omega$ is reached by at least one minimizing geodesic, we can parametrize the tube as
\begin{equation}
	\mathcal T_S(\ve)=\{\exp_S^\perp(t,p)\,:\, p\in S\setminus\Sigma(S),\,0<t\leq \min\{\ve,t_\cut(p)\}\}.
\end{equation}
The same formula, with $S$ replaced by a compact subset $U\subset S$, defines the localized half-tube $\mathcal T_U(\ve)$, consisting of points reached by minimizers based in $U\subset S$. In \cref{sec:distance_from_S} we show that such parametrization is actually a diffeomorphism outside of a negligible set. The change of variables formula then gives
\begin{equation}\label{eq:intro-tube-integral}
	\vol(\mathcal T_U(\ve))	=\int_{U\setminus\Sigma(S)}	\int_0^{\min\{\ve,t_\cut(p)\}}	J_S(t,p)\,dt\wedge dp,
\end{equation}
where we have denoted 
\begin{equation}
    J_S(t,p)\,dt\wedge dp:=(\exp_S^\perp)^* (dx\wedge dy\wedge dz),\qquad p\in S\setminus\Sigma(S).
\end{equation}
 Note that, by the change of variable formula, $J_S(t,p)$ is the Jacobian determinant of the normal exponential map computed in the coordinates $(t,p)\in \mathbb R^+\times (S\setminus\Sigma(S))$. 
 
 If $U$ is non-characteristic, the cut time is uniformly bounded below. Thus, for all sufficiently small $\ve$, the minimum in \eqref{eq:intro-tube-integral} can be replaced by $\ve$. Then, as rigorously established in \cite{BaloghSteiner,BossioRizziRossi2024}, a Steiner formula for non-characteristic surfaces can be obtained by Taylor expanding $J_S(t,p)$ in time along the normal geodesics and integrating term by term. More explicitly, writing
\begin{equation}\label{eq:intro-curv-fomrs}
	\sigma_S^{(n)}:=\partial_t^nJ_S(0,p)dp,\qquad p\in S, \,n\in\mathbb N,
\end{equation}
using $U\setminus\Sigma(S)=U$ and $\min\{t_\cut(p),\ve\}=\ve$, and exchanging the order of integration in \eqref{eq:intro-tube-integral} one finds
\begin{equation}\label{eq:intro-non-char-tube}
		\vol(\mathcal T_U(\ve))
	=\sum_{n=1}^\infty\ve^n a_n(U),\qquad a_n(U):=\frac{1}{n!}\int_U\sigma_S^{(n-1)}.
\end{equation}
We call $\sigma_S^{(n)}\in \Lambda^2(S\setminus\Sigma(S))$ the $n$-th \emph{curvature form}. In particular,
$\sigma_S^{(0)}=\sigma_H$ is the horizontal area form. If $N$ denotes the horizontal gradient of the signed distance from $S$, which restricts to $S$ to the sub-Riemannian \emph{unit normal}, then
\begin{equation}\label{eq:div_k-1}
	\sigma_S^{(n)}=\operatorname{div}^{n}(N)\,\sigma_H,
\end{equation}
where $\operatorname{div}^{n}(N)$ is the $n$-th divergence of $N$ (cf. \cref{ssec:curve_form}). Following  \cite{BarBoss2024}, the first three curvature forms can be expressed explicitly as 
\begin{equation}\label{eq:trecoeff}
	\sigma^{(0)}_S=\sigma_H,\qquad \sigma^{(1)}_S=-\mathcal H \,\sigma_H,\qquad \sigma^{(2)}_S=(-N(\mathcal H)+\mathcal H^2) \sigma_H,
\end{equation}
where $\mathcal H$ is the sub-Riemannian mean curvature.

In the presence of characteristic points, however, in computing the half-tube volume one encounters two new main difficulties.
\begin{enumerate}
	\item The singular behavior of $\exp_S^\perp(t,p)$ as $p$ approaches $\Sigma(S)$ is reflected in a singular behavior of the integrand $J_S(t,p)\,dt\wedge dp$. Computing the volume asymptotics thus requires control of these singularities and of their contribution to the integral.
	\item The cut time $t_\cut(p)$ tends to zero as $p$ approaches $\Sigma(S)$ (cf. \cref{cor:t_cut_continuous}). Hence, for any fixed $\ve>0$, the minimum cannot be replaced by $\ve$ uniformly near the characteristic set. The cut time therefore determines a	nontrivial integration domain whose dependence on $\ve$ must be controlled. 
\end{enumerate}
We face these difficulties for surfaces having \emph{non-degenerate} characteristic points (cf. \cref{def:non_deg_op}). We start by analyzing the simplest example, the horizontal plane in the Heisenberg group.

\subsubsection{The horizontal plane} \label{s:thp}
Let us consider the surface $S_0$ in $\hei$ given by
\begin{equation}
    S_0:=\{(x,y,z)\in \hei : z=0\},
\end{equation}
whose only characteristic point is the origin and it is non-degenerate (see \cref{lem:ab+1neq0}). For such model surface, the Jacobian determinant of the normal exponential map and the cut time can be computed explicitly. In polar coordinates defined by $(x,y)=(r\cos\theta,r\sin\theta)$ we have
\begin{equation}\label{eq:intro-Jac-plane}
    J_{S_0}(t,r,\theta)=\frac{r^2}4\left[1+\cos\left(\frac{2t}{r}\right)+\left(\frac{t}{r}\right)\sin\left(\frac{2t}{r}\right)\right],\qquad t_\cut(r,\theta)=\frac{\pi}{2}r.
\end{equation}
These formulas are derived in \cref{l:det_exp_perp} and \cref{thm:intro-homogeneous-cut-time} and agree with \cite[Thm.\ 1.8]{RizziRossi}, where the authors compute explicitly the distance from $S_0$. In particular, denoting $B_R=\{(x,y,0)\in S_0: x^2+y^2\leq R^2\}$ and applying \eqref{eq:intro-tube-integral} with $U=B_R$, we obtain the following asymptotics as $\ve\to 0^+$ for the volume of the localized half-tube 
\begin{equation}\label{eq:intro-piano}
    \mathrm{vol}(\mathcal T_{B_R}(\ve))=\left(\frac{\pi R^3}{3}\right)\ve
	+C\ve^4+o(\ve^4),
	\qquad
	C=\frac{1-\pi\operatorname{Si}(\pi)}{12}<0.
\end{equation}
To compare this expansion with the non-characteristic formula
\eqref{eq:intro-non-char-tube}, we have computed the first four curvature
forms
\begin{equation}
\label{eq:trivially-zero}
    \sigma_{S_0}^{(0)}=\frac{r^2}{2}dr\wedge d\theta,\qquad \sigma_{S_0}^{(1)}= \sigma_{S_0}^{(2)}= \sigma_{S_0}^{(3)} = 0.
\end{equation}
For $n=1,2,3$, the coefficient of $\ve^n$ has the same expression as in \eqref{eq:intro-non-char-tube}
\begin{equation}
    a_n(B_R)=\frac{1}{n!}\int_{B_R}\sigma_S^{(n-1)},\qquad n=1,2,3.
\end{equation}
However, at order four, this rule fails: the coefficient of $\ve^4$ is nonzero although $\sigma_{S_0}^{(3)}$ vanishes identically. Moreover, the constant $C<0$ is independent of the radius $R$, i.e., of the size of the localizing domain, suggesting a curvature contribution concentrated at the characteristic point. We stress that $Q=4$ is the homogeneous dimension of $\hei$. 

\subsection{Homogeneous surfaces}The example of the horizontal plane shows that, beyond order $\ve^3$,
the half-tube volume contains contributions that are not captured by
the usual curvature integrals. To better understand this phenomenon,
we study the localized half-tube volume for a more general class of models, the \emph{homogeneous surfaces}.
Given $a,b\in\mathbb R$, we define
\begin{equation}
    S_{a,b}
    :=\left\{(x,y,z)\in\hei:
    z=\frac14(ax^2+by^2)\right\}.
\end{equation}
These surfaces have a characteristic point at the origin, which is non-degenerate when $1+ab\neq0$, see \cref{lem:ab+1neq0}. The horizontal plane is
recovered by taking $a=b=0$. These surfaces are \emph{homogeneous}, namely they are invariant under the Heisenberg dilations
\begin{equation}
    \delta_\ve:\hei\to\hei,\qquad
    \delta_\ve(x,y,z):=(\ve x,\ve y,\ve^2z),
    \qquad \ve>0.
\end{equation}
This property yields a scaling identity for the normal
exponential map, useful to address the two difficulties (1) and (2) described above.
\subsubsection{The tube integrand of homogeneous surfaces} On a surface $S_{a,b}$ with a non degenerate characteristic point, i.e., with $1+ab\neq 0$, we introduce adapted polar coordinates $(\rho,\vartheta)$ defined by 
\begin{equation}\label{eq:intro-homogeneous-polar}
 	(ax+y,by-x)=(\rho\cos\vartheta,\rho\sin\vartheta).
 \end{equation}
For $\ve>0$, since $S_{a,b}$ is invariant under dilations, $\delta_\ve$ maps its outward normal geodesics to
 outward normal geodesics, dilating their base point and their length by $\ve$. In terms of the normal exponential map \eqref{eq:intro-normal-exp}, in the coordinates $(t,\rho,\vartheta)$, this fact translates to 
 \begin{equation}\label{eq:intro-homogeneous-exponential}
 	\delta_\ve(\exp_{S_{a,b}}^\perp(t,\rho,\vartheta))
 	=\exp_{S_{a,b}}^\perp(\ve t,\ve\rho ,\vartheta).
 \end{equation}
  Taking Jacobian determinants of both sides of \eqref{eq:intro-homogeneous-exponential} with
$\ve>0$ fixed, then setting $\ve=1/\rho$ and multiplying both sides by $\rho^4$ we find
\begin{equation}\label{eq:intro-homogeneous-jacobian}
    J_{S_{a,b}}(t,\rho,\vartheta)
    =\rho^2J_{S_{a,b}}(t/\rho,1,\vartheta).
\end{equation}
where $J_{S_{a,b}}(t,\rho,\vartheta)$ denotes the Jacobian determinant of $\exp^\perp_{S_{a,b}}(t,\rho,\vartheta)$ computed with respect to the coordinates $(t,\rho,\vartheta)$.
The function $
J_{S_{a,b}}(\tau,1,\vartheta)$ depends smoothly on $(\tau,\vartheta)$, and its explicit expression is computed in \cref{l:det_exp_perp}. Thus, the singular behavior of the Jacobian as $\rho$ approaches $0$ is encoded in the ratio $t/\rho$.
Differentiating \eqref{eq:intro-homogeneous-jacobian} with respect to $t$ we may compute the $n$-th curvature form \eqref{eq:intro-curv-fomrs} 
 \begin{equation}\label{eq:intro-homogeneous-curvature-derivatives}
 	\sigma_{S_{a,b}}^{(n)}
 	=\rho^{2-n}\phi_n^{(a,b)}(\vartheta)d\rho\wedge d\vartheta,\qquad \phi_n^{(a,b)}(\vartheta):=\partial_t^n J_{S_{a,b}}(0,1,\vartheta).
 \end{equation}
The angular coefficients $\phi_n^{(a,b)}(\vartheta)$ are computed explicitly in \cref{prop:hom_curv_form}. Equation \eqref{eq:intro-homogeneous-curvature-derivatives} shows that, in general, due to the factor $\rho^{2-n}$, for $n\geq 3$ the form $\sigma_{S_{a,b}}^{(n)}$ is not locally integrable near the characteristic point, explaining why, in the presence of characteristic points, the coefficients of the half-tube expansion
cannot be obtained through the ordinary curvature integrals
appearing in \eqref{eq:intro-non-char-tube}.
\subsubsection{The cut time of homogeneous surfaces}
The cut time is classically characterized in terms of critical points and loss of injectivity of the normal exponential map. We prove such a characterization for surfaces in the Heisenberg setting in \cref{sec:cut_time}, and we compute explicitly the cut time for homogeneous surfaces. The following result is established in \cref{part:cut-time} of the paper.
\begin{theorem}
	\label{thm:intro-homogeneous-cut-time}
	Let $S_{a,b}$ be a homogeneous surface with a non-degenerate characteristic point, then 
	\begin{equation}\label{eq:intro-homogeneous-cut-time}
		t_\cut^{a,b}(x,y)=\left(\frac{\pi}{2}-\arctan\bigl(\max\{a,b\}\bigr)\right)\sqrt{(ax+y)^2+(by-x)^2}.
	\end{equation}
\end{theorem}
 \noindent We note that, setting $a=b=0$ we recover the cut time of the horizontal plane displayed in \eqref{eq:intro-Jac-plane}. Denoting $c_{a,b}=\frac{\pi}{2}-\arctan\bigl(\max\{a,b\}\bigr)$, in the polar coordinates of \eqref{eq:intro-homogeneous-polar}, the above result rewrites as 
 \begin{equation}\label{eq:intro-t_cut_a,b}
     t_\cut^{a,b}(\rho,\vartheta)=c_{a,b}\,\rho.
 \end{equation}
We see that the cut time is linear in $\rho$, and it is thus smooth in polar coordinates $(\rho,\vartheta)$.
\subsubsection{The half-tube volume of homogeneous surfaces}
Combining the analysis of the tube integrand with the regularity of
the cut time, we obtain the following  Steiner formula for homogeneous surfaces, proved in \cref{part:volume}. The expansion reveals a phenomenon absent in the case of the plane: when $a+b\neq0$, the leading correction beyond order $\ve^3$ is proportional to $\ve^4\log\ve$.
\begin{theorem}[Steiner formula for homogeneous surfaces]
	\label{thm:intro-steiner-hom}
	Let $S_{a,b}$ be a homogeneous surface with a non-degenerate characteristic point, i.e., $1+ab\neq 0$, and let $U\subset S_{a,b}$ be compact and having the origin in its interior. Then there exist
	real coefficients $\alpha_n(U)$, $n\geq4$, such that,
	for every integer $N\geq4$
 \begin{equation}\label{eq:intro-main-expansion_homogeneous}
		\begin{aligned}
			\vol(\mathcal T_{U}(\ve))
			={}&\sum_{n=1}^3\frac{\ve^n}{n!}
			\int_U\sigma_{S_{a,b}}^{(n-1)}+\frac{\pi(a+b)}{12|1+ab|}\ve^4\log\ve+\sum_{n=4}^N\ve^n
			\alpha_n(U)+
			o(\ve^N).
		\end{aligned}
	\end{equation}
\end{theorem}
\noindent We stress that the coefficient of such logarithmic term depends only on $a,b$ and is independent of the localization set $U$, suggesting a curvature contribution concentrated at the characteristic point.  For the horizontal plane, corresponding to $a=b=0$, this $\ve^4\log\ve$ correction vanishes. 

\subsubsection{Derivation of the $\ve^4\log\ve$ correction.} We now briefly explain how this logarithmic term arises, postponing the analysis of the higher-order coefficients (including the explanation of the $\ve^4$ term in the example of the horizontal plane) to the following subsections \ref{s:general} and \ref{s:curvature}.  

For $R>0$, in polar coordinates \eqref{eq:intro-homogeneous-polar} we set $B_R=\{\rho \leq R\}\subset S_{a,b}$. We  compute the localized half-tube volume around $B_R$. According to \eqref{eq:intro-homogeneous-jacobian} and \eqref{eq:intro-t_cut_a,b}, the Jacobian and the cut time satisfy 
\begin{equation}
\label{eq:nonlabel}
    J_{S_{a,b}}(t,\rho,\vartheta)=\rho^2 J_{S_{a,b}}(t/\rho,1,\vartheta),\qquad t_\cut^{a,b}(\rho,\vartheta)=c_{a,b}\rho.
\end{equation}
To simplify the exposition, in the following we assume that the coordinates $(\rho,\vartheta)$ are positively oriented (see \cref{rmk:hom_or} for details on orientation), and that $c_{a,b}=1$, so that $t_\cut^{a,b}(\rho,\vartheta)=\rho$. An application of equation \eqref{eq:intro-tube-integral} using \eqref{eq:nonlabel}  gives 
\begin{equation}\label{eq:intro-tube-vol-hom}
    \vol(\mathcal T_{B_R}(\ve))=\int_0^{2\pi}\int_0^R\int_0^{\min\{\ve,\rho\}}\rho^2J_{S_{a,b}}(t/\rho,1,\vartheta)\,dt\,d\rho\,d\vartheta.
\end{equation}
 Thus, making the change of coordinates $t=u\rho$ and defining the smooth function
\begin{equation}\label{eq:intro-Phi}
	\Phi(v):=\int_0^{2\pi}\int_0^v J_{S_{a,b}}(u,1,\vartheta)\,du\,d\vartheta,
\end{equation}
equation \eqref{eq:intro-tube-vol-hom} becomes
\begin{equation}\label{eq:intro-volume-splitting}
	\vol(\mathcal T_{B_R}(\ve))=\int_0^{\ve}\rho^3\Phi(1)\,d\rho+\int_{\ve}^{R}\rho^3\Phi\left(\frac{\ve}{\rho}\right)\,d\rho.
\end{equation}
The first integral is proportional to $\ve^4$. The logarithmic term arises from the second integral. To see this, note that only monomials of the kind $\rho^{3}v^k$, $k\geq 0$, appear in the Taylor expansion of $\rho^3\Phi(v)$. After setting $v=\ve/\rho$, for every $k\geq 0$ the contribution of such monomial to the integral is 
\begin{equation}
	\ve^k\int_\ve^{R}\rho^{3-k}\,d\rho.
\end{equation}
A logarithm appears precisely when $k=4$, in which case this expression becomes
\begin{equation}
	\ve^{4}\bigl(\log R-\log\ve\bigr).
\end{equation}
For rigorous Taylor expansions, with suitable remainder estimates, see \cref{l:integrato}.

\subsection{The general case} \label{s:general}
We now consider a general smooth oriented embedded surface $S=\partial\Omega\subset\hei$. Let $p\in\Sigma(S)$ be a non-degenerate characteristic point. After applying an isometry of $\hei$, we may assume that $p$ is the origin and that $S$ is locally given by
\begin{equation}\label{eq:intro-homogeneous-graph}
    S=\left\{(x,y,z)\in\hei:z=\frac12u(x,y)\right\},
    \qquad
    u(x,y)=\frac12(ax^2+by^2)+O\bigl(|(x,y)|^3\bigr),
\end{equation}
where $u$ is smooth near the origin and $a,b\in\mathbb R$. Note that the quadratic part of \eqref{eq:intro-homogeneous-graph} defines the homogeneous surface $S_{a,b}$. Thus, homogeneous surfaces are not merely examples for which the tube geometry can be computed explicitly: they also provide local approximations to general surfaces in the vicinity of their non-degenerate characteristic points. 

 \subsubsection{The tube integrand of a general surface}
 Since $S$ can be seen as a smooth perturbation of $S_{a,b}$, it is possible to express the Jacobian $J_{S}(t,\rho,\vartheta)$ of the normal exponential map of $S$ as a smooth perturbation of $J_{S_{a,b}}(t,\rho,\vartheta)$. In particular we prove that 
 \begin{equation}\label{eq:intro-homogeneous-jacobian-approximation}
 	\begin{aligned}
 		J_S(t,\rho,\vartheta)=\rho^2\bigl(J_{S_{a,b}}(t/\rho,1,\vartheta)+\rho R_J(t/\rho,\rho,\vartheta)\bigr),
 	\end{aligned}
 \end{equation}
 where $J_{S_{a,b}}(\tau,\rho,\vartheta)$ and $R_J(\tau,\rho,\vartheta)$ are smooth (see \cref{prop:iter_diver_singularities} and \cref{rmk:in_hom_coords}). As in the homogeneous case, the singularities of the Jacobian are encoded in the ratio $t/\rho$. Such singularities are reflected in the curvature forms.
 \begin{prop}\label{prop:intro-curvature-approximation}
 	For every integer $n\geq0$, the $n$-th curvature form of $S$ satisfies
 	\begin{equation}\label{eq:intro-curvature-approximation}
 		\sigma_S^{(n)}=\rho^{2-n}\left(\phi_n^{a,b}(\vartheta)+\rho\psi_n(\rho,\vartheta)\right)d\rho\wedge d\vartheta,
 	\end{equation}
 	where $\phi_n^{a,b}$ is the angular coefficient of the homogeneous model and $\psi_n$ is smooth, including at $\rho=0$.
 \end{prop}
\noindent As a consequence, the curvature forms $\sigma_S^{(n)}$ are, in general, locally integrable near the characteristic point if and only if $n=0,1,2$. 

\subsubsection{Cut time and stability}\label{ss:intro-cut-time}
The natural idea to compute the cut time for a general surface $S$ is  to exploit homogeneous approximations: since $S$ is a perturbation of $S_{a,b}$ near the origin, $\exp^\perp_{S}$ is a perturbation of $\exp^\perp_{S_{a,b}}$. However, the latter fact does not by itself ensure that $t_\cut$ can be smoothly approximated by $t_\cut^{a,b}$. To obtain such a result, we must also establish stability of the critical points and self intersections of the map $\exp^\perp_{S_{a,b}}$, within a suitable class of mappings.

A fine analysis of the normal exponential maps of homogeneous models, carried out in \cref{part:cut-time} of the paper, shows that the cut-locus configuration of $S_{a,b}$ is stable under perturbations only if $a\neq b$. This motivates the following definition. 

\begin{defi}\label{def:intro-geometric-stability}
	A characteristic point $p\in\Sigma(S)$ is called \emph{geometrically stable} if it is non-degenerate and the parameters of its homogeneous approximation satisfy $a\neq b$.
\end{defi}
\noindent  We remark that geometric stability is a generic property of characteristic points: an arbitrarily small smooth perturbation of a characteristic surface makes its characteristic points geometrically stable. 

The following theorem, which computes the cut time near a geometrically stable characteristic point, is the second main result of \cref{part:cut-time} of the paper.
\begin{theorem}\label{thm:intro-cut-time-stability}
	Let $S=\partial\Omega\subset\hei$ be a smooth oriented embedded surface. Let $p \in \Sigma(S)$ be a geometrically stable characteristic point, and let $S_{a,b}$ be its homogeneous approximation. Then there exists a neighborhood of $p$ where, in polar coordinates	\eqref{eq:intro-homogeneous-polar}, the cut time is smooth and satisfies
    \begin{equation}\label{eq:intro-cut-time-approximation}
		t_\cut(\rho,\vartheta)=c_{a,b}\rho+\rho^2R_\cut(\rho,\vartheta),
	\end{equation}
	where $c_{a,b}=\pi/2-\arctan(\max\{a,b\})$ and $R_\cut\in C^\infty(\R^+\times \mathbb S^1)$ is a smooth remainder.
\end{theorem}
\noindent The proof of this theorem relies on a regularization of the exponential map inspired by \cite{agrexp} (see also \cite[Chapter 19]{Agrachev}), and on the theory of Legendrian singularities \cite[$\S$ 6]{arnold}.
\subsubsection{The half-tube volume in the general case}
Combining the analysis of the tube integrand with the regularity of the cut time, we obtain the following  Steiner formula, which is established in \cref{part:volume}.
\begin{theorem}[Steiner formula for characteristic surfaces]
	\label{thm:intro-steiner}
	Let $S=\partial\Omega\subset\hei$ be a smooth embedded	surface, whose characteristic points are geometrically stable. Let $U\subset S$ be a compact set having all its characteristic points in its interior. Then there exist real coefficients $\alpha_n(U),\beta_n(U)$, $n\geq4$, such that, for every integer $N\geq4$, as $\ve\to0^+$,
	\begin{equation}\label{eq:intro-main-expansion}
			\vol(\mathcal T_U(\ve))=\sum_{n=1}^3\frac{\ve^n}{n!}\int_U\sigma_S^{(n-1)}+\sum_{n=4}^N\ve^n\bigl(\alpha_n(U)+\beta_n(U)\log\ve\bigr)
			+o(\ve^N).
	\end{equation}
\end{theorem}
\noindent The first three coefficients retain the same expression as in the non-characteristic case, in accordance with the fact that the corresponding curvature forms remain locally integrable at characteristic points. We see that, differently from the homogeneous case, also corrections of the form $\ve^n\log\ve$ for $n>4$ may appear (we stress once again that $Q=4$ is the homogeneous dimension of $\hei$). This is due to the correction terms $\rho R_J$ and $\rho^2R_\cut$ in the identities \eqref{eq:intro-homogeneous-jacobian-approximation} and \eqref{eq:intro-cut-time-approximation}. 

We additionally remark that, in contrast with the non-characteristic case, \cref{thm:intro-steiner} shows that, in general, the volume around a characteristic surface is not analytic in $\ve$.

\subsection{Curvature distributions} \label{s:curvature}
A natural question is whether the coefficients $\alpha_n(S),\beta_n(S)$ in \eqref{eq:intro-main-expansion} are also the total masses of suitably defined curvature measures. Following Federer's viewpoint \cite{FedererCM}, we investigate this question by localizing the tube volume through the metric projection onto $S$.

Outside the negligible cut locus, every $q\in\hei\setminus\cut(S)$ has a unique nearest point on $S$, denoted $\mathsf p(q)$. Equivalently, $\mathsf p(q)$ is the initial point of the unique minimizing normal geodesic from $S$ to $q$. The metric projection $\mathsf p_\ve:=\mathsf p|_{\mathcal T_S(\ve)}: \mathcal T_S(\ve)\to S$ is smooth outside the cut locus and admits a measurable extension across it. Any such extension defines the same push-forward measure
\begin{equation}
	\mu_\ve:=(\mathsf p_\ve)_\#
	\bigl(\vol|_{\mathcal T_S(\ve)}\bigr).
\end{equation}
For $\varphi\in C^\infty(S)$, its pairing is the weighted tube volume
\begin{equation}\label{eq:intro-pairing}
	\langle\mu_\ve,\varphi\rangle
	:=\int_S \varphi\, \de\mu_\ve=\int_{\mathcal T_S(\ve)}
	(\varphi\circ\mathsf p_\ve)\,\de\vol.
\end{equation}
The interpretation \emph{\`a la Federer} of the asymptotics of the family of measures $\mu_\ve$ shows that the coefficients are curvature distributions, rather than curvature measures. To describe the resulting coefficients, we first specify the regularization of the divergent curvature integrals. For $n\geq4$,
the integral
\begin{equation}
	\int_{\{t_\cut>\delta\}}\varphi\,\sigma_S^{(n-1)}
\end{equation}
has an expansion, as $\delta\to0^+$, consisting of negative powers of $\delta$, a possible logarithm, a constant term, and $o(1)$, see \cref{lem:finite_part_boring}. The constant term of this expansion is the \emph{finite part}, denoted by
\begin{equation}
	\bigl\langle\Fp(\sigma_S^{(n-1)}),\varphi\bigr\rangle.
\end{equation}
The next theorem is the final main result of \cref{part:volume} of the paper. 

\begin{theorem}[Curvature distributions]
	\label{thm:intro-distributions}
	Let $S=\partial\Omega\subset\hei$ be a smooth embedded surface, whose characteristic points are geometrically stable. Then, there exist distributions $\alpha_n,\beta_n\in\mathcal D'(S)$, such that, for every $\varphi\in C_c^\infty(S)$ and every integer $N\geq4$, as $\ve\to0^+$,
	\begin{equation}\label{eq:intro-weighted-expansion1}
			\langle\mu_\ve,\varphi\rangle=\sum_{n=1}^3\frac{\ve^n}{n!}\int_S\varphi\,\sigma_S^{(n-1)}+\sum_{n=4}^N\ve^n			\bigl(\alpha_n(\varphi)+\beta_n(\varphi)\log\ve\bigr)
			+o(\ve^N).
	\end{equation}
	Moreover, for every $n\geq4$ it holds 
    \begin{enumerate}
        \item the distribution $\alpha_n$ has order at most $n-3$. In addition, there is a distribution $T_n\in\D'(S)$ of order at most $n-4$ such that
	\begin{equation}\label{eq:intro-finite-part}
		\alpha_n(\varphi)=\frac1{n!}\langle\Fp(\sigma_S^{(n-1)}),\varphi\rangle+T_n(\varphi),\qquad\operatorname{supp}T_n
		\subset\Sigma(S);
	\end{equation}
        \item the distribution $\beta_n$ is of order at most $n-4$ and $\supp\beta_n\subset\Sigma(S)$. 
    \end{enumerate} 
    Finally, if $S$ is homogeneous, the results hold requiring only non-degenerate characteristic points.
\end{theorem}
\noindent Note that if $\Sigma(S)=\emptyset$ then the finite part coincides with the integral $\int_S\varphi\sigma_S^{(n-1)}$ and the support of $T_n$ and $\beta_n$ is empty, thus we recover the non characteristic expansion \eqref{eq:intro-non-char-tube}.

For $n=4$ the distributions $\beta_4$ and $T_4$ in \eqref{eq:intro-weighted-expansion1} have order 0. In \cref{sec:beta4}, we compute them explicitly showing that they depend only on the homogeneous approximation of $S$ at the characteristic points. 

\begin{prop}\label{lem:intro-beta4-T4}
		Let $S=\partial\Omega\subset\hei$ be a smooth embedded	surface, whose characteristic points are geometrically stable. Let $p\in\Sigma(S)$ and let $S_{a,b}$ be its oriented homogeneous approximation. Let $\varphi\in C_c^\infty(S)$ such that $\supp\varphi\cap\Sigma(S)=\{p\}$, then 
	\begin{equation}
		\beta_4(\varphi)=\beta_{a,b}\,\varphi(p),\qquad T_4(\varphi)=\tau_{a,b}\,\varphi(p),
	\end{equation}
    where $\beta_{a,b}$ and $\tau_{a,b}$ are constants depending only on the parameters $a,b$. Finally, if $S$ is homogeneous, the results hold requiring only non-degenerate characteristic points.
\end{prop}
\noindent Both constants can be explicitly computed. The first one coincides with the coefficient of $\ve^4\log\ve$ found for the homogeneous expansion, i.e., 
\begin{equation}
    \beta_{a,b}=\frac{\pi(a+b)}{12|1+ab|}.
\end{equation}
The second one has a more involved expression, but can nonetheless be computed
when $a=b=0$  (cf.\ \cref{lem:beta4_dirac}). One has 
\begin{equation}
\tau_{0,0}=\frac{1-\pi\operatorname{Si}(\pi)}{12},
\end{equation}
while the finite part of $\sigma^{(3)}_{S_0}$ is identically zero, cf.\ \eqref{eq:trivially-zero}, recovering the coefficient of the $\ve^4$ correction for the horizontal plane described in Section~\ref{s:thp}.

\subsection{Organization of the paper}
Section~\ref{s:prel} recalls preliminaries on the Heisenberg group and on characteristic points. In \cref{part:volume}, \cref{sec:distance_from_S} establishes formula \eqref{eq:intro-tube-integral} for the tube volume. In \cref{sec:normal_exp}, we study the normal exponential map and the curvature forms, proving \cref{prop:intro-curvature-approximation}. This analysis leads to the volume expansion in \cref{sec:volume_half-tube}, which contains the proof of \cref{thm:intro-steiner-hom,thm:intro-steiner}, and its distributional refinement in \cref{s:distri}, where we prove \cref{thm:intro-distributions}. The results of \cref{part:volume} are obtained assuming the validity of the results of \cref{part:cut-time}, which is devoted to the cut time.

In \cref{sec:eem}, we introduce and study the extended exponential map. In \cref{sec:cut-time-hom}, we focus on the homogeneous case proving \cref{thm:intro-homogeneous-cut-time}. Finally, in \cref{sec:cut_time_general,sec:proof_of_cut_time_general}, we show how the theory of Legendrian singularities can be used to prove the perturbative result on the cut time of \cref{thm:intro-cut-time-stability}. \cref{part:appendix} contains appendices which include explicit computations for homogeneous surfaces in Appendix~\ref{a:f}, the technical results on logarithmic expansions and their coefficients in Appendix~\ref{app:noname}, the required facts about distributions in Appendix~\ref{app:distributions}, and Legendrian fronts in Appendix~\ref{app:Legendrian}.

\subsection*{AI disclosure} AI tools were used for English-language proofreading and double-checking of some algebraic computations. They were not used to generate or develop mathematical content. 

\subsection*{Acknowledgments} This project has received partial funding from the European Research Council (ERC) under the European Union’s Horizon 2020 research and innovation programme (grant agreement GEOSUB, No. 945655). The authors are grateful to SISSA and the University of Padua, where part of this project was developed.  

\section{Preliminaries: the Heisenberg group, surfaces and characteristic points}\label{s:prel}
Recall that the \emph{Heisenberg group}, denoted with $\mathbb H^1$, is a Lie group, whose underlying smooth manifold is $\mathbb R^3$ and has group law
\begin{equation}\label{eq:group_law}
    (x,y,z)\star(x',y',z')=\left(x+x',y+y',z+z'+\frac{1}{2}(-y,x)\cdot\begin{pmatrix}
        x' \\ y'
    \end{pmatrix}\right),\qquad (x,y,z),(x',y',z')\in\mathbb R^3.
\end{equation}
Defining the following left-invariant frame of vector fields
\begin{equation}\label{eq:fields_Xi}
	X_1=\partial_x-\frac{y}{2}\partial_z,\qquad X_2=\partial_y+\frac{x}{2}\partial_z,\qquad X_0=\partial_z,
\end{equation}
we endow $\mathbb H^1$ with the sub-Riemannian structure $(\dis, g)$, where $\dis$ is generated by the frame $\{X_1,X_2\}$ and $g$ is obtained by declaring $X_1,X_2$ orthonormal. Note that the dual frame associated with $X_1,X_2,X_0$ is given by the 1-forms
\begin{equation}\label{eq:the_coframe}
	\nu_1=dx,\qquad \nu_2=dy,\qquad \nu_0=dz+\frac{1}{2}(ydx-xdy).
\end{equation}
The differential form $\nu_0$ is the so called \emph{(normalized) contact form}. In this setting, we say that a curve $\gamma:[0,T]\to \mathbb H^1$ is \emph{horizontal} if it is absolutely continuous and for a.e.\ $t\in [0,T]$, 
\[\dot{\gamma}(t)\in \mathcal D_{\gamma(t)}=\mathrm{span}\{X_1,X_2\}|_{\gamma(t)}.
\]
For such curves we define the length as $\ell(\gamma)=\int_{0 }^{T} \| \dot{\gamma} (t)\|_gdt$. Note that the length of a horizontal curve is invariant by reparametrization. The \emph{sub-Riemannian distance} $d_{SR}$ between two points $p,q\in \mathbb H^1$ is defined as
\begin{equation}
\label{eq:distance}
d_{SR}(p,q)=\inf \left\{ \ell(\gamma) : \gamma\in \mathcal{C}_{p,q}\right\}
\end{equation}
where $\mathcal{C}_{p,q}$ denotes the set of horizontal curves connecting $p,q$. By Chow-Rashevsky theorem, $d_{SR}$ is finite, continuous and induces the manifold topology. We conclude this preliminary section recal\-ling that $d_{SR}$ is $1$-homogeneous with respect to the (anisotropic) dilations given by 
\begin{equation}
\label{eq:dilations}
    \delta_\ve :\hei\to \hei,\qquad \delta_\ve(x,y,z) = (\ve x,\ve y,\ve^2z),
\end{equation}
for every $\ve>0$. More precisely, $d_{SR}(\delta_\ve(p),\delta_\ve(q))=\ve d_{SR}(p,q)$, for every $p,q\in\hei$ and every $\ve>0$. For further details we refer to \cite{Agrachev}.

\subsection{Geodesics and exponential map}
In this subsection we recall some basic facts concerning geodesics of the Heisenberg group. We say that a curve $\gamma:[0,T]\to \mathbb H^1$ is \emph{length-minimizing} (or a \emph{minimizing geodesic}) if  it is parametrized by constant speed and
\[
d_{SR}(\gamma(0),\gamma(T))= \ell(\gamma).
\]
A \emph{geodesic} is a curve $\gamma:[0,T]\to \mathbb H^1$  that is locally length-minimizing. We define the linear Hamiltonian functions associated to the frame $\{X_0,X_1,X_2\}$ as follows
\begin{equation}
    h_i:T^*\mathbb H^1\to \mathbb R,\qquad h_i(\lambda)=\langle\lambda, X_i\rangle,\qquad i=0,1,2.
\end{equation}
Then, the \emph{\sr Hamiltonian} of $\hei$ is the smooth function $H\in C^\infty(T^*\hei)$, defined as
\begin{equation}\label{eq:sub-Riemann_Hamilton}
    H=\frac12\left(h_1^2+h_2^2\right).
\end{equation}
Denoting with $\sigma\in \Omega^2(T^*\mathbb H^1)$ the standard symplectic form, we define the Hamiltonian vector field $\vec{H}\in\mathrm{vec}(T^*\mathbb H^1)$ as the symplectic gradient of $H$, i.e., 
\begin{equation}
    \iota_{\vec{H}}\sigma=dH.
\end{equation}
Note that the functions $(x,y,z,h_1,h_2,h_0)$ define global coordinates on $T^*\mathbb H^1$. In these coordinates, the Hamiltonian vector field can be explicitly expressed as
\begin{equation}
    \vec H=h_0(h_1\partial_{h_2}-h_{2}\partial_{h_1})+h_1X_1+h_2X_2.
\end{equation}
The Hamiltonian system associated to $H$ is then 
\begin{equation}
\label{eq:ham_sys}
    \dot\lambda(t) = \vec H(\lambda(t)), 
\end{equation}
for $\lambda:[0,T]\to T^*\hei$. Solutions to \eqref{eq:ham_sys} are called \emph{normal extremals} and are contained in a fixed level set of $H$. In particular, they are parametrized by arc-length if, for every $t\in [0,T]$,
\begin{equation}
    \label{eq:hamiltoniana1}
2H(\lambda(t))=1.
\end{equation}
Their projections are locally length-minimizing curves \cite[Thm.\ 4.65]{Agrachev}.
Conversely, since in the Heisenberg group there are no non-trivial abnormal geodesics, any length-minimizing curve $\gamma$ is the projection of a normal extremal $\gamma(t)=\pi\circ\lambda(t)$. The \emph{exponential map} of $\hei$ is then 
\begin{equation}
    \label{eq:expmaporiginal}
    \exp:T^*\mathbb H^1\to \mathbb H^1,\qquad \exp(\lambda)=\pi\circ e^{\vec H}(\lambda),
\end{equation}
where $e^{\vec{H}}:T^*\hei\to T^*\hei$ is the flow of $\vec H$ at time $t=1$ and $\pi:T^*\hei\to\hei$ is the standard projection. For a fixed $p\in\mathbb H^1$, we also define the exponential map at $p$, $\exp_p:T^*_p\mathbb H^1\to \mathbb H^1$, defined as the restriction of $\exp$ to the fiber $T_p^*\hei$.   
By the left-invariance, it is enough to study the exponential map at the identity, indeed it holds: 
\begin{equation}
\label{eq:exp_left-invariance}
	\exp_p(\lambda)=L_p\circ \exp_0\circ L_p^*(\lambda),\qquad \forall\, \lambda\in T_p^*\hei,
\end{equation}
where $L_p:\hei\to\hei$ is the left multiplication by $p$. By direct integration of \eqref{eq:ham_sys}, one can easily compute $\exp_0:T_0^* \mathbb R^3\to \hei$ (see e.g.\ \cite[Sec.\,7.1]{BarRiz2019}): given $\lambda = h_1\nu_1+h_2\nu_2+h_0\nu_0\in T_0^*\hei$ and $t\in\R$, we have 
\begin{equation}
\label{eq:exp_0_explicit}
    \exp_0(t\lambda)=\left(\frac{h_1 \sin(h_0 t) + h_2 (\cos(h_0 t) - 1)}{h_0},\frac{h_1 (1 - \cos(h_0 t)) + h_2 \sin(h_0 t)}{h_0},(h_1^2 + h_2^2) \frac{h_0 t - \sin(h_0 t)}{2 h_0^2}\right).
\end{equation}
\subsection{Surfaces and characteristic points}
We study surfaces embedded in $\hei$ and their local structure near \emph{characteristic points}. 
\subsubsection{Characteristic points}\label{ssec:char_points}
Let $S\subset \hei$ be a smooth embedded orientable surface. For a given $p\in S$, we say that $f\in C^\infty(\hei)$ is a locally defining function for $S$ near $p$, if there exists an open neighborhood $U\subset \hei$ of $p$ such that $f|_U$ is a submersion and $S\cap U=f^{-1}(0)\cap U$.

\begin{defi}
    A point $p\in S$ is called a \emph{characteristic point} if $\dis_p = T_pS$. Moreover we set 
\begin{equation}
    \Sigma(S):=\{p\in S\,:\,\mathcal \dis_p=T_pS\},\qquad S^\times :=S\setminus\Sigma(S).
\end{equation}
\end{defi}

It is not difficult to show that, since $\dis$ is a contact distribution, the characteristic set is a closed subset of $S$ contained in an embedded curve (see, e.g., \cite[Lemma 2.4]{Daniele}; for a more general discussion on the size of the characteristic set, we refer to \cite{Bal03} and references therein). In addition, there exists a smooth vector field $X$ tangent to $S$, called the \emph{characteristic vector field}, such that
        \begin{equation}\label{eq:char_vf}
            \mathrm{span}_{\mathbb R}\{X|_{q}\}=\begin{cases}
                \mathcal D_q\cap T_q S \qquad &q\in S^\times, \\
                \{0\} \qquad & q\in\Sigma(S).
            \end{cases}
        \end{equation}
In particular, given a smooth differential form $\omega\in \Omega^2(S)$ inducing the orientation of $S$, the characteristic vector field is defined by the equation 
\begin{equation}
    \iota_X\omega=\nu_0|_S,
\end{equation}
where $\nu_0$ is the contact form introduced in \eqref{eq:the_coframe}. Note that the characteristic vector field is not unique: it depends on the choice of $\omega$. However all characteristic vector fields differ by a never vanishing multiplicative factor.

In terms of a locally defining function $f$, a point $p\in S$ is characteristic if and only if the horizontal gradient $\nabla f$ of $f$ is zero at $p$, i.e., if $d_pf$ vanishes on vectors of $\mathcal D_p$. Recall that, given a smooth function $f:\hei\to\mathbb R$, its \emph{horizontal gradient} is defined as the unique horizontal vector field $\nabla f$ satisfying the following equality, for every horizontal vector field $Y$
\begin{equation}
    g(\nabla f, Y)=df(Y).
\end{equation}
It is then natural to introduce $ \mathrm{Hess}_{H}(f)$ the \emph{horizontal Hessian} of $f$, which in the orthonormal frame $\{X_1,X_2\}$ is expressed as 
  \begin{equation}
            \mathrm{Hess}_{H}(f)=\begin{pmatrix}
            X_1X_1(f) & X_2 X_1(f)\\
            X_1X_2(f) & X_2 X_2(f)
        \end{pmatrix}.
        \end{equation}
        
    \begin{defi}\label{def:non_deg_op}
        Let $f\in C^\infty(\hei)$ be a locally defining function near $p\in\Sigma(S)$. Then, we say that $p$ is a \emph{non-degenerate characteristic point} if and only if $\det \mathrm{Hess}_{H}(f)|_{p}\neq 0$.
    \end{defi}
    We stress that the notion of non-degenerate characteristic point does not depend on the choice of the locally defining function $f$ and on the choice of the orthonormal frame \cite{Daniele,BellBosc,BBP25}.
    Furthermore, non-degenerate characteristic points are necessarily isolated, see for instance \cite{BellBosc}. For a study of degenerate characteristic points see e.g.\ \cite{RossiIntegrability,BBP25}.
    
    We now study the local structure of a surface $S$ near a characteristic point $p\in \Sigma(S)$. Since left-translations are isometries of $\hei$, it is not restrictive to assume that $p=(0,0,0)$.
    \begin{lemma}\label{lem:nice_f}
    Let $S\subset \hei$ be a smooth, oriented, embedded surface with a characteristic point at $p=(0,0,0)$. Then, a locally defining function for $S$ near $p$ is given by 
        \begin{equation}\label{eq:local_S}
            f(x,y,z)=z-\frac{1}{2}u(x,y)
        \end{equation}
        where $u\in C^\infty(\mathbb R^2)$. Moreover, up to an isometry of $\hei$, there exist $a,b\in\mathbb R$ such that 
        \begin{equation}\label{eq:ab_u}
            u(x,y)=\frac{1}{2}(ax^2+by^2)+o(x^2+y^2),\qquad \text{as } (x,y)\to (0,0).
        \end{equation}
    \end{lemma}
    \begin{proof}
        Let $\proj:\R^{3}\to\mathbb R^2$ be the map $\proj(x,y,z)=(x,y)$. Observe that $\proj|_{S}:S\to \R^{2}$ is a local diffeomorphism near $p=(0,0,0)$. Indeed, since $p\in S$ is assumed to be characteristic, by \eqref{eq:fields_Xi}
        \begin{equation} \label{eq:horis}
            T_pS=\mathcal D_p=\mathrm{span}\{X_{1},X_{2}\}|_{p}=\mathrm{span}\{\partial_x,\partial_y\},
        \end{equation}
        therefore $d_p\mathfrak p|_S$ is the identity.
        Let $B\subset \mathbb R^2$ be a small enough neighborhood of the origin in $\mathbb R^2$ and let 
        \begin{equation}
            \varphi:B\to S,\qquad \varphi =(\proj|_{S})^{-1}.
        \end{equation}
        Then, being $\varphi\circ \proj|_{S}$ equal to the identity, we can write $\varphi(x,y)=(x,y,\frac{1}{2}u(x,y))$ for some smooth function $u:B\to \mathbb R$. Hence, the claim \eqref{eq:local_S} follows. We now show \eqref{eq:ab_u}. Since $p=(0,0,0)\in \Sigma(S)$, it must hold $f(0,0,0)=0$, which implies $u(0,0)=0$. Additionally, \eqref{eq:horis} forces  $\nabla u (0,0)=0$. Therefore $u$ is quadratic at the origin 
        \begin{equation}
            u(x,y)=\frac{1}{2}(ax^2+by^2+cxy)+o(x^2+y^2).
        \end{equation}
        Up to a rotation in the $(x,y)$-plane, which is an isometry of $\hei$, we may diagonalize the (standard) Hessian of $u$, which results in $c=0$.
    \end{proof}

    \begin{remark}
    \label{rmk:embedding}
        As a byproduct of the previous lemma we find an explicit embedding of $S$ around a characteristic point. Namely, if $p=(0,0,0)\in\Sigma(S)$, then there is a neighborhood $U\subset \hei$ of $p$ and a smooth embedding $\Phi:B\to \hei$, where $B\subset\R^2$ is a neighborhood of the origin in $\R^2$, such that $\Phi(B) = U\cap S$. In addition, $\Phi(x,y) = (x,y, \frac12 u(x,y))$, where $u\in C^\infty(\R^2)$ has the  expansion \eqref{eq:ab_u}.
    \end{remark}

    The non-degeneracy of a  characteristic point can be expressed through the numbers $a,b$ in \eqref{eq:ab_u}.    
    \begin{lemma}\label{lem:ab+1neq0}
        Let $S\subset \hei$ be a smooth, oriented,   embedded surface with a characteristic point at the origin $p=(0,0,0)$. Then  $p$ is non-degenerate if and only if $ab+1\neq 0$.  
    \end{lemma}
    \begin{proof}
        We apply \cref{def:non_deg_op} to the locally defining function given in \cref{lem:nice_f}:
        \begin{equation}
            \det \left.\begin{pmatrix}
            X_1X_1(f) & X_2 X_1(f)\\
            X_1X_2(f) & X_2 X_2(f)
        \end{pmatrix}\right|_p= \det \begin{pmatrix}
            -\frac{a}{2} & -\frac{1}{2}\\
            \frac{1}{2} & -\frac{b}{2}
        \end{pmatrix}=\frac{1}{4}(ab+1).
        \end{equation}
        It follows   that $p$ is non-degenerate if and only if $ab+1\neq 0$.
    \end{proof}
 
      Hence, a characteristic point $p$ is a stable singularity of the characteristic vector field $X$   if and only if $ab+1\neq 0$. We also introduce the notion of \emph{geometric stability} of a characteristic point, that we are going to need in what follows. 
    
    \begin{defi}
    \label{def:geometrically_stable_char_pts}
        Let $S\subset \hei $ be a smooth embedded surface with a characteristic point at the origin $p=(0,0,0)$ and locally expressed as in \eqref{eq:local_S}-\eqref{eq:ab_u}. Then $p$ is called \emph{geometrically stable} if the following inequalities hold 
        \begin{equation}
            ab+1\neq 0,\qquad a\neq b.
        \end{equation}
    \end{defi}
    Observe that, according to Lemma~\ref{lem:ab+1neq0}, a geometrically stable characteristic point is non-degenerate. As we show in \cref{part:cut-time}, the extra condition $a\neq b$  is related to the stability of the cut locus of the surface near the characteristic point.  
\subsubsection{Homogeneous surfaces}\label{ssec:hom_surf} We introduce a family of surfaces that constitute models for characteristic surfaces in the Heisenberg group. 

\begin{defi}
\label{def:homogeneous_surf}
    We define the oriented homogeneous surface $S_{a,b}$ of parameters $a,b\in \mathbb R$ as 
    \begin{equation}
        S_{a,b} : = \left\{(x,y,z)\in\hei:z=\frac1 4 (ax^2+by^2)\right\}\subset\hei,
    \end{equation}
    oriented towards the side $\{z>(ax^2+by^2)/4\}$.
\end{defi}
Note that the surface $S_{a,b}$ is smooth, orientable and embedded. If $ab+1\neq 0$, it has a single characteristic point at the origin $p=(0,0,0)$ which is non-degenerate, while if $ab+1=0$ it contains the following curve of degenerate characteristic point $\{(x,y,z)\in S_{a,b}\,:\,y=-ax\}$. Finally, it holds:   
    \begin{equation}
\label{eq:homogeneity_S_ab}
    \delta_r(S_{a,b}) = S_{a,b},\qquad\forall\,r>0,
\end{equation}
where $\delta_r$ is the dilation of \eqref{eq:dilations}.
These properties follow by the definition of $S_{a,b}$ and of \cref{lem:ab+1neq0}.
\part{The volume of the half-tube around characteristic surfaces}
\label{part:volume}

In this part, we delve into the study of characteristic surfaces, and tubes around them. From now on, we will always consider surfaces that are smooth, embedded and bound an open set $\Omega\subset \hei$. In this way, $S=\partial\Omega$ is without boundary. 

\section{The half-tube and the sub-Riemannian distance from a surface}\label{sec:distance_from_S}

We fix a surface $S\subset \hei$ and we study the distance from $S$. Since $S$ bounds a regular domain $\Omega\subset \hei$, so that $S=\partial \Omega$, we introduce the \emph{signed distance function}: 
\begin{equation}
    \mathrm{dist}_S:\hei\to \mathbb R,\qquad \mathrm{dist}_S(p)
    =\begin{cases}
        \inf_{q\in S}d_{SR}(p,q),\qquad &p\in \hei\setminus \Omega,\\
        -\inf_{q\in S}d_{SR}(p,q),\qquad &p\in \Omega,
    \end{cases}
\end{equation}
where $d_{SR}$ denotes the sub-Riemannian distance. 
 \begin{remark}\label{rmk:1_lip} The distance from $S$ is $1$-Lipschitz with respect to the sub-Riemannian distance. Indeed if $p,q$ lie on the same side of $\hei\setminus S$, the triangle inequality gives
$|\mathrm{dist}_S(p)-\mathrm{dist}_S(q)|\le d_{SR}(p,q)$.
Otherwise, for $p$ and $q$ on opposite sides of $S$, every horizontal curve $\gamma$ joining $p$ to $q$ meets $S$, so
\begin{equation}
    |\mathrm{dist}_S(p)-\mathrm{dist}_S(q)|
    =|\mathrm{dist}_S(p)|+|\mathrm{dist}_S(q)|
    \le \operatorname{length}(\gamma),
\end{equation}
and taking the infimum over $\gamma$ proves the claim.
\end{remark}
We can now define the (positive) half-tube.
\begin{defi}
    Let $S=\partial\Omega\subset\hei$ be an embedded smooth surface, and fix $\ve>0$. Then, the \emph{(positive) half-tube} of size $\ve>0$ around $S$ is the set  
    \begin{equation}
        \mathcal T_S(\ve) := \{p\in\hei\,:\,0<\mathrm{dist}_S(p)\leq\ve\}.
    \end{equation}
\end{defi}
To better describe the set $\mathcal T_S(\ve)$ let us introduce some terminology. We say that a horizontal curve $\gamma : [0,T] \to \hei$ is a
\emph{minimizing geodesic from $S$} if it is a unit speed
length-minimizing curve, such that
\[
\gamma(0)\in S, 
\qquad
\mathrm{dist}_S(\gamma(t)) = t,\qquad \forall \,t\in[0,T].
\]
If $\gamma(0)=p$, we say that the minimizing geodesic $\gamma$ \emph{is based} at $p$. If $\gamma_p$ is a minimizing geodesic from $S$ based at $p$, by \cite[Thm.~12.4]{book_control} there exists an initial covector $\lambda^\perp(p)\in T_p^*\hei$ satisfying the following
normalization and transversality conditions
\begin{equation}\label{eq:twocond}
2H(\lambda^\perp(p))=1,\qquad \langle \lambda^\perp(p), T_p S\rangle=0,
\end{equation}
and such that the following equality is satisfied 
\begin{equation}
    \label{eq:gammapi}\gamma_p(t)=\exp_p(t \lambda^\perp(p)).
\end{equation}
We refer to $\lambda^\perp$ as the \emph{normal covector} of $S$. Note that, for a fixed $p\in S$ which is non-characteristic, the two conditions given in \eqref{eq:twocond} determine $\lambda^\perp(p)$ up to a sign. The additional requirement that the signed distance is positive along $\gamma_p$ fixes the sign of $\lambda^\perp(p)$. Furthermore, if $p$ is a characteristic point, then $T_pS=\mathcal{D}_p$, therefore the two conditions in \eqref{eq:twocond} are incompatible since  $\langle\lambda,\mathcal{D}_p\rangle=0$ implies $H(\lambda)=0$. Therefore, we have the following.

\begin{lemma}
\label{lem:no_abn}
    Every nonconstant minimizing geodesic from $S$ is based at $S^\times=S\setminus \Sigma(S)$.
\end{lemma}

Conversely, if a unit-speed horizontal curve can be written as in \eqref{eq:gammapi} for some initial covector satisfying \eqref{eq:twocond} then it is a minimizing geodesic up to its cut time. 

\begin{defi}
Let $S=\partial\Omega\subset\hei$ be an embedded smooth surface, and let $p\in S^\times$. Then, the \emph{cut time} of the unique minimizing geodesic from $S$ based at $p$ is
\begin{equation}
t_{\mathrm{cut}}(p):=\sup\{T>0: \mathrm{dist}_S(\gamma_p(t))=t , \ \forall\, t\in [0,T]\}.
\end{equation}
The \emph{cut locus} of $S$ is the set of points where minimizing geodesics from $S$ lose minimality, namely  
\begin{equation}
    \cut(S):=\{\gamma_p(t_\cut(p)): p\in S^\times,  t_\cut(p)<\infty\}\subset\hei,
\end{equation} 
\end{defi}

Recalling that $\hei$ is complete, the half-tube around $S$ can then be explicitly described once the normal exponential map and the cut time are known. Indeed, the following equality holds 
\begin{equation}
\label{eq:preloc}
\mathcal T_S(\ve)=\{\exp^\perp_S(t,p)\,:\,p\in S^\times, \ 0< t \leq \min\{\eps,t_\mathrm{cut}(p)\}\},
\end{equation}
where we define the \emph{normal exponential map} on the complement 
  of the characteristic set $S^\times=S\setminus \Sigma(S)$ 
$$\exp^\perp_S:\mathbb R^+\times S^\times\to \hei,\qquad \exp^\perp_S(t,p)=\exp_p(t\lambda^\perp(p)).$$
\begin{remark}\label{rmk:orientation}
    The normal exponential map is a smooth map between oriented manifolds. Indeed, endowing $\mathbb R^+\times S^\times$ with the product orientation, the following map is an isomorphism of oriented vector bundles
    \begin{equation}
        d\exp^\perp|_{\{0\}\times S^\times}:T(\mathbb R^+\times S^\times)|_{\{0\}\times S}\to T\hei|_S,
    \end{equation}
    and $\exp^\perp_S$ is an orientation preserving immersion close to $\{0\}\times S^\times$.
\end{remark}

\begin{remark}\label{rmk:preloc3} The characterization given in \eqref{eq:preloc} permits us to define the (positive) half-tube around an open set $U\subset S$ as follows
\begin{equation}
\label{eq:preloc3}
\mathcal T_U(\ve)=\{\exp^\perp_S(t,p)\,:\,p\in U\cap S^\times, \ 0< t \leq \min\{\eps,t_\mathrm{cut}(p)\}\},
\end{equation}
namely, the set of points in the half-tube reached by minimizing geodesics from $S$ based at $U$.  All our results on the tube around $S$ will hold also for its local counterpart.
\end{remark}

With the final goal of achieving a formula for the volume of the tube, we further investigate the general properties of the cut time function.

\subsection{General properties of the cut time function}\label{sec:cut_time}

Let $S=\partial\Omega\subset\hei$ be a smooth, embedded, oriented surface. Let us define the notion of \emph{Maxwell time} of a point $p\in S^\times$ as 
    \begin{equation}\label{eq:t_Max}
        t_{\Max}(p)=\inf\{t>0\,:\,\exists\,\, \bar p\in S^\times\setminus\{p\}\,\text{ \rm s.t. }\exp^\perp_S(t,p) = \exp^\perp_S(t,\bar p)\},
    \end{equation}
 and the notion of \emph{focal time} (or \emph{conjugate time}) of a point $p\in S^\times$ as 
 \begin{equation}\label{eq:t_foc}
        t_{\foc}(p)=\inf\{t>0\,:\,d_{(t,p)}\exp^\perp_S \text{ is not surjective}\}.
    \end{equation}
    
    We have the following classical characterization of the cut time function. We include the proof for the sake of completeness.
\begin{theorem}\label{thm:t_cut=min_t_max_t_foc}
    For every $p\in S^\times$, it holds $t_\cut(p)=\min\{t_\foc(p),t_\Max(p)\}$.
\end{theorem}
\begin{proof} 
We first prove the lower bound $t_\cut(p)\geq\min\{t_\foc(p),t_\Max(p)\}$ and it is not restrictive to assume $t_\cut(p)<\infty$. Fix $p\in S^\times$ and let $\gamma(t)=\exp^\perp_S(t,p)$, $t\in \R^+:=[0,+\infty)$. Take a monotone decreasing sequence of positive real numbers $\{\ve_n\}$ convergent to $0$, and set $$q_n = \gamma(t_\cut(p)+\ve_n).$$ By definition of cut time 
  \begin{equation}
        \mathrm{dist}_S(q_n)<t_\cut(p)+\ve_n.
  \end{equation}
 By completeness of the Heisenberg group, we can now find a sequence $\{\sigma_n\}$ of constant speed geodesics realizing the distance between $S$ and $\{q_n\}$, i.e., 
\begin{equation}
     \sigma_n:[0,1]\to \hei,\qquad\sigma_n(0)\in S,\qquad \sigma_n(1)=q_n,\qquad\ell(\sigma_n)=\mathrm{dist}_S(q_n).
 \end{equation} 
 The $\sigma_n$'s are parametrized by constant speed, and their lengths are uniformly bounded. Therefore by the Ascoli-Arzel\`a theorem, up to sub-sequences, there exists a limit 
 \begin{equation}
     \sigma=\lim_{n}\sigma_n,
 \end{equation}
 in the uniform topology. Moreover $\sigma$ is Lipschitz and 
 \begin{equation}
\ell(\sigma)\leq\liminf_{n\to\infty} \ell(\sigma_n)=\liminf_{n\to\infty} \mathrm{dist}_S(q_n)=\mathrm{dist}_S(\sigma(1)).
 \end{equation}
 It follows that $\sigma$ is a length minimizing curve and therefore it is smooth. We reparametrize $\sigma$ by unit speed, thus $\sigma:[0,t_\cut(p)]\to \hei$, $\ell(\sigma)=t_\cut(p)$.
 If $\sigma\neq \gamma$ then, $\sigma$ is a length minimizer realizing the distance from $S$, thus there exists $\bar p\in S^\times$, $\bar p\neq p$ such that 
 \begin{equation}
     \sigma(t)=\exp^\perp_S(t,\bar p),\qquad \sigma(t_\cut(p))=\gamma(t_\cut(p)).
 \end{equation}
 It follows that  $t_\cut(p)\geq t_{\Max}(p)$. If otherwise $\sigma=\gamma$, then, up to reparametrizing $\sigma_n$ by constant speed, there exists $p_n\in S^\times$ converging to $p$ such that 
 \begin{equation}
      \sigma_n(t)=\exp^\perp_S(t,p_n),\qquad \forall\, t\in [0,\mathrm{dist}_S(q_n)].
 \end{equation}
By definition of $q_n$ and $\sigma_n$ we have
		\begin{equation}
	    \exp^\perp_S(\mathrm{dist}_S(q_n),p_n)=\exp^\perp_S(t_\cut(p)+\ve_n,p).
		\end{equation}
Therefore the map $\exp^\perp_S$ coincides on the two sequences $(\mathrm{dist}_S(q_n),p_n)\neq(t_\cut(p)+\ve_n,p)$ with the common limit $(t_\cut(p),p)$. It follows that $d_{(t_\cut(p),p)}\exp^\perp_S$ is not injective (thus non surjective) and therefore $t_\cut(p)\geq t_{\foc}(p)$. 

We now prove the upper bound $t_\cut(p)\leq\min\{t_\foc(p),t_\Max(p)\}$, for every $p\in S^\times$. The proof of $t_\cut(p)\leq t_\Max(p)$ is similar to the case where $S$ is a point, given for instance in \cite[Lem.\ 9]{RizziSerres}. Assume by contradiction that there exists $\ve>0$ such that $t_{\cut}(p)\geq t_{\Max}(p)+\ve$. By definition of $t_\Max(p)$, there is $\tau\in [t_{\Max}(p),t_{\Max}(p)+\ve)$ which is a Maxwell time, i.e., we can find $\bar p\in S^\times$, $\bar p\neq p$ such that $\exp^\perp_S(\tau, \bar p)=\exp^\perp_S(\tau, p)$.
Therefore the curve
\begin{equation}
    \sigma(t):[0,t_{\Max}(p)+\ve]\to \mathbb H^1,\qquad \sigma(t):=\begin{cases}
        \exp^\perp_S(t,\bar p),\qquad &0\leq t\leq \tau,\\
        \exp^\perp_S(t,p),\qquad &\tau\leq t\leq t_{\Max}(p)+\ve ,
    \end{cases}
\end{equation} 
is continuous, piecewise smooth, but not real-analytic. Furthermore, since $t_{\cut}(p)\geq t_\Max(p)+\ve$, $\sigma$ is a length minimizing curve. This gives a contradiction, since minimizing geodesics must be real-analytic. 

Finally, to prove that $t_\cut(p)\leq t_\foc(p)$, one has to show that a geodesic from $S$, based at $p$, is no longer minimizing after its first focal (or conjugate) point. This can be deduced studying the second variation of the length as done in \cite[Ch.\ 8]{Agrachev} for sub-Riemannian minimizing geodesics from a point, in absence of abnormal minimizers. The case of minimizers from a surface has been proved in \cite[Thm.\ 4.3]{GPT2003}.
\end{proof}
\begin{remark} Theorem~\ref{thm:t_cut=min_t_max_t_foc} says that cut points are either focal points or points reached by at least two minimizers. This dichotomy is classical in Riemannian geometry, see e.g.\ \cite{DoCarmo}, and proved in \cite[Ch.\ 8]{Agrachev} for sub-Riemannian minimizing geodesics from a point, in absence of abnormal minimizers. Here we extend such a result for minimizers from a surface $S$ in the Heisenberg group. 
\end{remark}

We now show that $t_\cut(p)>0$ for every $p\in S^\times$. This fact is already contained in \cite[Thm.\ 1.1]{BossioRizziRossi2024}, however our proof is different and gives additional information on $t_\foc$ and $t_\Max$.

\begin{corollary}\label{cor:pos_cut_time}
    For every $p\in S^\times$, it holds $t_\foc(p)$, $t_\Max(p)>0$. In particular, $t_\cut(p)>0$.
\end{corollary}

\begin{proof}
    We show that both $t_\foc(p)$ and $t_\Max(p)$ are strictly positive. By definition of the normal exponential map, $\exp^\perp_S|_{t=0}:S^\times\to \hei$ coincides with the inclusion of $S^\times $ in $\hei$. Furthermore $\exp^\perp_{S*}\partial_t|_{t=0}$ is transverse to $S^\times$, therefore $d_{(0,p)}\exp^\perp_S$ is injective and $t_\foc(p)>0$. Now assume by contradiction that $t_\Max(p)=0$. Then there exists $t_n>0$, $t_n\to 0$, and $p_n$ in $S^\times \setminus\{p\}$ such that 
    \begin{equation}
        \exp^\perp_S(t_n,p_n)=\exp^\perp_S(t_n,p).
    \end{equation}
    Note that $p_n\to p$, indeed 
    \begin{equation}
    \begin{aligned}
        d_{SR}(p_n,p)&\leq d_{SR}(p_n,\exp^\perp_S(t_n,p))+d_{SR}(\exp^\perp_S(t_n,p),p)\\
        &=d_{SR}(p_n,\exp^\perp_S(t_n,p_n))+d_{SR}(\exp^\perp_S(t_n,p),p)\leq 2t_n.
    \end{aligned}
    \end{equation}
    Therefore, in any neighborhood of $(0,p)$ there exists $(t_n,p_n)$ such that $\exp^\perp_S(t_n,p)=\exp^\perp_S(t_n,p_n)$, contradicting the fact that $d_{(0,p)}\exp^\perp_S$ is injective.
\end{proof}
\begin{corollary}\label{c:umaiuscolo}
    Let $U\subset S$ be an open set and $T:U\to \R$ be a continuous function such that $T(q)\leq t_\cut(q)$ for all $q\in U$. Then defining $\mathcal{U}=\{(t,q)\in\R^+\times U\,:t<T(q)\}$ the restriction
    \begin{equation}
\exp_S^\perp\big|_\mathcal{U}:\mathcal{U}\to \hei
    \end{equation}
    is a diffeomorphism onto its image. 
\end{corollary}
\begin{proof}
    Indeed, assume that there exists $(t,p),(\bar t,\bar p)\in \mathcal U$ such that 
     \begin{equation}
         \exp^\perp_S(t,p)=\exp^\perp_S(\bar t,\bar p)=:x.
     \end{equation}
     Since $t<t_\cut(p)$ and $\bar t< t_\cut(\bar p)$, then $t=\bar t=\mathrm{dist}_S(x)$. But then, by \cref{thm:t_cut=min_t_max_t_foc}, $t\geq t_\Max(p)\geq t_\cut(p)$, which is a contradiction. This shows that $\exp^\perp_S$ is injective on $\mathcal U$. In addition, since $t_\cut(p)\leq t_\foc(p)$ for every $p\in S^\times$ by \cref{thm:t_cut=min_t_max_t_foc}, then $\exp^\perp_S|_\mathcal U$ is an immersion. Thus, $\exp^\perp_S|_\mathcal U$ is an injective immersion between manifolds of the same dimension and hence a diffeomorphism onto its image.
\end{proof}
An important consequence of \cref{thm:t_cut=min_t_max_t_foc} is the continuity of the cut time function. 

\begin{corollary}\label{cor:t_cut_continuous}
    The cut time function $t_\cut:S^\times\to (0,\infty]$ is continuous. Moreover, extending $t_\cut$ by zero on $\Sigma(S)$, then $t_\cut:S\to [0,\infty]$ is continuous.  
\end{corollary}

\begin{proof}
    The first part of the statement can be shown following the same arguments as \cite[Prop.\ 8.76]{Agrachev} in which one proves separately the upper and the lower semicontinuity. The second part of the statement instead follows if we prove that, for every $p\in\Sigma(S)$ and for every sequence $\{p_n\}_{n\in\N}\subset S^\times$ converging to $p$, it holds 
    \begin{equation}
        t_\cut(p_n) \xrightarrow{n\to\infty} 0.
    \end{equation}
    Assume by contradiction that there is a subsequence $t_\cut(p_n)\geq \ve>0$, for every $n\in \N$. By definition of $t_\cut$, this means that, for every $n\in \N$, the unique minimizing geodesic from $S$ based at $p_n$ satisfies 
    \begin{equation}
    \label{eq:geod_p_n_k}
        \mathrm{dist}_S(\gamma_{p_n}(t))=t,\qquad\forall\,t\in [0,\ve].
    \end{equation}
    By a standard Ascoli-Arzel\`a argument, up to extracting a subsequence, the sequence of curves $\{\gamma_{{p_n}|_{[0,\ve]}}\}$ converges to $\bar\gamma:[0,\ve]\to \hei$. By \eqref{eq:geod_p_n_k} and by the continuity of $\mathrm{dist}_S$, $\bar \gamma$ is a minimizing geodesic from $S$. Moreover, it is based at $\lim_{n\to\infty}\gamma_{p_n}(0) = \lim_{n\to\infty} p_n = p$, which is a contradiction with \cref{lem:no_abn}.   
\end{proof}
In turn, the continuity of the cut time implies that the cut locus has zero Lebesgue measure in $\hei$. 

\begin{corollary}
\label{cor:cut_S_negligible}
    The cut locus of $S$ is negligible.
\end{corollary}

\begin{proof}
    As the cut time is a continuous function, cf.\ \cref{cor:t_cut_continuous}, its graph over $S^\times$, namely the set
    \begin{equation}
        \Gamma:=\{(t_\cut(p),p): p\in S^\times,   t_\cut(p)<\infty\}\subset\R^+\times S^\times, 
    \end{equation}
    has zero Lebesgue measure. Furthermore, the cut locus of $S$ satisfies the identity:
    \begin{equation}
        \cut(S) = \{\exp^\perp_S(t_\cut(p),p): p\in S^\times, t_\cut(p)<\infty\} = \exp^\perp_S(\Gamma).
    \end{equation}
    Since $\exp^\perp_S$ is a smooth function on $\R^+\times S^\times$, we conclude that $\cut(S)$ has zero Lebesgue measure.
\end{proof}

\subsection{A formula for the volume of the tube}
\label{sec:tube1}
We can now prove the formula for the volume of $\mathcal T_U(\ve)$ localized in an open set  $U\subset S$. Recall that, in $\hei$ there is a canonical measure, which we denote by $\vol$, given by the bi-invariant Haar measure. This measure coincides with the Lebesgue measure $\Leb^3$ and can be described in terms of the dual frame \eqref{eq:the_coframe} as 
\[
\vol(A) = \int_A \nu_0\wedge\nu_1\wedge \nu_2,
\]  
for every Lebesgue-measurable subset $A\subset\hei$. 
\begin{theorem}\label{thm:formula_for_volume}
    Let $S=\partial\Omega\subset\hei$ be a smooth, oriented, embedded surface. Let $U\subset S$ be a Borel subset, then 
    \begin{equation}
    \label{eq:formula_for_volume}
        \mathrm{vol}(\mathcal T_U(\ve))=\int_U\int_0^{\min\{\ve,t_\cut\}}\exp^{\perp*}_S(\nu_0\wedge\nu_1\wedge\nu_2).
    \end{equation}
\end{theorem}
\begin{proof}
Assume $U\subset S$ is open. Set $T_\ve:=\{(t,p)\in\mathbb R^+\times S^\times : p\in S^\times\cap U,  0<t<\min\{\ve, t_\cut(p)\}\}$. By \cref{c:umaiuscolo} $\exp^\perp_S|_{T_\ve}$ is a diffeomorphism onto its image.
 Now, observe that, by construction, we have 
\begin{equation}
    \mathcal T_U(\ve) \setminus \exp^\perp_S(T_\ve) = \{\exp^\perp_S(\min\{t_\cut(p),\ve\},p):p\in S^\times\cap U\}.
\end{equation}
The latter set is the image through the smooth map $\exp^\perp_S$ of the graph of the continuous map $p\mapsto \min\{t_\cut(p),\ve\}$, therefore it is negligible. Thus, 
\begin{equation}\label{eq:cut_negligible}
    \vol(\mathcal T_U(\ve)) = \vol(\exp^\perp_S(T_\ve)) = \int_{\exp^\perp_S(T_\ve)}\nu_0\wedge\nu_1\wedge\nu_2. 
\end{equation}
According to \cref{rmk:orientation}, $\exp^\perp_S$ is orientation preserving close to $\{0\}\times S^\times$. Since $\exp^\perp_S|_{T_\ve}$ is an embedding, and $T_\ve$ intersects every open neighborhood of $\{0\}\times U$, we deduce that $\exp^\perp_S|_{T_\ve}$ is an orientation preserving diffeomorphism onto its image, therefore we may apply the change of variable formula, finding 
\begin{equation}
\begin{aligned}
    \vol(\mathcal T_U(\ve))=\int_{T_\ve}\exp^{\perp*}(\nu_0\wedge\nu_1\wedge\nu_2)=\int_U\int_0^{\min\{\ve,t_\cut(p)\}}\exp^{\perp*}_S(\nu_0\wedge\nu_1\wedge\nu_2),
\end{aligned}
\end{equation}
where the last equality follows from Fubini's Theorem. Finally, if $U\subset S$ is a Borel set, it is enough to approximate it by open sets and exploit \eqref{eq:formula_for_volume} and dominated convergence theorem to conclude.
\end{proof}

\begin{remark}\label{rmk:additivity_of_vol}
    A consequence of the above theorem is that if $U,V\subset S$ are disjoint Borel subsets we have 
    \begin{equation}
        \vol(\mathcal T_{U\cup V}(\ve))=\vol(\mathcal T_{U}(\ve))+\vol(\mathcal T_{V}(\ve)).
    \end{equation}
\end{remark}
\begin{notation}\label{notation:determinant}
    Given a $C^1$ map $\varphi:M\to \mathbb R^n$, where $M$ is a smooth $n$-dimensional manifold, and local coordinates $(x_1,\dots,x_n)$, we fix the following notation for the Jacobian determinant of $\varphi$ evaluated at the point $p\in M$ of coordinates $(y_1,\dots,y_n)$ as
    \begin{equation}
        \left.\det\left(\frac{\partial \varphi(x_1,\dots,x_n)}{\partial(x_1,\dots,x_n)}\right)\right|_{(y_1,\dots,y_n)}:=\det\Big(\partial_{x_1}\varphi(y_1,\dots,y_n),\dots,\partial_{x_n}\varphi(y_1,\dots,y_n)\Big).
    \end{equation}
   We omit the evaluation whenever $(y_1,\dots,y_n)=(x_1,\dots,x_n)$.
\end{notation}
\begin{remark}\label{rmk:determinant}
    Let $(x,y)$ be an oriented pair of local coordinates on $S$, then computing the pullback of $\exp^{\perp*}_S(\nu_0\wedge\nu_1\wedge\nu_2)$ with respect to the coordinates $(t,x,y)\in\mathbb R^+\times S^\times$ gives 
    \begin{equation}
         \mathrm{vol}(\mathcal T_U(\ve))=\int_U\int_0^{\min\{\ve,t_\cut(p)\}}\det\left(\frac{\partial\exp^\perp_S(t,x,y)}{\partial(t,x,y)}\right)dtdxdy.
    \end{equation}
\end{remark}
 \section{Normal exponential map and curvature forms near characteristic points}\label{sec:normal_exp} We now study the normal exponential map near a non-degenerate characteristic point $p$ of $S$. Making use of the results of \cref{ssec:char_points} on the local structure of $S$ near $p$, and in particular of the embedding of \cref{rmk:embedding}, we may assume that $p=(0,0,0)$ and that
\begin{equation}\label{eq:S_near_p}
    S=\left\{(x,y,z)\in\hei\,:\, z=\frac{1}{2}u(x,y)\right\},
\end{equation}
where $u\in C^\infty(\mathbb R^2)$ satisfies the following for some $a,b\in\mathbb R$
\begin{equation}\label{eq:ab_u1}
            u(x,y)=\frac{1}{2}(ax^2+by^2)+o(x^2+y^2),\qquad \text{as } (x,y)\to (0,0).
\end{equation}
 We stress that we work under the non-degenerate assumption $ab+1\neq 0$. We orient $S$ so that its positive side is $\{z>\tfrac{1}{2}u(x,y)\}$. To compute the normal exponential map for a surface of the form \eqref{eq:S_near_p}, we first compute its positive normal covector $\lambda^\perp$.
\begin{lemma}\label{lem:lambda_perp}
    Let $h_i:S^\times \to\mathbb R$, for $i=0,1,2$, be the components of $\lambda^\perp$ with respect to the dual frame  given by  \eqref{eq:the_coframe}, i.e. $h_i=\langle\lambda^\perp,X_i\rangle$. Then the following equalities hold
    \begin{equation}
        h_0=\frac{2}{\sqrt{(\partial_x u+y)^2+(\partial_y u-x)^2}},\qquad h_1=-\frac{h_0}{2}(\partial_x u+y),\qquad h_2=-\frac{h_0}{2}(\partial_y u-x).
    \end{equation}
\end{lemma}
\begin{proof}
    Let $f(x,y,z)=z-u(x,y)/2$.  First observe that, thanks to \eqref{eq:twocond} the covector $\lambda^\perp(p)$ is parallel to $df|_{p}$ and the fact that $H$ is homogeneous of degree two gives the normalization
    \begin{equation}\label{eq:lambdaperpd}
        \lambda^\perp(p)=\frac{df|_{p}}{\sqrt{2H(df|_{p})}}.
    \end{equation}
    We then write $df$ with respect to the dual frame  \eqref{eq:the_coframe}  as follows
    \begin{equation}\label{eq:df}
    \begin{aligned}
         df=dz-\frac{\partial_x u}{2}\,dx-\frac{\partial_y u}{2}\,dy &=dz+\frac{1}{2}(ydx-xdy)-\frac{1}{2}(ydx-xdy)-\frac{\partial_x u}{2}\,dx-\frac{\partial_y u}{2}dy\\
         &=\nu_0-\frac{1}{2}(\partial_x u+y)\nu_1-\frac{1}{2}(\partial_y u-x)\nu_2.
    \end{aligned}
    \end{equation}
    We now compute the denominator in \eqref{eq:lambdaperpd}
    \begin{equation}\label{eq:2H(df)}
        \sqrt{2H(df)}=\sqrt{\langle df,X_1\rangle^2+\langle df,X_2\rangle^2}=\frac{1}{2}\sqrt{(\partial_x u+y)^2+(\partial_y u-x)^2}.
    \end{equation}
    Plugging \eqref{eq:df} and \eqref{eq:2H(df)} into \eqref{eq:lambdaperpd} and decomposing $\lambda^\perp=h_0\nu_0+h_1\nu_1+h_2\nu_2$, we obtain the result.
\end{proof}
\begin{example}\label{ex:homog2}
    For $a,b\in\R$ with $ab+1\neq 0$, consider the homogeneous surface 
    \[
    S_{a,b} = \left\{(x,y,z)\in\hei : z = \frac{1}{4}(ax^2+by^2)\right\}.
    \]
    Substituting the explicit expression of the defining function $u$ into the formulae of \cref{lem:lambda_perp}, we find 
    \begin{equation}
        h_0=\frac{2}{\sqrt{(ax+y)^2+(by-x)^2}},\quad h_1=\frac{-(ax+y)}{\sqrt{(ax+y)^2+(by-x)^2}},\quad h_2=\frac{-(by-x)}{\sqrt{(ax+y)^2+(by-x)^2}}.
    \end{equation}
\end{example}
\begin{remark}\label{rmk:diver_h0}
    We stress that, while $h_1(x,y),h_2(x,y)$ remain bounded, $h_0(x,y)$ diverges as $(x,y)\to (0,0)$, indeed from the asymptotic expansion \eqref{eq:ab_u1} we see that 
    \begin{equation}
        h_0(x,y)=\frac{2}{\sqrt{(ax+y)^2+(by-x)^2}}+O(1).
    \end{equation}
    Notice in particular that the quantity $h_0^{-2}$ is smooth in a neighborhood of the characteristic point.  Similar expansions hold for $h_1$ and $h_2$, indeed denoting $r=\sqrt{x^2+y^2}$ we have 
    \begin{equation}
        h_1(x,y)=\frac{-(ax+y)}{\sqrt{(ax+y)^2+(by-x)^2}}+O(r),\qquad h_2=\frac{-(by-x)}{\sqrt{(ax+y)^2+(by-x)^2}}+O(r),
    \end{equation}
    where the remainders are smooth in polar coordinates.
\end{remark}
We now compute the normal exponential map for a general $S$ in terms of the normal covector. In the following proposition, for $s\in \mathbb R$ we employ the notation
\begin{equation}
\label{eq:rotation}
        \mathfrak{R}(s)=\begin{pmatrix}
            \cos(s) &\sin(s)\\
            -\sin(s) & \cos(s)
        \end{pmatrix}.
    \end{equation}
    In the next statement, to lighten the notation, we use the shorthand $h_i=h_i(x,y)$, for $i=0,1,2$.

\begin{prop}\label{prop:normal_exponential_map}
    Let $h_0,h_1,h_2$ be the components of the normal covector and let $(X,Y,Z)$ be the components of the normal exponential map in coordinates, i.e., 
    $$\exp^\perp_S:\mathbb  R^+\times \mathbb R^2\setminus\{(0,0)\}\to \hei,\qquad \exp^\perp_S=(X,Y,Z).$$
    Then, the first two components are
\begin{equation}
\label{eq:simplified_xy}
\begin{pmatrix}
    X(t,x,y) \\ Y(t,x,y)
\end{pmatrix}=\begin{pmatrix}
    x \\ y
\end{pmatrix}+\frac{1}{h_0}\Big[\mathfrak{R}(-\pi/2)+\mathfrak{R}(\pi/2-h_0t)\Big]\begin{pmatrix}
    h_1 \\ h_2
\end{pmatrix}
\end{equation}
whereas the third one reads  
\begin{equation}
\label{eq:simplified_z}
    Z(t,x,y)=\frac{u(x,y)}{2}+\frac{h_0 t - \sin(h_0t)}{2 h_0^2}
			+ \frac{1}{2h_0}\begin{pmatrix}
        x & y
    \end{pmatrix}\Big[\mathrm{Id} +\mathfrak{R}\left(\pi-h_0t\right)\Big]\begin{pmatrix}
        h_1 \\ h_2
    \end{pmatrix}.
    \end{equation}
\end{prop}
\begin{proof}
For fixed $(x,y)\in \mathbb R^2\setminus\{(0,0)\}$, let $q=(x,y, \frac12 u(x,y))$. Let $L_q:\mathbb H^1\to\mathbb H^1$ denote the left-translation by $q$. By \eqref{eq:exp_left-invariance}, it holds
\begin{equation}
	\exp^\perp_S(t,x,y) = \exp_q(t\lambda^\perp(x,y))=L_q\circ \exp_0\circ L_q^*(t\lambda^\perp(x,y)),
\end{equation}
where $\exp_0:T_0^* \mathbb H^1\to\mathbb H^1$ is the exponential map at the identity. Since the dual frame  \eqref{eq:the_coframe} is left-invariant, the computation of $L_q^*(\lambda^\perp(x,y))$ is straightforward and gives
\begin{equation}
	\begin{aligned}
		L_q^*(\lambda^\perp(x,y))
        &=h_0(x,y)dz+h_1(x,y)dx+h_2(x,y)dy.
	\end{aligned}
\end{equation}
Combining the latter with the explicit expression of $\exp_0$, given in \eqref{eq:exp_0_explicit}, we may compute $\exp^\perp_S$. More precisely, collecting the first two components into a single $\mathbb R^2$-valued vector and using $h_1^2+h_2^2=1$, we find
\begin{equation}
\begin{aligned}
    L_q\circ \exp_0\circ L_q^*(t\lambda^\perp(x,y))
    &=\left(x,y,\frac{u(x,y)}{2}\right)\star\left(\frac{1}{h_0}\left(\mathfrak{R}(-\tfrac\pi2)+\mathfrak{R}(\tfrac\pi2-h_0 t)\right)\begin{pmatrix}
        h_1 \\
        h_2
    \end{pmatrix},\,\frac{h_0 t - \sin(h_0 t)}{2 h_0^2}\right).
\end{aligned}
\end{equation}
Recalling that the group law is given by \eqref{eq:group_law}, we compute the first two coordinates of the above multiplication, finding \eqref{eq:simplified_xy}. To prove \eqref{eq:simplified_z}, we compute 
\begin{equation}
\begin{aligned}
    Z&=\frac{u(x,y)}{2}+\frac{h_0 t - \sin(h_0 t)}{2 h_0^2}+\frac{1}{2h_0}(-y,x)\cdot\left(\mathfrak{R}(-\pi/2)+\mathfrak{R}(\pi/2-h_0t)\right)\begin{pmatrix}
    h_1 \\ h_2
\end{pmatrix}\\
&=\frac{u(x,y)}{2}+\frac{h_0 t - \sin(h_0t)}{2 h_0^2}+\frac{1}{2h_0}(x,y)\cdot \mathfrak{R}(\pi/2)\left(\mathfrak{R}(-\pi/2)+\mathfrak{R}(\pi/2-h_0t)\right)\begin{pmatrix}
    h_1 \\ h_2
\end{pmatrix}\\
&=\frac{u(x,y)}{2}+\frac{h_0 t - \sin(h_0 t)}{2 h_0^2}+\frac{1}{2h_0}(x,y)\cdot \left(\mathrm{Id}+\mathfrak{R}(\pi-h_0t)\right)\begin{pmatrix}
    h_1 \\ h_2
\end{pmatrix}. 
\end{aligned}
\end{equation}
\end{proof}
The above proposition shows that $\exp^\perp_S(t,x,y)$ does not extend smoothly to the origin $(x,y)=(0,0)$, since, by \cref{rmk:diver_h0}, $h_0(x,y)$ diverges as $(x,y)\to (0,0)$. Nevertheless, the resulting singularity can be removed changing the time parameter $t$, or equivalently, by a suitable reparametrization of the normal geodesics from $S$.
\begin{prop}\label{lem:smooth_E}
    Let $S\subset \hei$ be the surface of \eqref{eq:S_near_p}, then the following map is smooth:
    \begin{equation} \label{eq:smooth_E}
        E:\R^+\times \mathbb R^2\to \hei,\qquad E(s,x,y):=\exp^\perp_S\left(\frac{s}{h_0(x,y)},x,y\right).
    \end{equation}
\end{prop}
\begin{proof}
    We write the normal exponential map of $S$ componentwise as $\exp^\perp_S=(X,Y,Z)$ and observe that $E(s,x,y)=(X,Y,Z)|_{(s/h_0,x,y)}$. Then by \cref{prop:normal_exponential_map} and \cref{lem:lambda_perp} we have
    \begin{equation}
\left.\begin{pmatrix}
    X \\ Y
\end{pmatrix}\right|_{(s/h_0,x,y)}=\begin{pmatrix}
    x \\ y
\end{pmatrix}-\frac{1}{2}\Big[\mathfrak{R}(-\pi/2)+\mathfrak{R}(\pi/2-s)\Big]\begin{pmatrix}
   \partial_x u+y\\\partial_y u-x
\end{pmatrix},
\end{equation}
from which we see that the first two components are smooth and well defined at $(0,0)$. Similarly for the third component we find 
\begin{equation}
    Z(s/h_0,x,y)=\frac{u(x,y)}{2}+\frac{s - \sin(s)}{2 h_0^2}
			-\frac{1}{4}\begin{pmatrix}
        x & y
    \end{pmatrix}\Big[\mathrm{Id} +\mathfrak{R}\left(\pi-s\right)\Big]\begin{pmatrix}
   \partial_x u+y\\\partial_y u-x
\end{pmatrix}.
    \end{equation}
    By \cref{lem:lambda_perp}, the function $h_0^{-2}$ extends smoothly to $(x,y)=(0,0)$, therefore also the third component is smooth. This shows that the map $E$ is well-defined and smooth on $\R^+\times\R^2$.
\end{proof}  
 
\subsection{Normal exponential map for homogeneous surfaces}\label{ssec:exp_regularity}

Let $S_{a,b}$ be a homogeneous surface. The normal exponential of $S_{a,b}$ inherits the homogeneity property \eqref{eq:homogeneity_S_ab}. To better describe the latter property, we introduce a pair of polar coordinates adapted to $S_{a,b}$: define $(\rho,\vartheta)\in \mathbb R^+\times \mathbb S^1$ such that
        \begin{equation}\label{eq:x-y_polar}
            \begin{pmatrix}
            \rho\cos\vartheta\\
            \rho\sin\vartheta
            \end{pmatrix}=\begin{pmatrix}
                ax+y\\
                by-x
            \end{pmatrix}.
        \end{equation}

\begin{remark}\label{rmk:hom_or}
Under the assumption $1+ab\neq 0$, the above polar coordinates are well-defined. The parametrization  $(\rho,\vartheta)\mapsto(x,y)$ determined by~\eqref{eq:x-y_polar} has Jacobian determinant $\rho/(1+ab)$.
Hence, it is orientation-reversing when $1+ab<0$.
\end{remark}

\noindent As a consequence of \cref{ex:homog2}, in such coordinates, the components of the normal covector $\lambda^\perp$ are
        \begin{equation}
        \label{eq:components_covector}
            h_0=\frac{2}\rho,\qquad h_1=-\cos\vartheta,\qquad h_2=-\sin\vartheta.
        \end{equation} 
We use the notation $\exp^\perp_{S_{a,b}}(t,\rho,\vartheta)$ for the normal exponential map expressed in coordinates \eqref{eq:x-y_polar}.  
\begin{lemma}
    \label{lem:homogeneity_normal_exp}
    Let $a,b\in\mathbb R$ such that $ab+1\neq 0$ and let $S_{a,b}\subset\hei$ be the homogeneous surface as defined in \cref{def:homogeneous_surf}. Then, for every $(t,\rho,\vartheta)\in\mathbb R^+\times (0,+\infty)\times\mathbb S^1$ it holds
    \begin{equation}
        \exp^\perp_{S_{a,b}}(t,\rho,\vartheta)=\delta_\ve(\exp^\perp_{S_{a,b}}(t/\ve,\rho/\ve,\vartheta)),
    \end{equation}
    where $\delta_\ve$ is the dilation defined in \eqref{eq:dilations}.
\end{lemma}

\begin{proof}
    Fix $\ve>0$ and $p=(\rho,\vartheta)\in S_{a,b}$. A standard computation shows that 
    \begin{equation}
    \label{eq:dilation_exp}
        \delta_\ve^{-1}(\exp_{S_{a,b}}^\perp(t,\rho,\vartheta))= \delta_\ve^{-1}\left(\exp_p(t\lambda^\perp(p))\right) = \exp_{\delta_\ve^{-1}(p)} (t\ve^{-2}\delta_\ve^*(\lambda^\perp(p))).
    \end{equation}
    Now, in polar coordinates, we compare $\lambda^\perp(p)$ and $\delta_\ve^*(\lambda^\perp(p))$ using \eqref{eq:components_covector}: 
    \begin{align}
        \delta_\ve^*(\lambda^\perp(p)) &= \delta_\ve^* \left.\left(\frac{2}{\rho}\nu_0 - \cos\vartheta\nu_1 -\sin\vartheta\nu_2\right)\right|_p =  \left.\left(\frac{2}{\rho}\ve^2\nu_0 - \ve\cos\vartheta\nu_1 -\ve\sin\vartheta\nu_2\right)\right|_{\delta_\ve^{-1}(p)}\\
        &=\ve\left.\left(\frac{2\ve}{\rho}\nu_0 - \cos\vartheta\nu_1 -\sin\vartheta\nu_2\right)\right|_{\delta_\ve^{-1}(p)}= \ve\lambda^\perp(\delta_\ve^{-1}(p)). 
    \end{align}
    Substituting the latter identity into \eqref{eq:dilation_exp}, we obtain
    \[
    \delta_\ve^{-1}(\exp_{S_{a,b}}^\perp(t,\rho,\vartheta)) =   \exp_{\delta_\ve^{-1}(p)}(t\ve^{-1}\lambda^\perp(\delta_\ve^{-1}(p))) = \exp_{S_{a,b}}^\perp(t/\ve,\rho/\ve,\vartheta),
    \]
    which is the claimed identity.  
\end{proof}
\begin{remark}\label{rmk:E_ab_approx}
    Let $E_{a,b}(s,\rho,\vartheta)$ be the map defined in \cref{lem:smooth_E} associated to the surface $S_{a,b}$, expressed in  the coordinates of \eqref{eq:x-y_polar}. By definition of $E_{a,b}$ and by \eqref{eq:components_covector} we have 
    \begin{equation}
    \begin{aligned}
                E_{a,b}(s,\rho,\vartheta)=\exp^\perp_{S_{a,b}}(\tfrac{s\rho}{2},\rho,\vartheta)=\delta_\ve(\exp^\perp_{S_{a,b}}(\tfrac{s\rho}{2\ve},\tfrac{\rho}\ve,\vartheta))=\delta_\ve(\exp^\perp_{S_{a,b}}(\tfrac{s}{h_0(\rho/\ve,\vartheta)},\tfrac{\rho}\ve,\vartheta))=\delta_\ve(E_{a,b}(s,\tfrac\rho\ve,\vartheta)).
    \end{aligned}
    \end{equation}
\end{remark}

We conclude the section with an approximation lemma for the normal exponential map of a general surface $S$, in terms of the normal exponential map of $S_{a,b}$.  In particular, we show that the  singularities of $\exp^\perp_S$ at a characteristic point $p$ behave, at first order, as the ones of the homogeneous model at $p$.

\begin{lemma}\label{lem:approximation_exp_new}
    Let $S=\partial\Omega\subset\hei$ be a smooth, oriented, embedded surface with a non-degenerate characteristic point at $p=(0,0,0)$ and let $a,b\in\mathbb R$ be the numbers identified by \cref{lem:nice_f}. Then, there exists a smooth function $R_S:\R^+\times (0,\infty)\times\R^+\times\mathbb S^1\to  \hei$ such that, in coordinates \eqref{eq:x-y_polar}, 
    \begin{equation}
        \exp^\perp_{S}(t,\rho,\vartheta)=\delta_\ve(\exp^\perp_{S_{a,b}}(t/\ve,\rho/\ve,\vartheta)+\ve R_S(t/\ve,\rho/\ve,\rho,\vartheta)).
    \end{equation}
    for every $(t,\rho,\vartheta)\in\mathbb R^+\times (0,+\infty)\times\mathbb S^1$. 
\end{lemma}
\begin{proof}
    Firstly, observe that the statement makes sense since \eqref{eq:x-y_polar} are well-defined coordinates on $S$, being $p$ a non-degenerate characteristic point. Secondly, by the homogeneity properties of the Hamiltonian, we have that
    \[
    \delta_\ve^{-1}(\exp^\perp_{S}(t,\rho,\vartheta))
    = \pi\circ e^{t\vec{H}}\big(\ve^{-2}\delta_\ve^*(\lambda^\perp(\rho,\vartheta))\big) = \pi\circ e^{\frac{t}{\ve}\vec{H}}\big(\ve^{-1}\delta_\ve^*(\lambda^\perp(\rho,\vartheta))\big). 
    \]
    Now, combining \cref{lem:lambda_perp} with the Taylor expansion of $u$ of \cref{lem:nice_f} and with \eqref{eq:components_covector}, we deduce that, in polar coordinates, it holds
    \begin{equation}
    \label{eq:behavior_components_covector}
    h_0(\rho,\vartheta) = \frac2\rho+H_0(\rho,\vartheta),\qquad h_1(\rho,\vartheta) = -\cos\vartheta+\rho H_1(\rho,\vartheta),\qquad h_2(\rho,\vartheta) = -\sin\vartheta+\rho H_2(\rho,\vartheta), 
    \end{equation}
    where $H_i\in C^\infty(\R^+\times\mathbb S^1)$, for $i=0,1,2$. Since the dual frame $\{\nu_0,\nu_1,\nu_2\}$ satisfies  
    \[
    \delta_\ve^*\nu_0=\ve^2\nu_0,\qquad \delta_\ve^*\nu_1=\ve\nu_1,\qquad \delta_\ve^*\nu_2=\ve\nu_2,
    \]
    in coordinates $p=(\rho,\vartheta)$, we obtain
    \begin{align}
        \ve^{-1}\delta_\ve^*\lambda^\perp_S(p) &=  \left.\left[\left(\frac{2\ve}{\rho}+\ve H_0(\rho,\vartheta)\right)\nu_0+(-\cos\vartheta+\rho H_1(\rho,\vartheta))\nu_1
    +(-\sin\vartheta+\rho H_2(\rho,\vartheta))\nu_2\right]\right|_{\delta_\ve^{-1}(p)}\\ 
    &=\lambda^\perp_{S_{a,b}}(\rho/\ve,\vartheta)+\ve G(\rho/\ve,\rho,\vartheta),
    \end{align}    
    for some smooth covector $G$. Since the flow $e^{t\vec{H}}$ at a fixed time $t$ is a smooth function on $T^*\hei$, we compute
    \begin{align}
        \pi\circ e^{\frac{t}{\ve} \vec{H}}\big(\ve^{-1}\delta_\ve^*(\lambda^\perp(p))\big) &= \pi\circ e^{\frac{t}{\ve} \vec{H}}\big(\lambda^\perp_{S_{a,b}}(\rho/\ve,\vartheta)+\ve G(\rho/\ve,\rho,\vartheta)\big)\\ 
        &= \pi\circ e^{\frac{t}{\ve} \vec{H}}\big(\lambda^\perp_{S_{a,b}}(\rho/\ve,\vartheta)\big)+\int_0^1 \frac{d}{ds}\left[\pi\circ e^{\frac{t}{\ve}\vec{H}}\big(\lambda^\perp_{S_{a,b}}(\rho/\ve,\vartheta)+s\ve G(\rho/\ve,\rho,\vartheta)\big)\right]ds\\
        &= \exp^\perp_{S_{a,b}}(t/\ve,\rho/\ve,\vartheta)+\ve R_S(t/\ve,\rho/\ve,\rho,\vartheta).
    \end{align}
    where $R_S$ is the smooth map defined implicitly by the last identity. This concludes the proof.
\end{proof}
\begin{remark}\label{rmk:E_approx}
    An analogous statement of \cref{lem:approximation_exp_new} holds also for the map $E$, namely there exists a smooth function $R_E:[0,\infty)^3\times\mathbb S^1\to\hei$ such that, in coordinates \eqref{eq:x-y_polar}, 
    \begin{equation}
        E(s,\rho,\vartheta)=\delta_\ve(E_{a,b}(s,\rho/\ve,\vartheta)+\ve R_E(s,\rho/\ve,\rho,\vartheta)).
    \end{equation}
    for every $(s,\rho,\vartheta)\in\mathbb R^+\times \R^+\times\mathbb S^1$. Observe that the regularity of $R_E$ is guaranteed by the regularity of $E$, proved in \cref{lem:smooth_E}.     
\end{remark}

\subsection{The Jacobian determinant of the normal exponential map near characteristic points}
\label{sec:normal_jacobian}
We now study the singularities of the Jacobian determinant of the normal exponential map near a non-degenerate characteristic point.

\begin{defi}
	A smooth map
	$\xi:[0,r_0)\times\mathbb S^1\to\mathbb R^2$
	is called a system of polar coordinates centered at the origin if
	$\xi(0,\theta)=(0,0)$, its restriction to
	$(0,r_0)\times\mathbb S^1$ is a diffeomorphism onto a punctured
	neighborhood of the origin, and
	\begin{equation}
		\det\left(\frac{\partial\xi(r,\theta)}{\partial(r,\theta)}
		\right)
		=r\,D(r,\theta),
	\end{equation}
	where $D$ is smooth and non-vanishing on $[0,r_0)\times\mathbb S^1$. If in addition $D$ is positive, then the system of polar coordinates is called positive.
\end{defi}
 \begin{remark}
    Let $\xi$ be a system of polar coordinates and $\chi:\mathbb R^+\times\mathbb S^1\to\mathbb R^+\times\mathbb S^1$ a diffeomorphism, then $\xi\circ \chi$ is also a system of polar coordinates. If in addition $\xi$ is positive and $\partial_\theta\chi_2(0,\theta)>0$, then $\xi\circ \chi$ is also positive.
\end{remark}
On an oriented surface, a (positive) system of polar coordinates centered at a point
is defined through an (orientation-preserving) smooth chart
sending that point to the origin. In the sequel, for a smooth function $F:\mathbb R^2\to \mathbb R$ we abbreviate $F\circ\xi(r,\theta)$ by $F(r,\theta)$.

\begin{prop}
    \label{prop:iter_diver_singularities}
    Let $S=\partial\Omega\subset \hei$ be a smooth embedded oriented surface with a non-degenerate characteristic point $p\in\Sigma(S)$. Fix a system of polar coordinates $(r,\theta)$ on $S$ centered at $p$ and denote 
    \begin{equation}
        J_S(t,r,\theta)=\left.\det\left(\frac{\partial\exp^\perp_S(t,r,\theta)}{\partial(t,r,\theta)}\right)\right|_{(t,r,\theta)},\qquad (t,r,\theta)\in\mathbb R^+\times(0,+\infty)\times\mathbb S^1.
    \end{equation}
    Then, there exists a smooth function $\mathcal J_S:\mathbb R^+\times\mathbb R^+\times\mathbb S^1\to\mathbb R$ such that 
    \begin{equation}\label{eq:mathcal_J}
        J_S(t,r,\theta)=r^2\mathcal J_S(t/r,r,\theta),\qquad (t,r,\theta)\in \mathbb R^+\times(0,+\infty)\times\mathbb S^1.
    \end{equation}
\end{prop}

\begin{proof}
    By the chain rule and multiplicativity of the determinant, the validity of the statements is independent of the system of polar coordinates. Therefore we may prove the proposition for a particular set of polar coordinates. Let $(x,y)$ be coordinates near $p$ such that $S$ has locally the form \eqref{eq:S_near_p}, and let  $(\rho,\vartheta)$ be the polar coordinates defined in \eqref{eq:x-y_polar}. We first prove the statement for a $S=S_{a,b}$ (cf.\,\cref{def:homogeneous_surf}), with $a,b\in\R$ such that $ab+1\neq 0$. By \cref{lem:homogeneity_normal_exp} we have
    \begin{equation}
        \exp^\perp_{S_{a,b}}(t, \rho,\vartheta)=\delta_\ve(\exp^\perp_{S_{a,b}}(t/\ve,\rho/\ve ,\vartheta)),\qquad\forall\, (t,\rho,\vartheta) \in \R^+\times \R^+\times\mathbb S^1,\,\ve>0.
    \end{equation}
    For a fixed $\ve>0$, taking the Jacobian determinant of both sides of the above identity yields 
    \begin{align}\label{eq:J_a,b}
        J_{S_{a,b}}(t,\rho,\vartheta)=\ve^2 J_{S_{a,b}}(t/\ve, \rho/\ve,\vartheta).
    \end{align}
    Evaluating at $\ve=\rho$ we find 
    \begin{equation}
        J_{S_{a,b}}(t,\rho,\vartheta)=\rho^2 J_{S_{a,b}}(t/\rho,1,\vartheta).
    \end{equation}
    In this case the proof can be concluded setting $\mathcal J_{S_{a,b}}(u,\rho,\vartheta):=J_{S_{a,b}}(u,1,\vartheta)$, which is a smooth function. For a general surface $S$, thanks to \cref{lem:approximation_exp_new}, we have 
    \begin{equation}
         \exp^\perp_{S}(t,\rho,\vartheta)=\delta_\ve(\exp^\perp_{S_{a,b}}(t/\ve,\rho/\ve,\vartheta)+\ve R_S(t/\ve,\rho/\ve,\rho,\vartheta)).
    \end{equation}
    Taking the Jacobian determinant of both sides yields 
    \begin{equation}
        J_S(t,\rho,\vartheta)=\ve^4(\ve^{-2}J_{S_{a,b}}(t/\ve,\rho/\ve,\vartheta)+\ve^{-1}\mathcal R(t/\ve,\rho/\ve,\rho,\vartheta)).
    \end{equation}
    for a suitable smooth function $\mathcal R:\R^+\times(0,+\infty)\times \R^+\times \mathbb S^1\to\R$. Substituting $\ve=\rho$ we are left with 
    \begin{equation}
        J_S(t,\rho,\vartheta)=\rho^2(J_{S_{a,b}}(t/\rho,1,\vartheta)+\rho\mathcal R(t/\rho,1,\rho,\vartheta)).
    \end{equation}
    We conclude the proof setting $\mathcal J_S(u,\rho,\vartheta):=J_{S_{a,b}}{(u,1,\vartheta)}+\rho \mathcal R(u,1,\rho,\vartheta)$.
\end{proof}
\begin{remark}\label{rmk:in_hom_coords}
    The proof of \cref{prop:iter_diver_singularities} in shows that, in the coordinates $(\rho,\vartheta)$ of \eqref{eq:x-y_polar} it holds 
    \begin{equation}
        \mathcal J_S(u,\rho,\vartheta)=J_{S_{a,b}}(u,1,\vartheta)+\rho R_J(u,\rho,\vartheta),
    \end{equation}
    where $R_J:\R^+\times \R^+\times \mathbb S^1\to\R$ is a smooth function.
\end{remark}

\subsection{Curvature forms: iterated divergences of the distance}\label{ssec:curve_form}
We relate the determinant of the normal exponential map to the iterated divergences of the sub-Riemannian distance from $S$. As a byproduct, we are able to describe the singularities of these divergences on the characteristic set.

Recall that, given a smooth vector field $X$ on $\hei$, its \emph{divergence} with respect to the Lebesgue volume form $\nu_0\wedge\nu_1\wedge\nu_2$, is defined as the unique smooth function satisfying 
\begin{equation}
    L_X(\nu_0\wedge\nu_1\wedge\nu_2)=\mathrm{div}(X)\nu_0\wedge\nu_1\wedge\nu_2,
\end{equation}
where $L_X$ denotes the Lie derivative in the direction of $X$. In particular, if $X\in \Gamma(T\hei)$ is a horizontal vector field, namely $X=u_1 X_1+u_2 X_2$, for some $u_1,u_2\in C^\infty(\hei)$, then 
\begin{equation}
    \mathrm{div}(X)=X_1(u_1)+X_2(u_2).
\end{equation}
For every $n\in \mathbb N$, the \emph{$n$-th iterated divergence} of $X$ is defined as the unique smooth function satisfying 
\begin{equation}
    L_{X}^{n}(\nu_0\wedge\nu_1\wedge\nu_2)=\mathrm{div}^{n}(X)\nu_0\wedge\nu_1\wedge\nu_2,
\end{equation}
 with the convention that $\mathrm{div}^0(X)=1$.

\begin{prop}
\label{prop:iter_diver}
    Let $S=\partial\Omega\subset \hei$ be a smooth, embedded, oriented surface and let $\mathrm{dist}_S:\hei\to\mathbb R$ be the signed distance from $S$. Given a pair of local coordinates $(x,y)$ on $S$, set 
    \begin{equation}
        J_S(t,x,y)=\left.\det\left(\frac{\partial \exp^\perp_S(t,x,y)}{\partial(t,x,y)}\right)\right|_{(t,x,y)}.
    \end{equation}
    Then, $\mathrm{dist}_S$ is smooth on  a neighborhood of $S^{\times}=S\setminus\Sigma(S)$ and, 
    \begin{equation}
        \left.\mathrm{div}^{n}(\nabla\mathrm{dist}_S)\right|_{(x,y)}=\frac{\partial_t^n J_S(0,x,y)}{J_S(0,x,y)},\qquad (x,y)\in S\setminus\Sigma(S),
    \end{equation}
    where $\nabla\mathrm{dist}_S$ is the horizontal gradient of $\mathrm{dist}_S$.
\end{prop}
\begin{proof}
     Let $p\in S^\times$ and let $U\subset S^\times$ be a compact neighborhood of $p\in S^\times$. By \cref{cor:pos_cut_time} there exists $\ve_0>0$ such that 
  $  
        t_\cut(q)> \ve_0$, for all $q\in U
    $.
    \cref{c:umaiuscolo} then implies that the restricted map 
    \begin{equation}
        [0,\ve_0]\times U\to \hei,\qquad (t,q)\mapsto \exp_S^\perp(t,q),
    \end{equation}
    is a diffeomorphism onto its image. By definition of $t_\cut$, it holds 
    \begin{equation}
    \label{eq:known_identity}
        \mathrm{dist}_S(\exp_S^\perp(t,q)) = t, \qquad \forall\, (t,q)\in[0,\ve_0]\times U
    \end{equation}
    which then shows that $\mathrm{dist}_S$ is smooth on a neighborhood of $p$ (notice that this covers only half a neighborhood of $p$; a similar argument can be done for $t$ negative). Furthermore, differentiating \eqref{eq:known_identity} with respect to $t$ and using the definition of the horizontal gradient we find 
    \begin{equation}
    \label{eq:nabla_deltaS_exp}
        g(\nabla\mathrm{dist}_S,\exp_{S*}^\perp\partial_t)=\langle d(\mathrm{dist}_S),\exp_{S*}^\perp\partial_t\rangle=\langle \exp_S^{\perp *}d(\mathrm{dist}_S),\partial_t\rangle =\langle dt,\partial_t\rangle =1.
    \end{equation}
     Note that $\exp_{S*}^\perp\partial_t$ has norm $1$, because it is the velocity of the arc-length parametrized geodesic $t\mapsto \exp_S^\perp(t,q)$. Moreover, $|\nabla\mathrm{dist}_S|=1$, indeed \cref{rmk:1_lip} implies that $|\nabla\mathrm{dist}_S|\leq 1$, which combined with \eqref{eq:nabla_deltaS_exp} implies $|\nabla\mathrm{dist}_S|=1$. Since both $\exp_{S*}^\perp\partial_t$ and $\nabla\mathrm{dist}_S$ have norm $1$, \eqref{eq:nabla_deltaS_exp} implies that 
    \begin{equation}\label{eq:nabla_dist}
        \nabla\mathrm{dist}_S=(\exp^\perp_S)_*\partial_t.
    \end{equation}
    Since the iterated divergences of a vector field with respect to a given volume form are diffeomorphism invariant, computing the divergence of $\nabla\mathrm{dist}_S=\exp^\perp_{S*}\partial_t$ with respect to the volume form  $\nu_0\wedge\nu_1\wedge\nu_2$ is equivalent to computing the divergence of $\partial_t$ with respect to the volume form 
    \begin{equation}\label{eq:pull_volume}
        \exp_S^{\perp*}(\nu_0\wedge\nu_1\wedge\nu_2)=J_S(t,x,y)dt\wedge dx\wedge dy.
    \end{equation}
    Now, the $n$-th Lie derivative along $X=\partial_t$ of a volume form above is
    \begin{equation}
        L_X^n (J_S(t,x,y) dt\wedge dx\wedge dy) = \partial^n_t J_S(t,x,y)dt\wedge dx\wedge dy = \frac{\partial^n_t J_S(t,x,y)}{J_S(t,x,y)} J_S(t,x,y)dt\wedge dx\wedge dy.
    \end{equation}
    We deduce that 
    \begin{equation}
        \left.\mathrm{div}^{n}(\nabla\mathrm{dist}_S)\right|_{\exp_S^\perp(t,x,y)}=\frac{\partial^n_t J_S(t,x,y)}{J_S(t,x,y)}.
    \end{equation}
    Evaluating at $t=0$, since $\exp_S^\perp(0,x,y)=(x,y)$, we conclude the proof.
\end{proof}

\begin{defi}\label{defi:curvature_forms}
    Let $S=\partial\Omega\subset \hei$ be a smooth oriented embedded surface. The sub-Riemannian area form of $S$ is defined, outside of the characteristic set of $S$, as 
    \begin{equation}
        \sigma_H=\iota_{N}(\nu_0\wedge\nu_1\wedge\nu_2).
    \end{equation}
    where $N=\nabla \mathrm{dist}_S$ is the sub-Riemannian unit normal. Given $n\in\mathbb N$, the \emph{$n$-th curvature form} is defined, outside of the characteristic set of $S$, as 
    \begin{equation}
        \sigma_S^{(n)}:=\mathrm{div}^n(\nabla\mathrm{dist}_S)\sigma_H.
    \end{equation}
\end{defi}
\begin{lemma}\label{lem:area_curv}
    Under the hypotheses and notation of \cref{prop:iter_diver}, for $n\geq 1$ it holds 
    \begin{equation}
        \sigma_H=J_S(0,x,y)dx\wedge dy,\qquad \sigma_S^{(n)}=\partial_t^n J_S(0,x,y)dx\wedge dy.
    \end{equation}
\end{lemma}
\begin{proof}
       The form $\sigma_H$ is computed contracting \eqref{eq:pull_volume} with the vector field \eqref{eq:nabla_dist} and evaluating at $t=0$. The expression of $\sigma_S^{(n)}$ is obtained by multiplying $\sigma_H$ by $\mathrm{div}^n(\nabla\mathrm{dist}_S)$ (cf.\,\cref{prop:iter_diver}).
\end{proof}
We are now in position to describe the singularities of $\sigma_S^{(n)}$ at characteristic points. 
\begin{corollary}\label{cor:curv_forms1}
   Let $S=\partial\Omega\subset \hei$ be a smooth embedded oriented surface with a non-degenerate characteristic point $p\in\Sigma(S)$. Fix a system of polar coordinates $(r,\theta)$ centered at $p$, and set 
   \begin{equation}\label{eq:def_phi_n}
       \phi_n(r,\theta):=\partial_t^n\mathcal J_S(0,r,\theta),\qquad (r,\theta)\in\mathbb R^+\times\mathbb S^1,\,n\geq 0.
   \end{equation}
   where $\mathcal J_S$ is the function of \cref{prop:iter_diver_singularities}. Then $\phi_n\in C^\infty(\mathbb R^+\times\mathbb S^1)$ and 
    \begin{equation}\label{eq:explicit_curv_form}
        \sigma^{(n)}_S=r^{2-n}\phi_n(r,\theta)dr\wedge d\theta.
    \end{equation}
    In particular, if $(\rho,\vartheta)$ are the polar coordinates of \eqref{eq:x-y_polar} we have 
    \begin{equation}\label{eq:in_hom_coords}
        \phi_n(\rho,\vartheta)=\partial_t^{n}J_{S_{a,b}}(0,1,\vartheta)+O(\rho),
    \end{equation}
    and the remainder is identically zero in the homogeneous case $S=S_{a,b}$.
    \end{corollary}
\begin{proof}
    To prove \eqref{eq:explicit_curv_form}, we apply \cref{lem:area_curv} in the polar coordinates $(r,\theta)$ centered at $p$ and then compute derivatives $\partial_t^n J_S(0,r,\theta)$ by applying $\partial_t^n$ to both sides of \eqref{eq:mathcal_J} and evaluating the result at $t=0$. Equation \eqref{eq:in_hom_coords} then follows from \cref{rmk:in_hom_coords}.
\end{proof}
\begin{corollary}\label{cor:curv_forms}
    Let $S=\partial\Omega\subset \hei$ be a smooth embedded oriented surface whose characteristic points are non degenerate. Then 
    \begin{itemize}
    \item[(i)] For $n=0,1,2$, the forms $\sigma^{(n)}_S$ is locally integrable on $S$,
    \item[(ii)] Given $p\in\Sigma(S)$ and letting $a,b\in\R$ defined by \eqref{eq:ab_u}, for every $n\geq 3$, if $a\neq b$, the form $\sigma^{(n)}_S$ is \emph{not} locally integrable on $S$.
    \end{itemize}
\end{corollary}
    \begin{proof}
        Since the curvature forms are smooth outside of the characteristic set (cf.\,\cref{lem:area_curv}), it is sufficient to control integrability only near characteristic points. Item $(i)$ is then a direct consequence of \eqref{eq:explicit_curv_form}. Concerning item $(ii)$, it follows from \eqref{eq:explicit_curv_form} combined with \eqref{eq:in_hom_coords} and \cref{prop:hom_curv_form} that 
        \begin{equation}\label{eq:sigma_n_proof}
            \sigma^{(n)}_S=\rho^{2-n}(A_n\cos(2\vartheta)+B_n\sin(2\vartheta)+C_n+O(\rho))d\rho\wedge d\vartheta.
        \end{equation}
        Moreover, \cref{prop:hom_curv_form} ensures that, for $a\neq b$ and $n\geq 1$, $A_n$ and $B_n$ cannot vanish simultaneously. Thus the form \eqref{eq:sigma_n_proof} is not locally integrable at $\rho=0$ for $n\geq 3$ and $a\neq b$.
    \end{proof}

\section{The volume of the half-tube. Proof of Theorems~\ref{thm:intro-steiner-hom} and \ref{thm:intro-steiner}}
\label{sec:volume_half-tube}

In this section, we prove \cref{thm:intro-steiner}. Our starting point is the formula
\begin{equation}
    \mathrm{vol}(\mathcal T_U(\ve))
    =\int_U\int_0^{\min\{\ve,t_\cut(p)\}}
    \exp^{\perp*}_S(\nu_0\wedge\nu_1\wedge\nu_2),
\end{equation}
established in \cref{thm:formula_for_volume}. The integrand, namely the Jacobian determinant of the normal exponential map, has been analyzed in the previous sections. The remaining ingredient is the behavior of the cut time, described by \cref{thm:intro-homogeneous-cut-time,thm:intro-cut-time-stability}. Although the proofs of these theorems appear later, in \cref{part:cut-time}, they do not rely on any results from the present section. We apply these theorems here to derive the volume asymptotics.

\subsection{Steiner formula for non-characteristic surfaces}
With the results collected so far, we give a streamlined proof of the Steiner tube formula for non-characteristic surfaces, originally obtained in  \cite{BaloghSteiner}. 
\begin{lemma}\label{lem:non_char}
     Let $S=\partial\Omega\subset \hei$ be a smooth oriented embedded surface, and let $U\subset S$ be an open relatively compact subset, such that $\overline{U}\cap\Sigma(S)=\emptyset$, then $\inf_{p\in U}t_\cut(p)>0$ and for any $0<\ve<\inf_{p\in U}t_\cut(p)$ it holds
    \begin{equation}
        \mathrm{vol}(\mathcal T_U(\ve))=\sum_{n=0}^{\infty}\frac{\ve^{n+1}}{(n+1)!}\int_U\sigma_S^{(n)}.
    \end{equation}
\end{lemma}

\begin{proof}
    By \cref{cor:pos_cut_time}, for any non characteristic $q\in S$, $t_\cut(q)>0$. Since $\overline U$ does not intersect the characteristic set, and $t_\cut$ is continuous by \cref{cor:t_cut_continuous}, there exists $\ve_0>0$ such that $t_\cut(q)\geq\ve_0$ for all $q \in U.$ Let $\ve<\ve_0$, and letting $(x,y)$ be oriented coordinates on $U$, by \cref{rmk:determinant}, in the notation of \cref{prop:iter_diver} we have 
    \begin{align}
        \mathrm{vol}(\mathcal T_U(\ve)) &=\int_U\int_0^{\min\{\ve,t_\cut(p)\}} J_S(t,x,y)\,dtdxdy=\int_0^\ve\left(\int_UJ_S(t,x,y)\,dxdy\right)\!dt.
    \end{align}
    Since $U$ is relatively compact and does not contain characteristic points, the exponential map is smooth, and real-analytic in $t$,  therefore
    \begin{equation}
        J_S(t,x,y)=\sum_{n=0}^{\infty}\frac{t^n}{n!}\partial_t^{n}J_S(0,x,y).
    \end{equation}
     We may substitute the above expression in the formula for the volume, and exchange the sum with the integral, since the series is uniformly convergent in compact subsets of $\mathbb R^+\times S^\times$ (this can be seen from the explicit expression of the exponential map given in \eqref{eq:exp_0_explicit}, see also \cite{BossioRizziRossi2024}), we find
    \begin{equation}
        \mathrm{vol}(\mathcal T_U(\ve))=\sum_{n=0}^\infty \int_0^\ve\frac{t^n}{n!}\left(\int_U\,\partial_t^{n}J_S(0,x,y)\,dxdy\right)dt=\sum_{n=0}^\infty \frac{\ve^{n+1}}{(n+1)!}\int_U\,\sigma_S^{(n)},
    \end{equation}
    where in the last equation we integrated with respect to $t$ and used \cref{lem:area_curv}.
\end{proof}

  \subsection{The half-tube around small sub-level sets of $t_\cut$.} When characteristic points are present, the cut time is not uniformly bounded below, see \cref{cor:t_cut_continuous}, and the behavior of the volume of the half-tube depends on the shape of the sub-level sets of the cut time.

\begin{lemma}\label{lem:rho_0}
    Let $S=\partial\Omega\subset \hei$ be a smooth oriented embedded surface containing only geometrically stable characteristic points. Let $U\subset S$ be a compact set having all its characteristic points in its interior. There exists $r_0>0$ such that 
   \begin{equation}
       B(r_0):=\{q\in U\cap S\,:\, t_\cut(q)\leq r_0\},
   \end{equation}
   contains all the characteristic points of $U$. Moreover $B(r_0)$ is a union of finitely many, disjoint, smooth embedded compact balls, each one containing exactly one characteristic point, and each one admitting a positive system of polar coordinates $(r,\theta)\in [0,r_0]\times\mathbb S^1$ such that
   \begin{equation}\label{eq:coords_cut}
       t_{\cut}(r,\theta)=r.
   \end{equation}
\end{lemma}
\begin{proof}
    Since geometrically stable characteristic points are non-degenerate, they are isolated  (see for instance \cite[Thm.\ 1.1]{BBP25}). Thus, being $U$ compact, $U\cap\Sigma(S)=\{p_1,\dots,p_k\}$. By \cref{cor:t_cut_continuous,cor:pos_cut_time}, $t_\cut(p)=0$ if and only if $p\in U\cap\Sigma(S)$. Therefore, $U\cap\Sigma(S)\subset B(r_0)$, for every $r_0>0$, and in addition, by the continuity of the cut time, for every neighborhood $V$ of $U\cap\Sigma(S)$ there exists $r_0>0$ sufficiently small such that $B(r_0)\subset V$. We take $V$ to be a union of disjoint balls $B_i$ containing the characteristic point $p_i$, $i=1,\dots,k$, each and every one of them supporting polar coordinates $(\rho,\vartheta)$ such that \cref{thm:intro-cut-time-stability} holds. In such coordinates the function $t_\cut(\rho,\vartheta)$ is smooth and satisfies 
    \begin{equation}
        t_\cut(\rho,\vartheta) = \left(\frac\pi2-\arctan\left(\max\{a,b\}\right)\right)\rho+o(\rho),\qquad\text{as }\rho\to 0^+.
    \end{equation}
     Thus, by the inverse function theorem, there exists $\rho_0>0$ such that the following map
    \begin{equation}\label{eq:mapp_t_cut}
        [0,\rho_0]\times\mathbb S^1\to \R^+\times\mathbb S^1,\qquad (\rho,\vartheta)\mapsto (t_\cut(\rho,\vartheta),\vartheta)=:(r,\theta)
    \end{equation}
    is a smooth diffeomorphism onto its image, which contains $[0,r_0]\times\mathbb S^1$ (up to reducing $r_0>0$). This gives polar coordinates $(r,\theta)$, such that $t_\cut(r,\theta)=r$, defined on the whole $[0,r_0]\times\mathbb S^1$. Moreover, up to changing $\theta\mapsto-\theta$, we may assume that $(r,\theta)$ are positively oriented polar coordinates.
    Finally, the existence of such polar coordinates $(r,\theta)$ implies that, for each $i=1,\dots,k$, $B_i\cap B(r_0)=\{(r,\theta)\in B_i\,: r\leq r_0\}$ is a smooth closed ball.
\end{proof}
\begin{remark}\label{rmk:hom_rho_0}
    The above lemma also holds for homogeneous surfaces, with the only requirement that their characteristic point is not degenerate, i.e., for $S_{a,b}$ with $1+ab\neq 0$. Indeed in such case \cref{thm:intro-homogeneous-cut-time} gives $t_\cut(\rho,\vartheta)=c_{a,b}\rho$, where $c_{a,b}>0$. Thus $(r,\theta):=(c_{a,b}\rho,\vartheta)$ are polar coordinates in which the cut time coincides with the radius, well defined on every ball $B_{r_0}=\{c_{a,b}\rho\leq r_0\}$, for $r_0>0$. To ensure positive orientation of the pair $(r,\theta)$ we make the additional change $\theta\mapsto \mathrm{sign}(1+ab)\theta$, cf.\,\cref{rmk:hom_or}.
\end{remark}
\noindent We now compute the asymptotics of the volume of the half-tube around a small neighborhood of a geometrically stable characteristic point. 
\begin{lemma}\label{lem:strcuture_of_integral}
    Let $S=\partial\Omega\subset \hei$ be a smooth oriented embedded surface containing only geometrically stable characteristic points and let $B\subset S$ be one of the balls of \cref{lem:rho_0} with the coordinates $(r,\theta)$ of \eqref{eq:coords_cut}. Then, the function $\Phi:\mathbb R^2\to \mathbb R$ defined by 
    \begin{equation}\label{eq:F(s,r)}
        \Phi(s,r):=r^3\int_0^{2\pi}\int_0^{s}\mathcal J_S(u,r,\theta)dud\theta,
    \end{equation}
    where $\mathcal J_S$ is defined in \eqref{eq:mathcal_J} is smooth. Furthermore, for $0\leq \ve\leq r_0$, the following formula holds 
    \begin{equation}\label{eq:strcuture_of_integral}
        \vol(\mathcal T_B(\ve))=\int_0^{\ve}\Phi(1,r)dr+\int_\ve^{r_0}\Phi\left(\frac{\ve}{r},r\right)dr.
    \end{equation}
\end{lemma}
\begin{proof}
    The smoothness of $\Phi$ defined in \eqref{eq:F(s,r)} is a direct consequence of \cref{prop:iter_diver_singularities}. We now show \eqref{eq:strcuture_of_integral}. Recall that, according to \cref{lem:rho_0}, in coordinates $(r,\theta)$, we have $B=\{(r,\theta):0\le r\le r_0,\ 0\le\theta\le2\pi\}$ and $t_\cut(r,\theta)=r$. Thus, starting from the identity \eqref{eq:formula_for_volume}, we express the integrals in coordinates $(t,r,\theta)$, using the notation of \cref{prop:iter_diver_singularities}:
    \begin{equation}
    \begin{split}
        \vol(\mathcal T_B(\ve))&=\int_B\int_0^{\min\{\ve,t_\cut\}}\exp^{\perp*}(\nu_0\wedge\nu_1\wedge\nu_2)=\int_0^{2\pi}\int_0^{r_0}\int_0^{\min\{\ve,r\}} J_S(t,r,\theta)dtdrd\theta\\
         &=\int_0^{r_0}r^3\int_0^{2\pi}\int_0^{\min\{\frac{\ve}{r},1\}}    \mathcal J_S(u,r,\theta)dud\theta dr,
    \end{split}
    \end{equation}
    where, in the last equality, we have performed the change of coordinates $t=ur$, used \eqref{eq:mathcal_J} and exchanged the order of integration in $r$ and $\theta$. Therefore, by definition of $\Phi(s,r)$ we have 
    \begin{equation}
    \begin{aligned}
         \vol(\mathcal T_B(\ve))&=\int_0^{r_0}\Phi\left(\min\left\{\frac{\ve}{r},1\right\},r\right)dr=\int_0^{\ve}\Phi(1,r)dr+\int_\ve^{r_0}\Phi\left(\frac{\ve}{r},r\right)dr.
    \end{aligned}
    \end{equation}
\end{proof}
\begin{remark}\label{rmk:Phi_hom}
    In the case of a non-degenerate homogeneous surface $S_{a,b}$, the above lemma holds on the balls $B=B_{r_0}=\{c_{a,b}\rho <r_0\}$, for $r_0>0$. Applying the change of variables $(r,\theta)=(c_{a,b}\rho,\pm\vartheta)$ we find 
    \begin{equation}
        \mathcal J_{S_{a,b}}(u,r,\theta)=\mathcal J_{S_{a,b}}(u,\theta)=\pm \left(\frac{1}{c_{a,b}}\right)^3 J_{S_{a,b}}(c_{a,b}u,1,\pm\theta),
    \end{equation}
    where $\pm=\mathrm{sign}(1+ab)$ and $J_{S_{a,b}}(t,\rho,\vartheta)$ is the Jacobian of the normal exponential map computed with respect to the coordinates $(t,\rho,\vartheta)$, and $c_{a,b}=\pi/2-\arctan(\max\{a,b\})$. In particular
    \begin{equation}\label{eq:Phi_a,b}
        \Phi_{a,b}(s,r)=r^3\int_0^{2\pi}\int_0^s\mathcal J_{S_{a,b}}(u,\theta)du\,d\theta.
    \end{equation}
\end{remark}
The asymptotics of the first integral in \eqref{eq:strcuture_of_integral}, as $\ve\to 0^+$ is a standard Taylor expansion as the function $\Phi(1,\cdot)$ is smooth. In order to describe the asymptotic structure of the second integral in \eqref{eq:strcuture_of_integral}, we introduce the finite part.  

    \begin{defi}[Finite part of a function]
    \label{def:finite_part_fun}
	Let $f:(0,+\infty)\to\mathbb R$ be a smooth function satisfying 
	 \begin{equation}\label{eq:finite-part-definition}
		f(\delta)=\sum_{j=1}^M c_{j}\delta^{-j}
		+c_{\ell}\log\delta+c_0+o(1),
		\qquad \text{as }\delta\to0^+.
	\end{equation}
	Then, we call $c_0$ the \emph{finite part} of $f$ and write $\Fp_{\delta\to0^+}f(\delta):=c_0$.
\end{defi}

  \begin{remark}
  \label{rmk:anti_matter}
    The definition of finite part of a function is well-posed, indeed the coefficients of the expansion \eqref{eq:finite-part-definition} are uniquely determined by $f$, as a consequence of \cref{lem:log_expansion_intrinsic}. Moreover, observe that $\Fp_{\delta\to0^+}f(\delta)$ is linear in $f$ and agrees with the limit $\lim_{\delta\to 0^+}f(\delta)$ whenever the latter exists.
\end{remark} 

The proof of the next lemma is deferred to Appendix~\ref{app:noname}.

   \begin{lemma}\label{l:integrato}
   	Let $r_0>0$ and let
   	$G:[0,1]\times[0,r_0]\to\mathbb R$ be smooth. Set
   	\begin{equation}\label{eq:mathcal_G}
   		\mathcal G(\ve)
   		=\int_\ve^{r_0}G\left(\frac\ve{r},r\right)\,dr.
   	\end{equation}
   	There are uniquely determined real coefficients $\gamma_n,\beta_n$
   	such that, for every integer $N\geq0$,
   	\begin{equation}\label{eq:singular-integral-expansion}
   		\mathcal G(\ve)
   		=\sum_{n=0}^N\ve^n
   		\bigl(\gamma_n+\beta_n\log\ve\bigr)+o(\ve^N),
   		\qquad \ve\to0^+.
   	\end{equation}
   	Moreover, $\gamma_0=\int_0^{r_0}G(0,r)\,dr$,
   	$\beta_0=0$, and, for $n\geq1$,
   	\begin{equation}\label{eq:abstract_coefficients}
   		\begin{aligned}
   			\gamma_n
   			&=\Fp_{\delta\to0^+}\left(
   			\frac1{n!}\int_\delta^{r_0}
   			\frac{\partial_s^nG(0,r)}{r^n}\,dr
   			+\frac1{(n-1)!}\int_\delta^1
   			\frac{\partial_r^{n-1}G(s,0)}{s^{n+1}}\,ds
   			\right),\\
   			\beta_n
   			&=-\frac{\partial_s^n\partial_r^{n-1}G(0,0)}
   			{n!(n-1)!}.
   		\end{aligned}
   	\end{equation}
   \end{lemma}
   We are in position to obtain the asymptotic expansion of the volume of the half-tube around a characteristic point.

   \begin{prop}\label{prop:expansion_vol_around_char_pts}
   	Let $S=\partial\Omega\subset\hei$ be a smooth oriented embedded surface whose
   	characteristic points are geometrically stable, and let $B$ be one of the balls of \cref{lem:rho_0}. There are uniquely determined real coefficients $\alpha_n(B)$, $\beta_n(B)$,
   	$n\geq4$, such that, for every integer $N\geq4$,
   	\begin{equation}\label{eq:VOLLL}
   		\begin{aligned}
   			\vol(\mathcal T_B(\ve))
   			=\sum_{n=1}^3\frac{\ve^n}{n!}\int_B\sigma_S^{(n-1)}+\sum_{n=4}^N\ve^n
   			\bigl(\alpha_n(B)+\beta_n(B)\log\ve\bigr)+o(\ve^N), \qquad \text{as }\ve\to 0^+.
   		\end{aligned}
   	\end{equation}
   	In the polar coordinates $(r,\theta)$ of \cref{lem:rho_0},
   	where $t_\cut(r,\theta)=r$, the logarithmic coefficients are
   	\begin{equation}\label{eq:beta_n}
   		\beta_n(B)
   		=-\frac1{n!(n-4)!}
   		\int_0^{2\pi}\partial_r^{n-4}\phi_{n-1}(0,\theta)\,d\theta,
   		\qquad n\geq4,
   	\end{equation}
    where $\phi_n\in C^\infty(\R^+\times\mathbb S^1)$ is the smooth function defined in \cref{cor:curv_forms1}.
   \end{prop}
   
   \begin{proof}
   	By \cref{lem:strcuture_of_integral}, we have the identity
   	\begin{equation}\label{eq:local-volume-split}
   		\vol(\mathcal T_B(\ve))
   		=\int_0^\ve\Phi(1,r)\,dr
   		+\int_\ve^{r_0}\Phi(\ve/r,r)\,dr,\qquad 0<\ve\leq r_0.
   	\end{equation}
   	Due to the factor $r^3$ in \eqref{eq:F(s,r)}, the first integral in \eqref{eq:local-volume-split} has a Taylor
   	expansion starting at order four. Thus, for every $N\geq 4$, it holds:
   	\begin{equation}
   		\int_0^\ve\Phi(1,r)\,dr=\sum_{n=4}^N\omega_n\ve^n+o(\ve^N),\qquad\text{as }\ve\to 0^+.
   	\end{equation}
   	 An application of \cref{l:integrato}, with $\Phi$ in place of $G$, to the
   	second integral implies the existence of real coefficients $\gamma_n$, $\beta_n$ for $n\in\N$, such that, for every $N\geq 0$, we have  
   	\begin{equation}
   		\int_\ve^{r_0}\Phi(\ve/r,r)\,dr=\sum_{n=0}^{N}\ve^n(\gamma_n+\beta_n\log\ve)+o(\ve^N),\qquad\text{as }\ve\to 0^+.
   	\end{equation}
   Additionally, \cref{l:integrato} gives the explicit expression of the coefficients: for $n=0$ we have 
   \begin{equation}\label{eq:zero-coeff}
   	\gamma_0=\int_0^{r_0}\Phi(0,r)\,dr=0,\qquad \beta_0=0.
   \end{equation}
   For $\gamma_n$, with $n=1,2,3$, due to the factor $r^3$ in \eqref{eq:F(s,r)}, $\partial_r^{n-1}\Phi(s,0)=0$ and, using \eqref{eq:abstract_coefficients}, we find
   \begin{equation}\gamma_n
   			=\Fp_{\delta\to0^+}\left(
   			\frac1{n!}\!\int_\delta^{r_0}
   			\frac{\partial_s^n\Phi(0,r)}{r^n}\,dr
   			+\frac1{(n-1)!}\int_\delta^1
   			\frac{\partial_r^{n-1}\Phi(s,0)}{s^{n+1}}\,ds
   			\right)=\Fp_{\delta\to0^+}\!\left(
   		\frac1{n!}\int_\delta^{r_0}
   		\frac{\partial_s^n\Phi(0,r)}{r^n}\,dr\right)\!.
   \end{equation}
   Recalling the definition of $\Phi$ in \eqref{eq:F(s,r)}, we may use \cref{cor:curv_forms1}, and in particular \eqref{eq:def_phi_n}, to obtain 
   \begin{equation}\label{eq:FI}
   	\begin{aligned}
   		\partial_s^n\Phi(0,r)&=r^3\int_{0}^{2\pi}\phi_{n-1}(r,\theta)d\theta.
   	\end{aligned}
   \end{equation}
   Therefore
   \begin{equation}
       \gamma_n=\frac{1}{n!}\Fp_{\delta\to0^+}\left(\int_{\delta}^{r_0}\int_{0}^{2\pi}r^{3-n}\phi_{n-1}(r,\theta)d\theta dr\right)=\frac{1}{n!}\int_{0}^{r_0}\int_{0}^{2\pi}r^{3-n}\phi_{n-1}(r,\theta)d\theta dr=\frac{1}{n!}\int_{B}\sigma_S^{(n-1)},
   \end{equation}
   where, in the last equality, we used that the $n$-th curvature form satisfies $\sigma_S^{(n)}=r^{2-n}\phi_n(r,\theta)\,dr\wedge d\theta$, cf.\,\cref{cor:curv_forms1}. 
  Similarly, the formula for $\beta_n$ in \eqref{eq:abstract_coefficients} gives zero for $n<4$ while, for $n\geq4$, gives
   	\begin{equation}
   		\begin{aligned}
   				\beta_n
   			&=-\frac{\partial_s^n\partial_r^{n-1}\Phi(0,0)}
   			{n!(n-1)!}=-\frac{1}
   			{n!(n-1)!}\partial_r^{n-1}\left.\left(r^3\int_{0}^{2\pi}\phi_{n-1}(r,\theta)d\theta\right)\right|_{r=0} \\
            &=-\frac1{n!(n-1)!}\binom{n-1}{3}3!\int_{0}^{2\pi}\partial_r^{n-4}\phi_{n-1}(0,\theta)d\theta\\
            &=-\frac1{n!(n-4)!}\int_{0}^{2\pi}\partial_r^{n-4}\phi_{n-1}(0,\theta)d\theta.
   		\end{aligned}
   	\end{equation}
        Therefore, we may conclude the proof setting 
    \begin{equation}
    \begin{aligned}
        \alpha_n(B)&:=\begin{cases}
            \displaystyle\frac{1}{n!}\int_B\sigma_S^{(n-1)},&n=1,2,3,\\[1em]
            \displaystyle \omega_n+\gamma_n, &n\geq 4,
        \end{cases}\\
        \beta_n(B)&:=\beta_n=-\frac{1}{n!(n-4)!}\int_0^{2\pi}\partial_r^{n-4}\phi_{n-1}(0,\theta)d\theta,\qquad n\in\mathbb N,\,n\geq 4.
    \end{aligned}
    \end{equation}
   \end{proof}
\begin{remark}\label{rmk:beta_n_hom}
    Collecting the observation of \cref{rmk:hom_rho_0} and \cref{rmk:Phi_hom} we see that the proof of the previous proposition applies to non-degenerate homogeneous surfaces $S_{a,b}$ as well. In particular, the expression of $\Phi_{a,b}$ in \eqref{eq:Phi_a,b} shows that the functions $\phi_n(r,\theta)$ of \eqref{eq:FI} are actually independent of $r$, i.e. $\phi_n(r,\theta)=\phi_n(\theta)$ in the homogeneous case. Therefore $\beta_n(B)\neq 0$ necessarily implies $n=4$. 
\end{remark}
\subsection{Steiner's formula for characteristic surfaces} 
Collecting the results of the previous sections we may prove \cref{thm:intro-steiner}.
\begin{proof}[Proof of \cref{thm:intro-steiner}]
    We apply \cref{lem:rho_0} to the set $U$ and find $r_0>0$ such that $B(r_0)$ is a finite union of disjoint embedded balls $B_1,\dots B_k$ centered at the (finitely many) characteristic points of $U$. Therefore $U$ is the disjoint union of $U\setminus B(r_0)$, $B_1,\dots,B_k$ and, as a consequence, using \cref{rmk:additivity_of_vol}, we have
\[
 \vol(\mathcal T_U(\ve)) = \vol(\mathcal T_{U\setminus B(r_0)}(\ve))+\sum_{i=1}^k\vol(\mathcal T_{B_i}(\ve)).
\]
Since the closure of the set $U\setminus B(r_0)$ does not contain characteristic points, the asymptotics of $\ve\mapsto\vol(\mathcal T_{U\setminus B(r_0)}(\ve))$, as $\ve\to0^+$, follows by \cref{lem:non_char} and it holds: 
\begin{equation}
\label{eq:VOLLL_1}
  \vol(\mathcal T_{U\setminus B(r_0)}(\ve))=\sum_{n=1}^{\infty}\frac{\ve^{n}}{n!}\int_{U\setminus B(r_0)}\sigma^{(n-1)}. 
\end{equation}
On the other hand, the asymptotics of $\ve\mapsto\vol(\mathcal T_{B_i}(\ve))$, as $\ve\to0^+$, for $i=1,\ldots,k$, is a consequence of \cref{prop:expansion_vol_around_char_pts}. More precisely, for every $i=1,\ldots,k$, there is a collection of real numbers $\{\alpha_n(B_i),\beta_n(B_i)\}_{n\in\N}$ such that, given $N\in\N$, it holds  
\begin{equation}\label{eq:VOLLL_2}
    \vol(\mathcal T_{B_i}(\ve))=\sum_{n=1}^3\frac{\ve^n}{n!}\int_{B_i}\sigma^{(n-1)}+\sum_{n=4}^{N}\ve^n(\alpha_n(B_i)+\beta_n(B_i)\log\ve)+o(\ve^N),\qquad\text{as }\ve\to0^+.
\end{equation}
Combining \eqref{eq:VOLLL_1} and \eqref{eq:VOLLL_2}, we find coefficients $\{\alpha_n(U),\beta_n(U)\}_{n\in\N,n\geq 4}$ such that \eqref{eq:intro-main-expansion} holds. Note that these coefficients are uniquely determined by \cref{lem:log_expansion_intrinsic}. 
   \end{proof}
   We prove the analogue statement for non-degenerate homogeneous surfaces.
   \begin{proof}[Proof of \cref{thm:intro-steiner-hom}]
       In light of the observations contained in \cref{rmk:hom_rho_0}, \cref{rmk:Phi_hom} and \cref{rmk:beta_n_hom} we may repeat verbatim the proof of \cref{thm:intro-steiner} deducing the validity of \eqref{eq:intro-main-expansion} in the homogeneous case as well. Moreover, by \cref{rmk:beta_n_hom}, $\beta_n(U)\neq 0$ implies $n=4$. We may compute $\beta_4(U)$ applying \eqref{eq:abstract_coefficients} with $G(s,r)=\Phi_{a,b}(s,r)$, with the latter computed in \cref{rmk:Phi_hom}, finding 
    \begin{equation}
        \beta_4=-\frac{\mathrm{sign}(1+ab)}{24}\int_0^{2\pi}\partial_s^3J_{a,b}(s,1,\mathrm{sign}(1+ab)\theta)|_{s=0}d\theta.
    \end{equation}
    The Jacobian $J_{a,b}(s,\rho,\vartheta)$ is computed explicitly in \cref{l:det_exp_perp}. In particular, integration with respect to $\theta$ has the effect of canceling the term $s(a-b)\cos(2\vartheta+2s)$ in \eqref{eq:hom_mathcal_J_explicit}, while multiplying the remaining piece by $2\pi$, leaving us with 
    \begin{equation}
    \begin{aligned}
         \beta_4&=\frac{-\pi}{48|1+ab|}\partial_s^3\Bigl[(1+ab)
			+(1-ab)\bigl(\cos(2s)+s\sin(2s)\bigr)
			+(a+b)\bigl(s\cos(2s)-\sin(2s)\bigr)
			\Bigr]\Big{|}_{s=0}.
    \end{aligned}
    \end{equation}
    The function $(1+ab)+(1-ab)\bigl(\cos(2s)+s\sin(2s)\bigr)$ is even so it has vanishing third derivative at zero. The remaining term $(a+b)\bigl(s\cos(2s)-\sin(2s)\bigr)$ has third derivative at zero $-4(a+b)$. This confirms the value of $\beta_4$ in \eqref{eq:intro-main-expansion_homogeneous}.
   \end{proof}
  
   \section{Distributional character of higher-order coefficients. Proof of Theorem~\ref{thm:intro-distributions}}\label{s:distri}
Recall that the \textit{metric projection} of $S$ is the (possibly multivalued) map
  \begin{equation}
  	{\sf p}:\hei\to S,\qquad {\sf p}(q) = \arg\min_{p\in S} d_{SR}(p,q).
  \end{equation}
  In particular, ${\sf p}(q)$ is the starting point of a geodesic realizing the distance from $q$ to $S$, i.e.,
  \begin{equation}
  	|\mathrm{dist}_S(q)| = d_{SR}(q,p),\qquad\forall\,p\in {\sf p}(q). 
  \end{equation}
  By \cref{thm:t_cut=min_t_max_t_foc} and \cref{cor:cut_S_negligible}, the set $\mathcal T_S(\ve)\setminus \cut(S)$ has full measure in $\mathcal T_S(\ve)$, and the restriction 
  \begin{equation}
  	{\sf p}_\ve:=\mathsf p|_{\mathcal T_S(\ve)\setminus \cut(S)}:\mathcal T_S(\ve)\setminus \cut(S)\to S
  \end{equation}
  is a smooth function, since, for every $p\in S^\times$ and every $t\in (0,t_\cut(p))$, it satisfies
  \begin{equation}
  	\label{eq:almost_left-inverse}
  	{\sf p}_\ve\circ\exp^\perp_S(t,p) = p.
  \end{equation}
    This also implies that the push-forward measure $({\sf p}_\ve)_\#\vol$ is well-defined and we have
  \begin{equation}
  	\vol(\mathcal T_S(\ve)) = \int_S d({\sf p}_\ve)_\#\vol = ({\sf p}_\ve)_\#\vol(S).
  \end{equation}
  Now, for a given $U\subset S$ and $\varphi\in C^\infty(U)$, we study the asymptotic expansion, at $\ve=0$, of the function: 
  \begin{equation}
  	\ve\mapsto \vol_\varphi(\mathcal T_U(\ve)) := \int_U\varphi\, d({\sf p}_\ve)_\#\vol = \int_{\mathcal T_U(\ve)} (\varphi\circ {\sf p}_\ve) d\vol.
  \end{equation}
  Note that, for $U=S$, $\vol_\varphi(\mathcal T_S(\ve))$ coincides with the pairing $\langle \mu_\ve,\varphi\rangle$ defined in \eqref{eq:intro-pairing}.  The next theorem can be proven following the arguments of \cref{sec:volume_half-tube}. We sketch its proof for the reader's convenience.

  \begin{theorem}
  	\label{thm:expansion_weighted_vol_S}
  	Let $S=\partial\Omega\subset \hei$ be a smooth, oriented, embedded surface. Assume that $\Sigma(S)$ consists only of geometrically stable characteristic points. Let $\varphi\in C_c^\infty(S)$. Then, there exist real numbers $\{\alpha_n(\varphi),\beta_n(\varphi)\}_{n\in\N}\subset \R$ such that, for every $N\geq 4$,  
  	\begin{equation}
  		\label{eq:expansion_claim}
  		\vol_\varphi(\mathcal T_S(\ve))=\alpha_1(\varphi)\ve+\alpha_2(\varphi)\ve^2+\alpha_3(\varphi)\ve^3+\sum_{n=4}^{N}\ve^n(\alpha_n(\varphi)+\beta_n(\varphi)\log\ve)+o(\ve^N),
  	\end{equation}
    as $\ve\to0^+$. Moreover, letting $\sigma_S^{(n)}$ be the curvature forms of \cref{defi:curvature_forms}, we have
  	\begin{equation}
  		\alpha_n(\varphi)=\frac{1}{n!}\int_S\varphi\,\sigma_S^{(n-1)},\qquad n=1,2,3.
  	\end{equation}
    Finally, if $S$ is homogeneous, the results hold requiring only non-degenerate characteristic points.
  \end{theorem}
  
  \begin{proof}
  	Firstly, combining \cref{thm:formula_for_volume} and the identity \eqref{eq:almost_left-inverse}, we may write
  	\begin{equation}
  		\label{eq:expression_vol_varphi}
  		\mathrm{vol}_\varphi(\mathcal T_S(\ve))=\int_S\varphi(p)\int_0^{\min\{\ve,t_\cut(p)\}}\exp^{\perp*}_S(\nu_0\wedge\nu_1\wedge\nu_2),
  	\end{equation}
  	Then, we consider a compact set $U\subset S$, having its characteristic points in its interior, such that $\supp \varphi\subset U$ (note that this set always exists since $\Sigma(S)$ consists of isolated points). Then, we have
    \[
        \vol_\varphi(\mathcal T_S(\ve)) = \vol_{\varphi}(\mathcal T_U(\ve)).
    \]
    We now split $U$ as in the proof of \cref{thm:intro-steiner} and then, by \cref{rmk:additivity_of_vol}, we can split the volume accordingly: 
  	\begin{equation}
  		\label{eq:splitting_vol_varphi}
  		\vol_{\varphi}(\mathcal T_U(\ve))=\vol_\varphi(\mathcal T_{U\setminus B(r_0)}(\ve))+\sum_{i=1}^k\vol_\varphi(\mathcal T_{B_i}(\ve)),
  	\end{equation}
  	where we recall that $B(r_0)$, defined in \cref{lem:rho_0} (see \cref{rmk:hom_rho_0} for the non-degenerate homogeneous case), is the finite union of the balls $B_i$, each of which contains a single characteristic point. Note that $\overline{U\setminus B(r_0)}$ does not contain characteristic points, thus $\inf_{U\setminus B(r_0)}t_\cut>0$. This means that, exploiting \eqref{eq:expression_vol_varphi}, one can easily adapt the proof of \cref{lem:non_char} to the weighted volume to obtain that the term $\vol_\varphi(\mathcal T_{U\setminus B(r_0)}(\ve))$ in \eqref{eq:splitting_vol_varphi} has the expansion: 
  	\begin{equation}
  		\label{eq:expansion1}
  		\mathrm{vol}_\varphi(\mathcal T_{U\setminus B(r_0)}(\ve))=\sum_{n=0}^{\infty}\frac{\ve^{n+1}}{(n+1)!}\int_{U\setminus B(r_0)}\varphi\,\sigma_S^{(n)}.
  	\end{equation}
  	For ease of notation, set $B:=B_i$, for some $i=1,\ldots,k$, and let $(r,\theta)$ be the polar coordinates identified by \cref{lem:rho_0}. To discuss the term $\vol_\varphi(\mathcal T_B(\ve))$ in \eqref{eq:splitting_vol_varphi}, using the notation of \cref{prop:iter_diver_singularities}, we define
  	\begin{equation}\label{eq:FI(s,r)_varphi}
  		\Phi_\varphi:\mathbb R^+\times\mathbb R^+\to \mathbb R,\qquad \Phi_\varphi(s,r):=r^3\int_0^{2\pi}\varphi(r,\theta)\int_0^{s}\mathcal J(u,r,\theta)du\,d\theta.
  	\end{equation}
  	As in the case where $\varphi\equiv 1$, $\Phi_\varphi$ is smooth and, repeating the proof of \cref{lem:strcuture_of_integral} (or \cref{rmk:Phi_hom} for the non-degenerate homogeneous case), one can show that, for $0\leq \ve\leq r_0$, it holds 
  	\begin{equation}\label{eq:structure_of_integral_varphi}
  		\vol_\varphi(\mathcal T_B(\ve))=\int_0^{\ve}\Phi_\varphi(1,r)dr+\int_\ve^{r_0}\Phi_\varphi\left(\frac{\ve}{r},r\right)dr.
  	\end{equation}
  	Reasoning as in the proof of \cref{prop:expansion_vol_around_char_pts} (see \cref{rmk:beta_n_hom} for the analogous statement for non-degenerate homogeneous surfaces), the first integral in \eqref{eq:structure_of_integral_varphi} is a smooth function at $\ve=0$, hence it admits a standard Taylor expansion. The second integral in \eqref{eq:structure_of_integral_varphi} can be discussed using \cref{l:integrato}, with $G=\Phi_\varphi$ therein. Combining these two ingredients, for a fixed $N\in\N$, we obtain the expansion: 
  	\begin{equation}
  		\label{eq:expansion2}
  		\vol_\varphi(\mathcal T_B(\ve)) = \sum_{n=0}^N \ve^n\left(\gamma_n+\beta_n\log\ve\right)+o(\ve^N),\qquad\text{as }\ve\to 0^+,
  	\end{equation}
  	for a collection of real numbers $\{\gamma_n,\beta_n\}_{n\in\N}$. Finally, combining \eqref{eq:expansion1} with  \eqref{eq:expansion2} applied for each ball $B_i$, $i=1,\ldots,k$, we obtain \eqref{eq:expansion_claim}. The characterization of the first three coefficients of \eqref{eq:expansion_claim} can be obtained as in \cref{prop:expansion_vol_around_char_pts}. 
  \end{proof}
  
    \begin{corollary}\label{cor:noncharacteristic_coefficients}
  	Under the hypotheses and notations of \cref{thm:expansion_weighted_vol_S}, assume that $\supp(\varphi)\cap\Sigma(S)=\emptyset$.
  	Then, for every $n\geq1$ it holds
  	\begin{equation}
  		\alpha_n(\varphi)
  		=\frac1{n!}\int_S\varphi\,\sigma_S^{(n-1)},
  		\qquad
  		\beta_n(\varphi)=0.
  	\end{equation}
  \end{corollary}
  \begin{proof}
      Observe that $\supp(\varphi)$ is a compact set that does not contain characteristic points. Hence, we may choose $r_0>0$ so small that $\overline{B(r_0)}\cap\supp(\varphi)=\emptyset$. The result then follows from \eqref{eq:expansion1}.
  \end{proof}

  \subsection{The finite part of a curvature form}

  In light of the non-integrable behavior of the high-order curvature forms at characteristic points (cf. \cref{cor:curv_forms}), we now introduce a notion of regularization of their (potentially) divergent integrals. The regularization procedure is performed using the cut time. 

  \begin{defi}[Finite part of a curvature form]
  \label{def:finite_part_boring}  
      Let $S=\partial\Omega\subset\hei$ be a smooth embedded surface and let $\sigma_S^{(n)}$ be its $n$-th curvature form, for $n\in\N$. The \emph{finite part} of $\sigma_S^{(n)}$ is the functional $\mathrm{Fp}(\sigma_S^{(n)})$ on $C^\infty_c(S)$ defined by
    \begin{equation}
        \langle\mathrm{Fp}(\sigma_S^{(n)}),\varphi\rangle :=\mathrm{Fp}_{\delta\to 0^+}\left(\int_{\{t_\cut>\delta\}}\varphi\,\sigma^{(n)}_S\right),\qquad\forall\,\varphi\in C^\infty_c(S).
    \end{equation}
  \end{defi}  

The next lemma shows that the finite part of a curvature form is well-defined and defines a distribution on $S$, as long as $\Sigma(S)$ contains only geometrically stable points. 

\begin{lemma}\label{lem:finite_part_boring}
    Assume that $\Sigma(S)$ consists of geometrically stable characteristic points. Then, for every $n\in\N$ and every $\varphi\in C^\infty_c(S)$, the function
    \begin{equation}
    \label{eq:higgs_boson}
    \delta\mapsto \int_{\{t_\cut>\delta\}}\varphi\,\sigma^{(n)}_S
    \end{equation}
    has an expansion of the form \eqref{eq:finite-part-definition} as $\delta\to 0^+$. Moreover, for every $n\geq 2$, the finite part of a curvature form is a distribution on $C^\infty_c(S)$ of order at most $n-2$. Finally, if $S$ is homogeneous, the results hold requiring only non-degenerate characteristic points.
\end{lemma}

\begin{proof}
    For $n\leq 2$, \cref{cor:curv_forms} gives local integrability of
$\sigma_S^{(n)}$, hence, given $\varphi\in C^\infty_c(S)$,  
    \[
    \exists\lim_{\delta\to 0^+} \int_{\{t_\cut>\delta\}} \varphi\,\sigma_S^{(n)} = \int_S \varphi\,\sigma_S^{(n)}. 
    \]
    Hence, taking into account \cref{rmk:anti_matter}, the finite part of $\sigma_S^{(n)}$ is well-defined and coincides with the evaluation of the integral. For $n\geq 3$, let $0<\delta<r_0$ be small enough for \cref{lem:rho_0} to hold (cf.\,\cref{rmk:hom_rho_0} for the homogeneous case). Then 
    \begin{equation}
    \label{eq:sum_of_integralz}
        \int_{\{t_\cut>\delta\}}\varphi\,\sigma^{(n)}_S= \int_{\{\delta<t_\cut<r_0\}}\varphi\,\sigma^{(n)}_S+\int_{\{t_\cut\geq r_0\}}\varphi\,\sigma^{(n)}_S.
    \end{equation}
    The second integral of \eqref{eq:sum_of_integralz} is independent of $\delta$ so it contributes with a constant to the asymptotic expansion of $\int_{\{t_\cut>\delta\}}\varphi\,\sigma^{(n)}_S$. We now deal with the first integral of \eqref{eq:sum_of_integralz}. By \cref{lem:rho_0}, by additivity of the integral, the latter splits into $|\Sigma(S)\cap\supp\varphi|$ summands of the form 
    \begin{equation}
    \label{eq:suffering}
        \delta\mapsto\int_{\{\delta<t_\cut<r_0\}\cap B}\varphi \,\sigma_S^{(n)},
    \end{equation}
    where $B$ is one of the disjoint balls of \cref{lem:rho_0}. Therefore, it is sufficient to show that the finite part of one of such terms is well-defined and a distribution on $C^\infty_c(S)$. To deal with \eqref{eq:suffering}, we use the coordinates $(r,\theta)$ of \cref{lem:rho_0}, in which $t_\cut(r,\theta)=r$. \cref{cor:curv_forms1} then gives 
    \begin{equation}
    \label{eq:the_one_summand?}
          \int_{\{\delta<t_\cut<r_0\}\cap B}\varphi\,\sigma^{(n)}_S=\int_\delta^{r_0}\int_{0}^{2\pi}\varphi(r,\theta)r^{2-n}\phi_n(r,\theta)d\theta dr=\int_\delta^{r_0}r^{2-n}a_{\varphi,n}(r)dr,
    \end{equation}
    where we have denoted 
    \begin{equation}\label{eq:def_an_varphi}
  	a_{\varphi,n}(r):=\int_{0}^{2\pi}\varphi(r,\theta)\phi_n(r,\theta)d\theta.
  \end{equation}
    Now, the function $a_{\varphi,n}$ is smooth at $r=0$, thus it admits a Taylor expansion, up to any order $N\in\N$, of the form: 
    \begin{equation}\label{eq:chatgpt_is_useless}
        a_{\varphi,n}(r)=\sum_{k=0}^{N}\frac{a_{\varphi,n}^{(k)}(0)}{k!}r^k + r^{N+1}R_{\varphi,N+1}(r),\qquad\text{as }r\to 0^+,
    \end{equation}
    where $R_{\varphi,N+1}$ is a smooth function on $\R^+$. Adding and subtracting to $a_{\varphi,n}$ in \eqref{eq:the_one_summand?} its Taylor polynomial up to order $n-3$, we obtain: 
    \begin{equation}
    \label{eq:the_two_summands}
        \int_\delta^{r_0}r^{2-n}a_{\varphi,n}(r)dr = \int_\delta^{r_0}R_{\varphi,n-2}(r)dr+\sum_{k=0}^{n-3}\int_{\delta}^{r_0}\frac{a_{\varphi,n}^{(k)}(0)}{k!}r^{2-n+k}dr.
    \end{equation}
    Integrating the second summand in \eqref{eq:the_two_summands}, term by term, yields 
    \begin{equation}
        \sum_{k=0}^{n-3}\int_{\delta}^{r_0}\frac{a_{\varphi,n}^{(k)}(0)}{k!}r^{2-n+k}dr=\sum_{k=0}^{n-4}\frac{a_{\varphi,n}^{(k)}(0)(r_0^{3-n+k}-\delta^{3-n+k})}{k!(3-n+k)}+\frac{a_{\varphi,n}^{(n-3)}(0)}{(n-3)!}\log\left(\frac{r_0}{\delta}\right).
    \end{equation}
    Concerning the first summand of \eqref{eq:the_two_summands}, by smoothness of $R_{\varphi,n-2}(r)$, we have 
    \begin{equation}
        \int_\delta^{r_0}R_{\varphi,n-2}(r)dr=\int_0^{r_0}R_{\varphi,n-2}(r)dr+o(1).
    \end{equation}
    Therefore, the function \eqref{eq:suffering} has an expansion of the kind \eqref{eq:finite-part-definition}, and its finite part, denoted by 
    \begin{equation}
         \langle T_B,\varphi\rangle:=\int_0^{r_0}R_{\varphi,n-2}(r)dr+\sum_{k=0}^{n-4}\frac{a_{\varphi,n}^{(k)}(0)r_0^{3-n+k}}{k!(3-n+k)}+\frac{a_{\varphi,n}^{(n-3)}(0)}{(n-3)!}\log r_0,
    \end{equation}
    defines a distribution on $C^\infty_c(S)$. Indeed, for every $k$, $a^{(k)}_{\varphi,n}$ (cf.\ \eqref{eq:def_an_varphi}) depends linearly and continuously on $\varphi$ by \cref{lem:distribution_T}, hence 
     \begin{equation}
         C^\infty_c(S)\ni \varphi\mapsto \sum_{k=0}^{n-4}\frac{a_{\varphi,n}^{(k)}(0)r_0^{3-n+k}}{k!(3-n+k)}+\frac{a_{\varphi,n}^{(n-3)}(0)}{(n-3)!}\log r_0,
     \end{equation}
     is a distribution of order at most $n-3$. On the other hand, by Lagrange remainder theorem one deduces that the mapping 
    \begin{equation}
    \label{eq:assignment_remainder}
        C^\infty_c(S)\ni \varphi\mapsto \int_0^{ r_0}R_{\varphi,n-2}(r)dr=\frac{1}{(n-3)!}\int_0^{ r_0}\int_0^1\int_0^{2\pi} (1-t)^{n-3} \partial^{n-2}_s\left(\varphi(s,\theta)\phi_{n}(s,\theta)\right)|_{s=tr}d\theta dtdr.
    \end{equation}
    is linear. While arguing as in the proof of \cref{lem:distribution_T}, we infer the continuity of \eqref{eq:assignment_remainder} in the $C^{n-2}$-topology. Finally, $T_B$ defines a distribution of order at most $n-2$, as claimed.
\end{proof}

  \subsection{Expression of the coefficients around a characteristic point}
  
  In this section, we give an explicit expression of the coefficients of \cref{thm:expansion_weighted_vol_S}, relating them to the finite part of the curvature forms. Recall that, given $p\in \Sigma(S)$ a geometrically stable characteristic point and $\varphi\in C_c^\infty(S)$, then, in coordinates $(r,\theta)$ of \cref{lem:rho_0}, the function $\Phi_\varphi(s,r)$ defined in \eqref{eq:FI(s,r)_varphi} 
  is smooth. More precisely, letting $\mathcal J_S(u,r,\theta)$ be the function of \cref{prop:iter_diver_singularities}, we can write
  \begin{equation}\label{eq:FI_varphi}
  	\begin{aligned}
  		\Phi_\varphi(s,r)&= r^3\int_0^{2\pi}\varphi(r,\theta)\int_0^s\mathcal J_S(u,r,\theta)\,dud\theta.
  	\end{aligned}
  \end{equation}
  Recall the notation \eqref{eq:def_an_varphi}:
  \begin{equation}
  	a_{\varphi,n-1}(r)=\int_{0}^{2\pi}\varphi(r,\theta)\phi_{n-1}(r,\theta)d\theta,\qquad \phi_n(r,\theta)=\partial_s^{n}\mathcal J_S(s,r,\theta)|_{s=0},
  \end{equation}
  where the function $\phi_n$ satisfies (cf.\,\eqref{eq:def_phi_n})
  \begin{equation}
      \sigma_S^{(n)}=r^{2-n}\phi_n(r,\theta)\,dr d\theta.
  \end{equation}
  For the next lemma, recall the notion of finite part of a function, given in \cref{def:finite_part_fun}, and of a curvature form, given in \cref{def:finite_part_boring}.
  \begin{lemma}\label{lem:string_theory}
   	Under the hypotheses of \cref{thm:expansion_weighted_vol_S}, assume that $\supp(\varphi)\subset B$, where $B$ is one of the balls of \cref{lem:rho_0}. Then, for every $n\in\N$, $n\geq 4$, in the coordinates of \cref{lem:rho_0}, it holds
   	\begin{equation}\label{eq:alpha_n_varphi}
   			\alpha_n(\varphi)=\frac1{n!}\bigl\langle\Fp(\sigma_S^{(n-1)}),\varphi\bigr\rangle
   			+T^B_n(\varphi),\qquad 
   				\beta_n(\varphi)=-\frac{a_{\varphi,n-1}^{(n-4)}(0)}{n!(n-4)!}
   	\end{equation}
   where $T^B_n$ is a distribution on $B$ of order $n-4$ and has the explicit expression
   		\begin{equation}\label{eq:local-profile-h}
   				T^B_n(\varphi)
   				=\frac{h_{\varphi,n}(1)}{n}
   				+\Fp_{\delta\to0^+}\left(
   				\int_\delta^1\frac{h_{\varphi,n}(s)}{s^{n+1}}\,ds\right),\qquad h_{\varphi,n}(s)
   				:=\frac{\partial_r^{n-1}\Phi_\varphi(s,0)}{(n-1)!}.
   		\end{equation}
    \end{lemma}
  \begin{proof}
  	Fix $\varphi\in C_c^\infty(S)$, with $\supp(\varphi)\subset B$, and $n\geq4$. We recall the splitting \eqref{eq:structure_of_integral_varphi} and compute the contributions of the two integrals appearing therein separately. The function $\Phi_\varphi(s,r)$ is $r^3$ times a smooth
  	function. Therefore, for $N\geq 4$, a standard Taylor expansion of $\Phi_\varphi$ and integration give
  	\begin{equation}\label{eq:contr_1}
  		\int_0^\ve\Phi_\varphi(1,r)\,dr=\sum_{n=4}^N
  		\frac{\ve^n}{n!}
  		\partial_r^{n-1}\Phi_\varphi(1,0)
  		+O(\ve^{N+1})=\sum_{n=4}^N
  		\ve^n\frac{h_{\varphi,n}(1)}{n}
  		+O(\ve^{N+1}),\qquad\text{as }\ve\to0^+.
  	\end{equation}
  	For the second integral in \eqref{eq:structure_of_integral_varphi}, we apply \cref{l:integrato} with
  	$G=\Phi_\varphi$. Formula \eqref{eq:abstract_coefficients} then gives
  	\begin{equation}\label{eq.contr2}
  		\begin{aligned}
  			\gamma_n
  			=\Fp_{\delta\to0^+}\left(\frac1{n!}
  			\int_\delta^{ r_0}
  			\frac{\partial_s^n\Phi_\varphi(0,r)}{r^n}\,dr
  			+\int_\delta^1\frac{h_{\varphi,n}(s)}{s^{n+1}}\,ds\right),\qquad 
  			\beta_n
  			=-\frac{
  				\partial_r^{n-1}\partial_s^n\Phi_\varphi(0,0)}
  			{n!(n-1)!}.
  		\end{aligned}
  	\end{equation} 
  	Using \eqref{eq:def_an_varphi} we find, for every sufficiently small $\delta>0$,
  	\begin{equation}
  		\begin{aligned}
  			\int_\delta^{ r_0}
  			\frac{\partial_s^n\Phi_\varphi(0,r)}{r^n}\,dr=\int_\delta^{ r_0}\int_{0}^{2\pi}\varphi(r,\theta)r^{3-n}\phi_{n-1}(r,\theta)d\theta dr=\int_{\{t_\cut>\delta\}}
  			\varphi\,\sigma_S^{(n-1)},
  		\end{aligned}
  	\end{equation}
  	where the last equality uses $\supp\varphi\subset B$. Thus, by \cref{def:finite_part_boring}, we deduce that
  	\begin{equation}
  		\Fp_{\delta\to0^+}
  		\left(\int_\delta^{ r_0}
  		\frac{\partial_s^n\Phi_\varphi(0,r)}{r^n}\,dr\right)
  		=\langle\Fp(\sigma_S^{(n-1)}),\varphi\rangle.
  	\end{equation}
  	Adding the contribution of order $\ve^n$ of the two integrals in \eqref{eq:structure_of_integral_varphi}, which are, respectively, the coefficient of order $\ve^n$ in \eqref{eq:contr_1} and $\gamma_n$ in \eqref{eq.contr2}, we obtain
  	\begin{equation}
  		\begin{aligned}
  			\alpha_n(\varphi)=\frac1{n!}
  			\bigl\langle\Fp(\sigma_S^{(n-1)}),\varphi\bigr\rangle+\Fp_{\delta\to0^+}\left(
  			\int_\delta^1\frac{h_{\varphi,n}(s)}{s^{n+1}}\,ds\right)
  			+\frac{h_{\varphi,n}(1)}n,
  		\end{aligned}
  	\end{equation}
  	which is the asserted formula for $\alpha_n(\varphi)$. Finally, the only contribution to order $\ve^n\log(\ve)$ is given by $\beta_n$ in \eqref{eq.contr2}, which, using \eqref{eq:def_an_varphi}, rewrites as
  	\begin{equation}
  		\beta_n(\varphi)
  		=-\frac1{n!(n-1)!}
  		\left.\partial_r^{n-1}\bigl(r^3a_{\varphi,n-1}(r)\bigr)\right|_{r=0}=-\frac{a_{\varphi,n-1}^{(n-4)}(0)}{n!(n-4)!}.
    \end{equation}
    Hence, the proof is complete if we show that $T_n\in\D'(B)$ of order at most $n-4$. Note that, by definition of $h_{\varphi,n}$ and by \eqref{eq:FI_varphi}, there exist smooth functions $J_k(s,\theta)$ such that 
    \begin{equation}
        h_{\varphi,n}(s)=\frac{\partial_r^{n-1}\Phi_\varphi(s,0)}{(n-1)!}=\sum_{k=0}^{n-4}\int_{0}^{2\pi}\partial_r^k\varphi(0,\theta) J_k(s,\theta)d\theta,
    \end{equation}
    The upper index is $n-4$ because exactly three of the $n-1$ derivatives with respect to $r$ have to fall on the factor $r^3$ for the result to be nonzero. Therefore 
    \begin{equation}
    \begin{aligned}
        \Fp_{\delta\to0^+}\!\left(
   				\int_\delta^1\frac{h_{\varphi,n}(s)}{s^{n+1}}\,ds\right)&=\sum_{k=0}^{n-4}\int_{0}^{2\pi}\partial_r^k\varphi(0,\theta) \Fp_{\delta\to0^+}\!\left(\int_\delta^{1}\frac{J_k(s,\theta)}{s^{n+1}}ds\right)d\theta\\ 
                &=:\sum_{k=0}^{n-4}\int_{0}^{2\pi}\partial_r^k\varphi(0,\theta) \hat J_k(\theta)d\theta,
    \end{aligned}
    \end{equation}
    where the swapping of the finite part and the integral is justified by \cref{lem:finite-part-swap}. The latter sum defines a distribution of order $n-4$, by direct application of \cref{lem:distribution_T}. Finally, the fact that $\varphi\mapsto h_{\varphi,n}(1)$ is a distribution of order $n-4$ follows by computing $h_{\varphi,n}(1)$ using \eqref{eq:FI_varphi} and applying \cref{lem:distribution_T}.
  \end{proof}

\subsection{Curvature distributions. Proof of \cref{thm:intro-distributions}}

We are now in position to study in detail the structure of the coefficients of the expansion of $\vol_\varphi(\mathcal T_S(\ve))$. The first three coefficients of the expansion are given by the curvature measures $\sigma_S^{(n-1)}$, for $n=1,2,3$. For $n\geq 4$, the coefficients no longer define curvature measures, but rather \emph{curvature distributions}. For further details around distributions on manifolds, we refer to Appendix~\ref{app:distributions}. 
\begin{proof}[Proof of \cref{thm:intro-distributions}]
We firstly observe that $\vol_\varphi(\mathcal T_S(\ve))$ is linear with respect to $\varphi$. Since, by \cref{thm:expansion_weighted_vol_S}, $\ve\mapsto\vol_\varphi(\mathcal T_S(\ve))$ admits an expansion of the form \eqref{eq:log_laurent_expansion}, its coefficients are uniquely determined, by \cref{lem:log_expansion_intrinsic}. Consequently, we deduce that $\alpha_n(\cdot)$ and $\beta_n(\cdot)$ are linear functionals on $C_c^\infty(S)$. We now show that they are distributions of the claimed orders, starting from $\alpha_n(\cdot)$. By \cref{lem:rho_0} (cf.\,\cref{rmk:hom_rho_0} for the homogeneous case), there exists a collection of (possibly) countably many disjoint embedded balls $\{B_i\}_{i\in\N}$ centered at the characteristic points of $S$ (cover $S$ with countably many compact sets satisfying the hypothesis of the lemma, and apply the latter to each one). 
We define $U=S\setminus \Sigma(S)$. Then, $\{U,{\rm int}(B_i)\}_{i\in\N}$ is an open cover for $S$. By \cref{prop:Dprime_sheaf}, $\alpha_n$ is a distribution of order at most $n-3$ if we show that restrictions $\alpha_n|_U:=\alpha_n|_{C^\infty_c(U)}$, $\alpha_n|_{{\rm int}(B_i)}:=\alpha_n|_{C^\infty_c({\rm int}(B_i))}$ are local distributions of the same order. Now, by \cref{cor:noncharacteristic_coefficients}, observe that 
\[
    \alpha_n|_U(\varphi)= \frac1 {n!}\int_U\varphi \,\sigma_S^{(n-1)},\qquad\forall\, \varphi\in C^\infty_c(U),
\]
thus $\alpha_n|_U\in\D'(U)$ of order $0$, and in particular it is continuous in $C^{n-3}_c(U)$. We are left to check the claim for $\alpha_n|_{{\rm int}(B_i)}$, for every $i\in\N$. Fix $i\in\N$ and set $B:={\rm int}(B_i)$, then, by \cref{lem:string_theory}, we have
\[
\alpha_n|_{B}(\varphi) = \frac1{n!}\langle\Fp(\sigma_S^{(n-1)}),\varphi\rangle+T_n^i(\varphi),\qquad\forall\,\varphi\in C^\infty_c(B),
\]
where we recall that $T_n^i\in\D'(B)$ of order at most $n-4$, supported at the unique characteristic point of $B$. On the other hand, the finite part of a curvature form $\Fp(\sigma_S^{(n-1)})$ is a distribution on $S$ of order at most $n-3$, by \cref{lem:finite_part_boring}. Therefore, we conclude that $\alpha_n|_{B}\in\D'(B)$ of order $\leq n-3$ as claimed. The proof of item $(i)$ is then completed, setting 
\[
T_n(\varphi) := \sum_{i\in\N} T_n^i(\varphi)\qquad\forall\,\varphi\in C^\infty_c(S).
\]
Note that, for every fixed $\varphi\in C^\infty_c(S)$, the above series is a finite sum. Indeed, $\Sigma(S)$ consists of geometrically stable characteristic points, which must be isolated. In addition the order of $T_n$ is at most $n-4$ and $\supp(T_n)\subset \bigcup_{i\in\N} \supp(T_n^i)\subset\Sigma(S)$.
Similarly for $\beta_n(\cdot)$, combining \cref{cor:noncharacteristic_coefficients} and \cref{lem:string_theory}, we have 
\begin{equation}
    \beta_n|_U=0,\qquad \beta_n|_{B}(\varphi)=-\frac{1}{n!(n-4)!}\int_{0}^{2\pi}\partial_r^{n-4}\left(\varphi(r,\theta)\phi_{n-1}(r,\theta)\right)|_{r=0}d\theta.
\end{equation}
\cref{lem:distribution_T} and \cref{prop:Dprime_sheaf} then imply that $\beta_n(\cdot)$ is a distribution of order at most $n-4$ supported at the characteristic set.
\end{proof}

\subsection{More on the logarithmic coefficients} The coefficient $\beta_n(\cdot)$ is a distribution of order at most $n-4$, supported on the characteristic set. For $\varphi\in C_c^\infty(S)$ whose support contains a single characteristic point $p$, \cref{lem:string_theory} allows one to compute $\beta_n(\varphi)$ explicitly as
\begin{equation}\label{eq:beta_n_polar_explain}
    \beta_n(\varphi)
    =-\frac{1}{n!(n-4)!}
    \left.\partial_r^{n-4}
    \left(
        \int_0^{2\pi}
        \varphi(r,\theta)\phi_{n-1}(r,\theta)\,d\theta
    \right)\right|_{r=0},
\end{equation}
where $(r,\theta)$ are positive polar coordinates centered at $p$ such that $t_\cut(r,\theta)=r$. Since iterated derivatives with respect to $r=t_\cut$ appear, one may think that $\beta_n$ depends on high order derivatives of the cut time function. However, remarkably enough, the value of the right hand side of \eqref{eq:beta_n_polar_explain} is independent of the polar coordinates used to compute it. 
\begin{prop}
\label{prop:beta_n_polar}
Under the hypotheses of \cref{thm:intro-distributions}, let
$p\in\Sigma(S)$ and let $U\subset S$ be an open set such that $U\cap \Sigma(S)=\{p\}$. 
For every systems of positive polar coordinates $(r,\theta)$ centered at $p$ we have 
\begin{equation}\label{eq:beta_n_polar}
    \beta_n(\varphi)
    =-\frac{1}{n!(n-4)!}
    \left.\partial_r^{n-4}
    \left(
        \int_0^{2\pi}
        \varphi(r,\theta)\phi_{n-1}(r,\theta)\,d\theta
    \right)\right|_{r=0},\qquad \forall\,\varphi\in C_c^\infty(U),\quad n\geq 4,
\end{equation}
where $\phi_n$ is defined in \cref{cor:curv_forms1}. In particular the right-hand side is independent of the choice of the positive polar
coordinates centered at $p$. 
\end{prop}
\begin{proof}
Fix $n\geq4$ and set $m=n-4$. In the chosen coordinates, write
\begin{equation}
    \varphi\,\sigma_S^{(n-1)}
    =r^{-m-1}f(r,\theta)\,dr\wedge d\theta,
    \qquad
    f(r,\theta)=\varphi(r,\theta)\phi_{n-1}(r,\theta).
\end{equation}
A Taylor expansion at $r=0$ gives
\begin{equation}
    f(r,\theta)
    =\sum_{j=0}^{m}r^jf_j(\theta)+O(r^{m+1}),
    \qquad
    f_j(\theta)=\frac1{j!}\partial_r^jf(0,\theta).
\end{equation}
We define $L_r(\varphi)$ as the coefficient of $\log\delta$ in
the expansion as $\delta\to 0^+$ of the function
\begin{equation}
    I_r(\delta)
    :=\int_{\{r>\delta\}}\varphi\,\sigma_S^{(n-1)}.
\end{equation}
Integrating the Taylor expansion term by term gives
\begin{equation}
    I_r(\delta)
    =\sum_{j=0}^{m-1}
    \frac{\delta^{j-m}}{m-j}
    \int_0^{2\pi}f_j(\theta)\,d\theta
    -\log\delta\int_0^{2\pi}f_m(\theta)\,d\theta
    +O(1),
    \qquad \delta\to0^+,
\end{equation}
where the sum is empty when $m=0$. In particular, the coefficient
of $\log\delta$ is
\begin{equation}\label{eq:curvature_log_residue}
    L_r(\varphi)
    =-\int_0^{2\pi}f_m(\theta)\,d\theta=-\frac{1}{(n-4)!}
    \left.\partial_r^{n-4}
    \left(
        \int_0^{2\pi}
        \varphi(r,\theta)\phi_{n-1}(r,\theta)\,d\theta
    \right)\right|_{r=0}.
\end{equation}
By equation \eqref{eq:alpha_n_varphi} of \cref{lem:string_theory}, if $(r,\theta)$ are the polar coordinates of \cref{lem:rho_0}, then $L_r(\varphi)$ coincides with $n!\beta_n(\varphi)$. It remains to show that $L_r(\varphi)$ is independent of the choice of polar coordinates. Let $(\widetilde r,\widetilde\theta)$ be another positively oriented smooth polar coordinate system. The radial component of the
coordinate change has the form
\begin{equation}
    \widetilde r=r\,a(r,\theta),
    \qquad a(0,\theta)>0,
\end{equation}
with $a$ smooth up to $r=0$. By the implicit function theorem,
the curve $\widetilde r=\delta$ can be written in the old coordinates $(r,\theta)$ as
$r=\delta b(\delta,\theta)$, where $b$ is smooth and positive.
Thus, for sufficiently small $\delta$,
\begin{equation}
    I_r(\delta)-I_{\widetilde r}(\delta)
    =\int_0^{2\pi}\int_\delta^{\delta b(\delta,\theta)}
    r^{-m-1}f(r,\theta)\,dr\,d\theta.
\end{equation}
For each Taylor term, the radial integral is
\begin{equation}
    \int_\delta^{\delta b(\delta,\theta)}
    r^{j-m-1}\,dr
    =
    \begin{cases}
        \displaystyle
        \delta^{j-m}
        \frac{b(\delta,\theta)^{j-m}-1}{j-m},
        & j<m,\\[6pt]
        \log b(\delta,\theta),
        & j=m.
    \end{cases}
\end{equation}
Since $b$ is smooth and positive, these expressions have asymptotic expansions containing only integer powers of $\delta$. The Taylor remainder contributes $O(\delta)$. Hence the difference contains no logarithmic term, proving $L_r(\varphi)=L_{\widetilde r}(\varphi)$.
\end{proof}

\subsection{The local coefficients of $\ve^4\log\ve$ and $\ve^4$}
\label{sec:beta4}
The lowest-order correction beyond order $\ve^3$ in the expansion of \cref{thm:intro-distributions} is given by
\begin{equation}
    \beta_4(\varphi)\ve^4\log\ve+\left(\frac{1}{24}\langle\Fp(\sigma^{(3)}),\varphi\rangle+T_4(\varphi)\right)\ve^4,
\end{equation}
where the distributions $\beta_4$ and $T_4$ are of order zero and supported on the characteristic set. We explicitly compute these distributions showing that, at a given characteristic point $p$, they only depend on the parameters $a,b$ of the local homogeneous approximation.

\begin{prop}\label{lem:beta4_dirac}
Under the hypotheses of \cref{thm:intro-distributions}, let $p\in\Sigma(S)$ and let $U\subset S$ be an open set such that $U\cap \Sigma(S)=\{p\}$. Then 
	\begin{equation}
		\beta_4|_U=\frac{\pi(a+b)}{12|1+ab|}\,\delta_p,\qquad T_4|_U=\tau_{a,b}\delta_p,
	\end{equation}
    where $\tau_{a,b}\in\mathbb R$ depends only on the parameters $a,b$ of the homogeneous approximation. In particular
    \begin{equation}
        \tau_{0,0}=\frac{1-\pi\Si(\pi)}{12}.
    \end{equation}
\end{prop}
\begin{proof}
	We first prove the formula for $\beta_4$. Let $\varphi\in C_c^\infty(U)$. By \cref{prop:beta_n_polar}, given positive polar coordinates centered at $p$ 
	\begin{equation}\label{eq:integral_beta}
		\beta_4(\varphi)
		=-\frac{\varphi(p)}{24}\int_0^{2\pi}\phi_3(0,\theta)\,d\theta.
	\end{equation}
	We may compute this integral in the positive adapted polar coordinates $(\rho,\mathrm{sign}(1+ab)\vartheta)$
	of \eqref{eq:x-y_polar}. Note that the factor $\mathrm{sign}(1+ab)$ ensures that these coordinates are positively oriented, see \cref{rmk:hom_or}. By \cref{cor:curv_forms1}, and in particular by equation \eqref{eq:in_hom_coords}, and by \cref{ex:hom_curv_forms} we obtain
	\begin{equation}
		\begin{aligned}
			\int_0^{2\pi}\pm\phi_3(0,\pm\vartheta)\,d\vartheta&=\int_0^{2\pi}\pm\phi_3^{(a,b)}(1,\pm\vartheta)\,d\vartheta\\
            &=-\frac{1}{|1+ab|}
			\int_0^{2\pi}
			\bigl(a+b+3(a-b)\cos(2\vartheta)\bigr)\,d\vartheta=-\frac{2\pi(a+b)}{|1+ab|}.
		\end{aligned}
	\end{equation}
    
    We now move to the formula for $T_4$. It is sufficient to compute $T_4|_B$, where $B$ is the ball of \cref{lem:rho_0} (cf.\,\cref{rmk:hom_rho_0} for the homogeneous case) centered at $p$. Given $\varphi\in C_c^\infty(B)$, \cref{lem:string_theory} gives
\begin{equation}
	T_4|_B(\varphi)=T^B_4(\varphi)
	=\frac{h_{\varphi,4}(1)}{4}
	+\Fp_{\delta\to0^+}\left(
	\int_\delta^1\frac{h_{\varphi,4}(s)}{s^{5}}\,ds\right),\qquad h_{\varphi,4}(s)
	=\frac{\partial_r^{3}\Phi_\varphi(s,0)}{3!},
\end{equation}
where $\Phi_\varphi(s,r)$ is defined in \eqref{eq:FI_varphi}, i.e.
\begin{equation}
	\Phi_\varphi(s,r)=r^3\int_0^{2\pi}\varphi(r,\theta)\int_0^{s} \mathcal J_S(u,r,\theta)du\,d\theta,
\end{equation}
whereas $(r,\theta)$ are the polar coordinates of \cref{lem:rho_0}, and $J_{S}(t,r,\theta)=r^2\mathcal J_S(t/r,r,\theta)$ is the Jacobian determinant of $\exp^\perp_S(t,r,\theta)$ (cf. \cref{prop:iter_diver_singularities}). The function $h_{\varphi,4}$ can then be explicitly computed:
\begin{equation}\label{eq:h_varphi,4}
	h_{\varphi,4}(s)
	=\varphi(p)\int_0^{2\pi}\int_0^s
	\mathcal J_S(u,0,\theta)\,du\,d\theta.
\end{equation}
 Let now $\widetilde{J}_S(t,\rho,\vartheta)$ be the Jacobian determinant of the normal exponential map computed with respect to the polar coordinates $(\rho,\vartheta)$ of \eqref{eq:x-y_polar}. By \cref{rmk:in_hom_coords} and \cref{thm:intro-cut-time-stability}, we have
\begin{equation}
	\widetilde{J}_S(t,\rho,\vartheta)
	=\rho^2\left(
	J_{S_{a,b}}\left(\frac t\rho,1,\vartheta\right)
	+\rho R_J\left(\frac t\rho,\rho,\vartheta\right)
	\right),
	\qquad
	t_\cut(\rho,\vartheta)=c_{a,b}\rho+O(\rho^2),
\end{equation}
	with smooth remainders. The change to cut time coordinates
	$(r,\theta)=(t_\cut(\rho,\vartheta),\pm\vartheta)$,  with $$\pm=\mathrm{sign}(1+ab),$$ has a smooth inverse
	$(\rho,\vartheta)=(q(r,\theta),\pm\theta)$, with $q(r,\theta)=r/c_{a,b}+O(r^2)$.
	Its Jacobian determinant is $\pm\partial_rq$, since the angular
	coordinate is unchanged. Hence
	\[
	J_S(t,r,\theta)
	=\pm\widetilde{J}_S(t,q(r,\theta),\pm\theta)\,\partial_rq(r,\theta).
	\]
	Evaluating at $t=ur$ and dividing by $r^2$ gives
	\[
	\mathcal J_S(u,r,\theta)
	=\pm\left(\frac qr\right)^2\partial_rq(r,\theta)
	\left[
	J_{S_{a,b}}\left(\frac{ur}{q(r,\theta)},1,\pm\theta\right)
	+q(r,\theta)R_J\left(\frac{ur}{q(r,\theta)},q(r,\theta),\pm\theta\right)
	\right].
	\]
	Since $q(r,\theta)/r$ extends smoothly with value $c_{a,b}^{-1}$ at $r=0$, and $\partial_rq(0,\theta)=c_{a,b}^{-1}$
	we obtain
	\begin{equation}\label{eq:normalized_jacobian_cut_limit}
		\mathcal J_S(u,0,\theta)=\pm c_{a,b}^{-3}J_{S_{a,b}}(c_{a,b}u,1,\pm\theta).
	\end{equation}
	Substituting the latter result in \eqref{eq:h_varphi,4} yields
	\begin{equation}
		h_{\varphi,4}(s)=\varphi(p)h_4^{a,b}(s),\qquad h_4^{a,b}(s):=\pm c_{a,b}^{-3}\int_0^{2\pi}\int_0^s
	J_{S_{a,b}}(c_{a,b}u,1,\pm\theta)\,du\,d\theta
	\end{equation}
	Finally we have 
	\begin{equation}
		T^B_4(\varphi)
		=\varphi(p)\left[\frac{h_{4}^{a,b}(1)}{4}
		+\Fp_{\delta\to0^+}\left(
		\int_\delta^1\frac{h_{4}^{a,b}(s)}{s^{5}}\,ds\right)\right]=:\varphi(p)\tau_{a,b}.
	\end{equation}
	We see from its definition that the quantity $\tau_{a,b}$ is a constant depending on $a,b$ alone.
    
    For $a=b=0$ we have $c_{0,0}=\pi/2$, $\mathrm{sign}(1+ab)=+$, and 
    \begin{equation}
        J_{S_{0,0}}=\frac{r^2}4\left[1+\cos\left(\frac{2t}{r}\right)+\left(\frac{t}{r}\right)\sin\left(\frac{2t}{r}\right)\right],
    \end{equation}
    therefore we may compute $h_4^{0,0}(s)$ explicitly as
    \begin{equation}
        h_4^{0,0}(s)=\frac{4}{\pi^2}\int_0^s\left[1+\cos\left(\pi u\right)+\left(\frac{\pi u}{2}\right)\sin\left(\pi u\right)\right]du=\frac{4}{\pi^2}\left[s-\frac{s}{2}\cos(\pi s)+\frac{3}{2\pi}\sin(\pi s)\right].
    \end{equation}
In particular, using the shorthand $h:=h_4^{0,0}$,
\begin{equation}
    \frac{h(1)}{4}=\frac{3}{2\pi^2},
\qquad
h(s)=\frac{8}{\pi^2}s-\frac{\pi^2}{30}s^5+O(s^7).
\end{equation}
Thus the finite part is
\[
\operatorname{Fp}_{\delta\to 0^+}\int_\delta^1\frac{h(s)}{s^5}\,ds
=
\lim_{\delta\to0^+}
\left[
\int_\delta^1\frac{h(s)}{s^5}\,ds
-\frac{8}{3\pi^2\delta^3}
\right]=-\frac{3}{2\pi^2}+\frac1{12}
-\frac{\pi}{12}\Si(\pi),
\]
Therefore 
\begin{equation}
\tau_{0,0}=\frac{h(1)}{4}+\operatorname{Fp}_{\delta\to 0^+}\int_\delta^1\frac{h(s)}{s^5}\,ds
=\frac{1-\pi\Si(\pi)}
{12}.
\end{equation}
\end{proof}

\part{Cut time analysis}
\label{part:cut-time}

This part of the paper is entirely devoted to the study of the cut time function, both in the homogeneous case and near geometrically stable characteristic points (see \cref{def:geometrically_stable_char_pts}). We first compute the cut time for a homogeneous surface. We then apply the theory of Legendrian singularities to show that, the cut time near a geometrically stable characteristic point can be obtained as a perturbation of the one of the homogeneous model. Legendrianity needs to be taken into account because the singularities of the normal exponential map at the cut locus are not stable in the class of smooth mappings.
\begin{notation}
	Throughout this part, given a map $f:M\to N$ between smooth manifolds $M$ and $N$, and given $U\subset M$, we set
	\begin{equation}
		\begin{aligned}
			\mathrm{Doub}(f,U)&=\{x\in U\,:\,\exists\, x'\in U,\,x'\neq x,\,f(x)=f(x')\}.
		\end{aligned}
	\end{equation}
	Whenever $U=M$ we simply write $\mathrm{Doub}(f,M)=\mathrm{Doub}(f)$. 
\end{notation}

\section{The extended exponential map}
\label{sec:eem}
The first step is to characterize the cut time of a smooth oriented embedded surface $S=\partial\Omega$, in terms of the double points of the  \emph{extended exponential map}, defined as 
\begin{equation}
    \mathscr F:\mathbb R^+\times S^\times\to\hei\times\mathbb R,\qquad \mathscr F(t,q):=(\exp^\perp_S(t,q),\ell(t,q)),\qquad \ell(t,q):=t.
\end{equation}
Note that $\ell(t,q)$ is the length of the normal geodesic from $S$ starting at $q$ at time $t$. 
\begin{remark}\label{rmk:lege}
    For $(t,q)\in \mathbb R^+\times S^\times$ let $\lambda(t,q):=e^{t\vec{H}}(\lambda^\perp(q))$. Then, Pontryagin maximum principle implies that 
    \begin{equation}\label{eq:lege}
        d\ell(t,q)-\langle\lambda(t,q),d\exp_S^\perp(t,q)\rangle=0.
    \end{equation}
    Therefore $\ker d\ell\supset \ker d\exp^\perp_S$. Thus, critical points of $\mathscr F=(\exp^\perp_S,\ell)$ coincide with those of $\exp^\perp_S$. 
\end{remark}
The next lemma shows that the double point locus of $\mathscr F$ contains all relevant information concerning the cut time: a cut point is either a double point or an accumulation of double points. 

\begin{lemma}\label{lem:foc_after_max}
    Let $S=\partial\Omega\subset \hei$ be a smooth, oriented, embedded surface and let $q\in S^\times$. If $t_\cut(q)<+\infty$, then 
    \[
    (t_\cut(q),q)\in \overline{\mathrm{Doub}(\mathscr F)}.
    \]
\end{lemma}
\begin{proof}
Put $t_0=t_{\cut}(q)$ and $x=\exp_S^\perp(t_0,q)$.
If $x$ is reached by more than one minimizing geodesic, then
$(t_0,q)\in\mathrm{Doub}(\mathscr F)$, and there is nothing to prove. We therefore assume that the minimizing geodesic reaching $x$ is unique. We first show that every neighborhood of $x$ contains a point reached by more than one minimizing geodesic. Suppose otherwise that each point $y$ of an open neighborhood $V$ of $x$ has a unique minimizer. Let $\beta := {\sf p}|_{V}:V\to S$ be the restriction to $V$ of the metric projection on $S$, which assigns to a point $y\in V$ the base point of the (unique) minimizing geodesic from $S$ reaching it. Since $\beta$ is single-valued by the assumptions, a standard argument shows that it is a continuous map. Up to shrinking $V$, we may also assume that $V\cap S=\varnothing$ and
$\beta(V)\subset S^\times$. Define
\begin{equation}
    \Lambda:V\to(0,\infty)\times S^\times,\qquad
\Lambda(y)=(\mathrm{dist}_S(y),\beta(y)).
\end{equation}
The function $\Lambda$ is continuous, by construction, and injective, since $\exp_S^\perp\circ\Lambda=\operatorname{Id}_V$. Then, Brouwer's invariance of domain implies that $\Lambda(V)$ is open and thus, since
$\Lambda(x)=(t_0,q)$, there exists $\eta>0$ such that
$(t_0+\eta,q)\in\Lambda(V)$. Therefore, there exists $y_\eta\in V$ such that $\Lambda(y_\eta)=(t_0+\eta,q)$. Applying $\exp^\perp_S$ to both sides of the latter equality, and recalling that $\exp_S^\perp\circ\Lambda=\operatorname{Id}_V$, we find 
\begin{equation}
    y_\eta=\exp_S^\perp(t_0+\eta,q),\qquad
\mathrm{dist}_S(y_\eta)=t_0+\eta.
\end{equation}
Thus, the normal geodesic from $q$ is minimizing beyond $t_0$, a contradiction. This shows that there is a sequence $\{x_n\}_{n\in\N}\subset\hei$ such that $x_n\to x$ and each $x_n$ is reached by at least two minimizing geodesics from $S$. Let $q_n\in{\sf p}(x_n)$ be the base point of one of them and set $t_n=\mathrm{dist}_S(x_n)$. Then, by construction, the pair
$(t_n,q_n)$ belongs to $\mathrm{Doub}(\mathscr F)$. Moreover, by the continuity of $\mathrm{dist}_S$, $t_n=\mathrm{dist}_S(x_n)\to \mathrm{dist}_S(x)=t_0$, while the
compactness of minimizers, together with uniqueness at $x$, gives
$q_n\to q$. Therefore $(t_n,q_n)\to(t_0,q)$, which proves the assertion.
\end{proof}
\begin{remark}\label{rmk:foc_after_max}
    The proof of \cref{lem:foc_after_max} provides pairs of minimizing normal geodesics with the same endpoint and length $t_n$, based at $q_n\ne q_n'$, such that
    \begin{equation}
        (t_n,q_n)\to(t_\cut(q),q).
    \end{equation}
    Since both geodesics are minimizing and meet at the same length it holds $t_n=t_\cut(q_n)=t_\cut(q_n')$.
\end{remark}

We define the canonical projection map as 
\[
\pi:\mathbb R^+\times S^\times\to S^\times,\qquad\pi(t,q)=q.
\]
The next lemma characterizes the cut time at $p\in S$ as the infimum of the length function along the intersection $\overline{\mathrm{Doub}(\mathscr F)}\cap\pi^{-1}(p)$.

\begin{lemma}\label{thm:t_cut_front}
	Let $S=\partial\Omega\subset \hei$ be a smooth, oriented, embedded surface. Then, for every $p\in S^\times$,
	\begin{equation}\label{eq:t_cut_char}
		t_\cut(p)=\inf \left\{\ell(t,q)\,:\,(t,q)\in \overline{\mathrm{Doub}(\mathscr F)},\,\pi(t,q)=p\right\}.
	\end{equation}
\end{lemma}
\begin{proof}
    Let $T(p)$ denote the right hand side of \eqref{eq:t_cut_char}. Since $\ell(t,q)=t$ and $\pi(t,q)=q$ we have
    \begin{equation}\label{eq:T(p)_char}
        T(p)=\inf\left\{t>0\,:\, (t,p)\in\overline{\mathrm{Doub}(\mathscr F)}\right\}.
    \end{equation}
    Let $\{(t_n,p_n)\}\subset\mathrm{Doub}(\mathscr F)$ converging to $(T(p),p)$. Then, for every $n\in\N$, $t_\Max(p_n)\leq t_n$. Consequently, by \cref{thm:t_cut=min_t_max_t_foc}, $t_\cut(p_n)\leq t_n$. Passing to the limit in the latter inequality we find $t_\cut(p)\leq T(p)$. On the other hand, if $t_\cut(p)<+\infty$ then by \eqref{eq:T(p)_char} and \cref{lem:foc_after_max} we have $t_\cut(p)\geq T(p)$, while if $t_\cut(p)=+\infty$ the latter inequality is obvious.
\end{proof}
\subsection{Reparametrization of geodesics}\label{ssec:s}
Given a smooth surface $S=\partial\Omega\subset\hei$ and $p\in\Sigma(S)$, making use of the embedding of \cref{rmk:embedding}, we may assume that $p=(0,0,0)$ and that
\begin{equation}\label{eq:surface_S1}
	S=\left\{(x,y,z)\in\hei\,:\, z=\frac{1}{2}u(x,y)\right\},\qquad u\in C^\infty(\mathbb R^2). 
\end{equation}
The surface $S$ is oriented towards $\{z>u(x,y)/2\}$. As observed in \cref{ssec:exp_regularity}, the analysis of the normal
exponential map is simplified by introducing the parameter $s=h_0(q)t$. Indeed, although the map $(t,q)\mapsto \exp^\perp_S(t,q)$ is singular
at characteristic points, the reparametrized map
\begin{equation}
    E:\mathbb R^+\times S^\times\to\hei,
    \qquad
    E(s,q)=\exp^\perp_S\left(\frac{s}{h_0(q)},q\right),
\end{equation}
is well defined and admits a smooth extension through the latter, as proved in \cref{lem:smooth_E}.
We shall therefore study the normal exponential map using the
coordinates $(s,q)$. We retain the symbol $\pi$ for the canonical projection
\begin{equation}
    \pi:\mathbb R^+\times S^\times\to S^\times,
    \qquad
    \pi(s,q)=q.
\end{equation} 
In these coordinates, the length function reads
\begin{equation}\label{eq:length_s}
    \ell:\mathbb R^+\times S^\times\to\mathbb R,
    \qquad
    \ell(s,q)=\frac{s}{h_0\circ\pi(s,q)}=\frac{s}{h_0(q)}.
\end{equation}
Here and in what follows, we use the same symbol $\ell$ for the length function expressed in the new coordinates. Similarly, to avoid introducing additional notation, we retain the symbol $\mathscr F$ for the extended exponential map and write
\begin{equation}
    \mathscr F:\mathbb R^+\times S^\times\to\hei\times\mathbb R,
    \qquad
    \mathscr F(s,q)=\bigl(E(s,q),\ell(s,q)\bigr).
\end{equation}
\begin{remark}\label{rmk:leges}
    For all $q\in S^\times$, the curve $s\mapsto E(s,q)$ is a reparametrization of the normal geodesic from $S$ based at $q$, indeed $E(s,q)=\gamma_q(s/h_0(q))$. In particular, reparametrizing normal lift of $\gamma_q$, the conclusion of \cref{rmk:lege} remains valid: 
    \begin{equation}\label{eq:contact_id1}
        d\ell(s,q)=\langle\lambda(s,q),dE(s,q)\rangle,\qquad \lambda(s,q)=e^{(s/h_0(q))\vec H}(\lambda^\perp(q)),\qquad \forall\,(s,q)\in \mathbb R^+\times S^\times.
    \end{equation}
\end{remark}
\subsection{A change of coordinates}\label{ss:change} 
We record here a useful change of coordinates which simplifies the expression of $\mathscr F$. Recall that, given $S\subset\hei$ of the form \eqref{eq:surface_S1}, the defining function $u\in C^\infty(\mathbb R^2)$ has the asymptotic expansion 
\begin{equation}
	u(x,y)=\frac{1}{2}(ax^2+by^2)+o(x^2+y^2),\qquad \text{as $(x,y)\to 0$},
\end{equation}
for $a,b\in\R$, see \cref{lem:nice_f}. Moreover, from now on we assume that $p\in \Sigma(S)$ is non-degenerate, namely $ab+1\neq 0$. For the sake of simplifying the formulae of the current part of the paper, we introduce the following notational convention. 
\begin{notation}\label{not:maps} Let $\Phi\colon M\to N$ be a map between smooth manifolds, and let \[ \varphi\colon U\subset\mathbb{R}^m\to M, \qquad \psi\colon V\subset\mathbb{R}^n\to N \] be local parametrizations inducing coordinates $(x_1,\ldots,x_m)$ on $M$ and $(y_1,\ldots,y_n)$ on $N$, respectively. Assume that $\Phi(\varphi(U))\subset\psi(V)$. By a slight abuse of notation, we write \[ \Phi(x_1,\ldots,x_m) := \bigl(\psi^{-1}\circ\Phi\circ\varphi\bigr)(x_1,\ldots,x_m) \] for the coordinate representation of $\Phi$. Thus, the source coordinates are indicated by the arguments of $\Phi$, whereas the target chart is suppressed from the notation and will be specified whenever needed. \end{notation}
First, we identify $S$ with $\R^2$, by means of the diffeomorphism
\begin{equation}\label{eq:global_x,y}
\mathbb R^2\to S,\qquad (x,y)\mapsto (x,y,\tfrac{1}{2}u(x,y)).
\end{equation}
Then, on the domain of $E$, which with the above identification is $\R^+\times \R^2$, we introduce coordinates defined by the following diffeomorphism 
\begin{equation}
	\varphi:\R^+\times \R^2\to \R^+\times \R^2,\qquad \varphi(s,u,v)=(s,x(s,u,v),y(s,u,v)),
\end{equation}
where we have denoted 
\begin{equation}\label{eq:uv_xy}
	\begin{pmatrix}
		x(s,u,v)\\
		y(s,u,v)
	\end{pmatrix}=\frac{1}{1+ab}
	\begin{pmatrix} 
		b & -1 \\
		1 & a
	\end{pmatrix}
	\begin{pmatrix} 
		\cos(\tfrac{s}{2}) & \sin(\tfrac{s}{2}) \\
		-\sin(\tfrac{s}{2}) & \cos(\tfrac{s}{2})
	\end{pmatrix}
	\begin{pmatrix}
		u\\ v
	\end{pmatrix}.
\end{equation}
On the codomain of $E$, which is $\hei$, we introduce coordinates induced by the diffeomorphism 
\begin{equation}\label{eq:psi}
	\psi:\hei\to \hei,\qquad \psi(x,y,z)=(ax+y,by-x,8(z-(ax^2+by^2)/4)).
\end{equation}
The non-degeneracy condition $ab+1\neq 0$ ensures that both maps are diffeomorphisms. We will be studying the normal exponential map in these new coordinates, i.e. 
\begin{equation}
	\psi\circ E\circ\varphi:\R^+\times\R^2\to\hei.
\end{equation}
In accordance with our convention of \cref{not:maps} on coordinates changes, we will denote 
	\begin{equation}\label{eq:E?}
		E(s,u,v)=\psi\circ E\circ\varphi(s,u,v),
	\end{equation}
and similarly, for the extended exponential map, 
\begin{equation}
	\mathscr F(s,u,v)=(\psi\circ E\circ\varphi(s,u,v),\ell\circ\varphi(s,u,v)).
\end{equation}

\begin{remark}\label{rmk:pi}
We specify the canonical projection in coordinates $(s,u,v)$. According to \cref{not:maps},  
\[
\pi(s,u,v) = \frac{1}{1+ab}
	\begin{pmatrix} 
		b & -1 \\
		1 & a
	\end{pmatrix}\mathfrak R(\tfrac{s}2)
	\begin{pmatrix}
		u\\ v
	\end{pmatrix},
\]
where $\mathfrak R(\cdot)$ is the rotation matrix defined in \eqref{eq:rotation}. Additionally, letting $(\rho,\vartheta)$ be the polar coordinates on $S$ defined in \eqref{eq:x-y_polar}, if we perform, in the domain, the change of coordinates $(s,u,v)\mapsto (s,\rho,\phi)$, where
\begin{equation}\label{eq:rho_phi}
    (u,v)=(\rho\cos\phi,\rho\sin\phi),
\end{equation}
and, on the target, we use the coordinates $(\rho,\vartheta)$, the canonical projection rewrites as 
\begin{equation}\label{eq:projection}
\pi(s,\rho,\phi)=(\rho,\phi-\tfrac{s}2).
\end{equation}
Analogously, by \eqref{eq:length_s}, the length function has the expression 
\begin{equation}
    \ell(s,\rho,\phi)=\frac{s}{h_0\circ\pi(s,\rho,\phi)}=\frac{s}{h_0(\rho,\phi-\tfrac{s}{2})}.
\end{equation}
Note that, by \cref{rmk:diver_h0}, $(s,\rho,\phi)\mapsto \ell(s,\rho,\phi)$ extends smoothly to $\rho=0$.
\end{remark}

\begin{remark}\label{r:cut_rho}
    According to \cref{thm:t_cut_front}, the cut time, in coordinates $(s,\rho,\phi)$ of \cref{rmk:pi}, can be characterized as 
\begin{equation}
\begin{aligned}
    t_\cut(\rho,\vartheta)&=\inf\{\ell(s,\rho',\phi)\,:\,(s,\rho',\phi)\in\overline{\mathrm{Doub}(\mathscr F)},\,\pi(s,\rho',\phi)=(\rho,\vartheta)\}.
\end{aligned}
\end{equation}
Since in these coordinates it holds $\pi(s,\rho,\phi)=(\rho,\phi-s/2)$ we have 
\begin{equation}
    t_\cut(\rho,\vartheta)=\inf\left\{\ell\left(s,\rho,\vartheta+\frac{s}{2}\right)\,:\,\left(s,\rho,\vartheta+\frac{s}{2}\right)\in\overline{\mathrm{Doub}(\mathscr F)}\right\}.
\end{equation}
\end{remark}
\section{The cut time of a homogeneous surface. Proof of Theorem~\ref{thm:intro-homogeneous-cut-time}}
\label{sec:cut-time-hom}

Given $a,b\in \mathbb R$, we now study the cut time function of the homogeneous surface $S_{a,b}$, cf.\ \cref{def:homogeneous_surf}. We assume that $a,b$ satisfy the non-degeneracy condition $ab+1\neq 0$, which we recall is equivalent to requiring that the origin is a non-degenerate characteristic point. Note that the map $(x,y,z)\mapsto (-y,x,z) $ is an isometry of $\hei$ mapping $S_{a,b}$ to $S_{b,a}$. For this reason, it is not restrictive to (and thus we will) assume that, in the remainder of the section, it holds
$$a\geq b.$$ 
We denote the normal exponential map \eqref{eq:E?} for the surface $S_{a,b}$ with $E_{a,b}$.
Similarly, the length function is denoted by $\ell_{a,b}$ and the extended exponential map by $\mathscr F_{a,b}=(E_{a,b},\ell_{a,b})$. 
\begin{lemma}\label{l:belle_coords}
	In the coordinates $(s,u,v)$ of \cref{ss:change}, the following holds
	\begin{equation}\label{eq:E_coords}
	 E_{a,b}(s,u,v)=\left(
	p_a(s)u,
	p_b(s)v,
		q_a(s)u^2+q_b(s)v^2
		\right),\qquad \ell_{a,b}(s,u,v)=\frac{s}{2}\sqrt{u^2+v^2},
	\end{equation}
	where, for $c\in \{a,b\}$ and $s\in \mathbb R^+$ we have defined 
	\begin{equation}
    \label{eq:def_p-q_c}
		p_c(s)=\cos\left(\tfrac{s}{2}\right)
		-c\sin\left(\tfrac{s}{2}\right),\qquad q_c(s)=s+\sin s-c(1-\cos s).
	\end{equation}
\end{lemma}
\begin{proof}
We first perform the change of coordinates in the domain. Let us introduce the notation
	\begin{equation}\label{eq:notation}
		A=
		\begin{pmatrix}
			a&1\\
			-1&b
		\end{pmatrix}, \qquad \xi=
	\begin{pmatrix}
	x\\y
\end{pmatrix},
\qquad
w=
\begin{pmatrix}
u\\v
\end{pmatrix}.
\end{equation}
	Then, by definition
	of $\varphi$, and in particular by \eqref{eq:uv_xy} we have 
	\begin{equation}\label{eq:xi_w}
		\xi=A^{-1}\mathfrak R(\tfrac{s}{2})w,
	\end{equation}
	where $\mathfrak R(\cdot)$ is the orthogonal matrix of \eqref{eq:rotation}. Consequently $A\xi=\mathfrak R(\tfrac{s}{2})w$, that is
	\begin{equation}
		\begin{pmatrix}
			ax+y\\
			by-x
		\end{pmatrix}
		=
		\mathfrak R(\tfrac{s}{2})
		\begin{pmatrix}
			u\\v
		\end{pmatrix}.
	\end{equation}
	Since $\mathfrak R(\cdot)$ is orthogonal, denoting $\rho=\sqrt{u^2+v^2}$ we find
	\begin{equation}
		\rho=\sqrt{(ax+y)^2+(by-x)^2}.
	\end{equation}
	Therefore, substituting our findings in the expressions of $h_0,h_1,h_2$ of \cref{ex:homog2} we find
	\begin{equation}\label{eq:h_i_notation}
		h_0=\frac{2}{\rho},
		\qquad
		\begin{pmatrix}
			h_1\\h_2
		\end{pmatrix}
		=
		-\frac{1}{\rho}\mathfrak R(\tfrac{s}2)w.
	\end{equation}
	We write the normal exponential map componentwise as $E_{a,b}=(X,Y,Z)$. By definition of the map $E_{a,b}$, combining \cref{prop:normal_exponential_map} with \eqref{eq:notation} and \eqref{eq:h_i_notation} we find 
	\begin{equation}
		\begin{aligned}
			\begin{pmatrix}
				X\\Y
			\end{pmatrix}
			&=
			\begin{pmatrix}
				x\\ y
			\end{pmatrix}
			+\frac{1}{h_0}
			\left[
			\mathfrak R\left(-\frac{\pi}{2}\right)
			+
			\mathfrak R\left(\frac{\pi}{2}-s\right)
			\right]
			\begin{pmatrix}
				h_1\\h_2
			\end{pmatrix}
			=
			\xi-\frac{1}{2}
			\left[
			\mathfrak R\left(-\frac{\pi}{2}\right)
			+
			\mathfrak R\left(\frac{\pi}{2}-s\right)
			\right]
			\mathfrak R(\tfrac{s}{2})w\\
			&=\xi-\frac{1}{2}
			\left[
			\mathfrak R\left(-\frac{\pi}{2}+\frac{s}{2}\right)
			+
			\mathfrak R\left(\frac{\pi}{2}-\frac{s}{2}\right)
			\right]w =\xi-\sin\tfrac{s}{2}\,w=\left[
			A^{-1}\mathfrak R\left(\frac{s}{2}\right)
			-\sin\frac{s}{2}\,\mathrm{Id}
			\right]w,
		\end{aligned}
	\end{equation}
	where we have used the property $\mathfrak R(\alpha)\mathfrak R(\beta)=\mathfrak R(\alpha+\beta)$ in the third equality and the identity 
	\begin{equation}
		\frac{1}{2}
		\left[
		\mathfrak R\left(-\frac{\pi}{2}+\theta\right)
		+
		\mathfrak R\left(\frac{\pi}{2}-\theta\right)
		\right]
		=
		\sin\theta\,\mathrm{Id},
	\end{equation}
	in the fourth one. Now, for $c\in\{a,b\}$, let $p_c(s)$ be defined as in \eqref{eq:def_p-q_c}. Expanding the matrix product $
	\left[
	A^{-1}\mathfrak R\left(\frac{s}{2}\right)
	-\sin\frac{s}{2}\,\mathrm{Id}
	\right]w$ gives
	\begin{equation}\label{eq:X,Y_domain}
	\begin{pmatrix}
		X\\ Y
	\end{pmatrix}\circ\varphi(s,u,v)
		=
		\frac{1}{1+ab}
		\begin{pmatrix}
			bp_a(s)&-p_b(s)\\
			p_a(s)&ap_b(s)
		\end{pmatrix}
		\begin{pmatrix}
			u\\v
		\end{pmatrix}=A^{-1}\begin{pmatrix}
		p_a(s) u\\ p_b(s)v
	\end{pmatrix}.
	\end{equation}
	We now compute the third component of $E_{a,b}$. Recalling that $S_{a,b}$ is the graph of (half) $u_{a,b}(x,y)=\frac{1}{2}(ax^2+by^2)$, equation \eqref{eq:simplified_z} gives
	\begin{equation}
		Z
		=
		\frac{1}{4}(ax^2+by^2)
		+\frac{s-\sin s}{2h_0^2}
		+
		\frac{1}{2h_0}
		\xi^T
		\left[
		\mathrm{Id}+\mathfrak R(\pi-s)
		\right]
		\begin{pmatrix}
			h_1\\h_2
		\end{pmatrix}.
	\end{equation}
	Note that the symmetric part of $A$ is the diagonal matrix $\mathrm{diag}(a,b)$, hence
	\begin{equation}
		ax^2+by^2=\xi^TA\xi=\xi^T\mathfrak R(\tfrac{s}{2})w,
	\end{equation}
	where we have used the identity $A\xi=\mathfrak R(\tfrac{s}{2})w$ obtained from \eqref{eq:xi_w}. Moreover, by \eqref{eq:h_i_notation} it holds 
	\begin{equation}
		\frac{s-\sin s}{2h_0^2}
		=
		\frac{s-\sin s}{8}(u^2+v^2),\qquad \frac{1}{2h_0}
		\begin{pmatrix}
			h_1\\h_2
		\end{pmatrix}
		=
		-\frac{1}{4}\mathfrak R(\tfrac{s}{2})w.
	\end{equation}
	It follows that the component $Z$ can be rewritten using the notation \eqref{eq:notation} as 
	\begin{equation}
		\begin{aligned}
			Z&=
			\frac{1}{4}\xi^T\mathfrak R(\tfrac{s}{2})w
			+\frac{s-\sin s}{8}(u^2+v^2)
			-\frac{1}{4}\xi^T
			\left[
			\mathrm{Id}+\mathfrak R(\pi-s)
			\right]
			\mathfrak R(\tfrac{s}{2})w\\
			&=\frac{1}{4}\xi^T\mathfrak R(\tfrac{s}{2})w
			+\frac{s-\sin s}{8}(u^2+v^2)
			-\frac{1}{4}\xi^T\mathfrak R(\tfrac{s}{2})w-\frac{1}{4}\xi^T\mathfrak R(\pi-\tfrac{s}{2})w\\
			&=\frac{s-\sin s}{8}(u^2+v^2)-\frac{1}{4}\xi^T\mathfrak R(\pi-\tfrac{s}{2})w=\frac{s-\sin s}{8}w^Tw+\frac{1}{4}\xi^T\mathfrak R\left(-\tfrac{s}{2}\right)w.
		\end{aligned}
	\end{equation}
	Substituting $\xi=A^{-1}\mathfrak R(\tfrac{s}{2})w$ we are left with 
		\begin{equation}
		Z
		=
		\frac{s-\sin s}{8}w^Tw
		+\frac{1}{4}w^T\left[\mathfrak R\left(-\tfrac{s}{2}\right)A^{-T}\mathfrak R\left(-\tfrac{s}{2}\right)\right]w.
	\end{equation}
    In particular, $Z$ is a quadratic form in $w$, hence we may replace $N(s):=\mathfrak R\left(-\tfrac{s}{2}\right)A^{-T}\mathfrak R\left(-\frac{s}{2}\right)$ by its symmetric part which is
	\begin{equation}
		\frac{1}{2}(N(s)+N(s)^T)=
		\frac{1}{2(1+ab)}
		\begin{pmatrix}
			(b-a)+(a+b)\cos s+2\sin s & 0\\
			0 & (a-b)+(a+b)\cos s+2\sin s
		\end{pmatrix}.
	\end{equation}
Replacing $N(s)$ with its symmetric part and expanding, we find
	\begin{equation}\label{eq:Z_s,u,v}
		Z\circ\varphi(s,u,v)=M_{11}(s)u^2+M_{22}(s)v^2,
	\end{equation}
	where we have denoted
	\begin{equation}\label{eq:Mii}
    \begin{aligned}
        M_{11}(s)
		&=
		\frac{
			(1+ab)s+(1-ab)\sin s+(b-a)+(a+b)\cos s
		}{
			8(1+ab)
		},\\
		M_{22}(s)
		&=
		\frac{
			(1+ab)s+(1-ab)\sin s+(a-b)+(a+b)\cos s
		}{
			8(1+ab)
		}.
    \end{aligned}
	\end{equation}
Summarizing we have
\begin{equation}
	E_{a,b}\circ \varphi(s,u,v)=(
		X, Y, Z
	)\circ \varphi(s,u,v)=\left(
	 A^{-1}\begin{pmatrix}
	 	p_a(s) u\\ p_b(s)v
	 \end{pmatrix},
 M_{11}(s)u^2+M_{22}(s)v^2
\right).
\end{equation}
	We now perform the change of coordinates in the codomain, finding
	\begin{equation}
		\psi\circ E_{a,b}\circ\varphi(s,u,v)
		=
		\left(\begin{pmatrix}
			aX+Y\\
			bY-X
		\end{pmatrix}
		\circ\varphi(s,u,v),\,
		8\left(Z-\frac{aX^2+bY^2}{4}\right)\circ\varphi(s,u,v)
		\right).
	\end{equation}
	Using equation \eqref{eq:X,Y_domain}, referring to notation \eqref{eq:notation}, we obtain 
	\begin{equation}
		\begin{pmatrix}
			aX+Y\\
			bY-X
		\end{pmatrix}
		\circ\varphi(s,u,v)=A	\begin{pmatrix}
			X\\ Y
		\end{pmatrix}\circ\varphi(s,u,v)=AA^{-1}\begin{pmatrix}
		p_a(s)u\\
		p_b(s)v
	\end{pmatrix}=\begin{pmatrix}
	p_a(s)u\\
	p_b(s)v
\end{pmatrix},
	\end{equation}
    which corresponds to the first two components of $E_{a,b}$ as claimed in \eqref{eq:E_coords}. Finally, for the third component of $E_{a,b}$, a standard computation gives
	\begin{equation}
		\begin{aligned}
			8\left(Z-\frac{aX^2+bY^2}{4}\right)\circ\varphi(s,u,v)
			&=
			\left(
			8M_{11}(s)
			-\frac{2b\,p_a(s)^2}{1+ab}
			\right)u^2+
			\left(
			8M_{22}(s)
			-\frac{2a\,p_b(s)^2}{1+ab}
			\right)v^2.
		\end{aligned}
	\end{equation}
	Expanding the squares of $p_c(s)$ and using the trigonometric identities
	\begin{equation}
		2\sin^2\tfrac{s}{2}=1-\cos s,
		\qquad
		2\cos^2\tfrac{s}{2}=1+\cos s,
		\qquad
		2\sin\tfrac{s}{2}\cos\tfrac{s}{2}=\sin s,
	\end{equation}
    we find
	\begin{equation}
		\begin{aligned}
			8M_{11}(s)
			-\frac{2b\,p_a(s)^2}{1+ab}
			=
			q_a(s),\qquad  8M_{22}(s)
			-\frac{2a\,p_b(s)^2}{1+ab}
			=
			q_b(s),
		\end{aligned}
	\end{equation}
	where $q_c(s)$ for $c\in\{a,b\}$ is as in \eqref{eq:def_p-q_c}. This concludes the proof of \eqref{eq:E_coords} for the map $E_{a,b}$. In a similar fashion, combining \eqref{eq:h_i_notation} with \eqref{eq:length_s}, we find the expression of the length function as well.
\end{proof}

\begin{remark}\label{r:identities}
    We record here a few useful facts about the functions $p_c,q_c$ defined in \eqref{eq:def_p-q_c}. Firstly, a standard computation shows the validity of the following identities
    \begin{equation}
    \label{eq:identities_remark}
        \begin{aligned}
            q_c(s)=s+2\sin(\tfrac{s}2)p_c(s),\qquad 
\sin(\tfrac{t}2)p_c(s)-\sin(\tfrac{s}2)p_c(t)
=\sin(\tfrac{t-s}2).
        \end{aligned}
    \end{equation}
    Thus, setting $P(s)=\mathrm{diag}(p_a(s),p_b(s))$ and $w=(u,v)$, we may write the third component of $E_{a,b}$ as
    \begin{equation}
        Z_{a,b}(s,w)=s|w|^2+2\sin(\tfrac{s}2)\langle P(s)w,w\rangle.
    \end{equation}
 Secondly, for any $c\in\mathbb R$, the first positive root of $p_c$ is
	\begin{equation}
		s_c:=\pi-2\arctan(c),
	\end{equation}
	and it is isolated. In particular, $p_c$ is positive in the interval $[0,s_c)$ and, if $a\geq b$, then $s_b\geq s_a$.
\end{remark}
With the final goal of characterizing the cut time of a homogeneous surface, we study the double points of its extended exponential map.

\begin{prop} \label{p:seipuntotre} 
Let $\mathscr F_{a,b}$ be the extended exponential map of $S_{a,b}$. If $\mathscr F_{a,b}(s,w)=\mathscr F_{a,b}(t,\tilde w)$ for some $(s,w),(t,\tilde w)\in (\R^+\times\R^2 \setminus \{0\})$, then $s=t$. Moreover
\begin{equation}
\label{eq:doub_F=smth}
    \overline{\mathrm{Doub}\bigl(\mathscr F_{a,b},[0,s_a]\times \left(\mathbb R^2\setminus\{0\}\right)\bigr)}= \{s_{a}\}\times(\mathbb R^2\setminus\{0\}),
\end{equation}
where the closure is taken in $\mathbb R^+\times\left(\mathbb R^2\setminus\{0\}\right)$.
\end{prop}
\begin{proof} Letting $w=(u,v)$ and $\widetilde w=(\widetilde u,\widetilde v)$, we study the equation $\mathscr F_{a,b}(s,w)=\mathscr F_{a,b}(t,\widetilde w)$. As a consequence of \cref{l:belle_coords} and taking into account \cref{r:identities}, the latter equation becomes
\begin{equation}
\label{eq:F=F}
    P(s)w=P(t)\widetilde w,\qquad Z_{a,b}(s,w)=Z_{a,b}(t,\widetilde w),\qquad s|w|=t|\widetilde w|.
\end{equation}
First, we claim that any solution of \eqref{eq:F=F} satisfies $s=t$. On the hand, if one among $t$ and $s$ is zero, equality of lengths and the non-vanishing of $w,\widetilde w$ imply $s=t=0$. Then, since $P(0)=\mathrm{id}$, the first equality of \eqref{eq:F=F} gives $w=\widetilde w$. On the other hand, if $t,s>0$, applying
\cref{r:identities} to the second equality of \eqref{eq:F=F}, we find 
\begin{equation}
\begin{aligned}
0=Z_{a,b}(t,\widetilde w)-Z_{a,b}(s,w)
&=t|\widetilde w|^2-s|w|^2+2\sin(\tfrac{t}2)\langle P(s)w,\widetilde w\rangle
-2\sin(\tfrac{s}2)\langle P(t)\widetilde w,w\rangle\\
&=t|\widetilde w|^2-s|w|^2+2\left\langle
\bigl(\sin(\tfrac{t}2)P(s)-\sin(\tfrac{s}2)P(t)\bigr)w,
\widetilde w
\right\rangle\\
&=t|\widetilde w|^2-s|w|^2
+2\sin(\tfrac{t-s}2)\langle w,\widetilde w\rangle.
\end{aligned}
\end{equation}
The last equality follows by applying the second identity of \eqref{eq:identities_remark} to each
diagonal entry of $P(\cdot)$. Dividing by $|w||\widetilde w|$ and using $s|w|=t|\widetilde w|$, we obtain
\begin{equation}\label{eq:unified_sinc}
t-s
=2\sin(\tfrac{t-s}2)
\frac{\langle w,\widetilde w\rangle}{|w||\widetilde w|}.
\end{equation}
By Cauchy-Schwarz inequality, for $s\ne t$, the right-hand side has absolute value strictly less than $|t-s|$. Hence, it must be $s=t$, concluding the proof the claim. We now prove \eqref{eq:doub_F=smth}. Let $(s,w)$, $(t,\tilde w)\in [0,s_a]\times\R^2\setminus\{0\}$ be distinct points such that $\mathscr F(s,w)=\mathscr F(t,\tilde w)$. Then, by the first part of the proof, we know that $s=t$. Now, if $s<s_a$, $P(s)$ is invertible, so the equality $P(s)w=P(t)\widetilde w$ of \eqref{eq:F=F} forces $w=\widetilde w$, which is impossible. Therefore, we must have $s=s_a$ and $w\neq \tilde w$. This implies that $\mathrm{Doub}(\mathscr F_{a,b},[0,s_a]\times (\mathbb R^2\setminus\{0\}))\subset \{s_a\}\times(\mathbb R^2\setminus\{0\})$. Conversely, observe that a direct computation using \cref{l:belle_coords} shows that, for every $(u,v)\in\R^2\setminus\{0\}$, it holds 
\[
 \mathscr F_{a,b}(s_a,u,v) = \mathscr F_{a,b}(s_a,-u,v), 
\]
which means that $\{s_a\}\times(\mathbb R^2\setminus\{u=0\})\subset \mathrm{Doub}(\mathscr F_{a,b},[0,s_a]\times (\mathbb R^2\setminus\{0\}))$.
\end{proof}
We can now prove the main result of the section.
\begin{proof}[Proof of \cref{thm:intro-homogeneous-cut-time}]
    We work in coordinates $(s,\rho,\phi)$ given in \cref{rmk:pi}. In these coordinates, the length function reads as $\ell(s,\rho,\phi)=s\rho/2$ (recall that $h_0=2/\rho$ for $S_{a,b}$ by \eqref{eq:components_covector}). Applying \cref{r:cut_rho}, the cut time is 
	\begin{equation}
    \label{eq:cut-time_aux}
		\begin{aligned}
			t_\cut(\rho,\phi)=\frac{\rho}{2}\inf\left\{s\,:(s,\rho,\phi+\tfrac{s}{2})\in\overline{\mathrm{Doub}(\mathscr F_{a,b})}\right\}.
	\end{aligned}
	\end{equation}
By \cref{p:seipuntotre}, the double locus of $\mathscr F_{a,b}:\R^+\times(\R^2\setminus\{0\})\to \hei\times\R^+$ satisfies
\[
 \{s_a\}\times(\mathbb R^2\setminus\{0\})
 \subset\overline{\mathrm{Doub}(\mathscr F_{a,b})}
 \subset[s_a,\infty)\times(\mathbb R^2\setminus\{0\}),
\]
where we recall that the closure is taken in $\mathbb R^+\times\mathbb R^2\setminus\{0\}$. Thus, the infimum in \eqref{eq:cut-time_aux} equals $s_a$, and so
\begin{equation}
	t_\cut(\rho,\phi)=\frac{\rho}{2}s_a=\rho\left(\frac{\pi}{2}-\arctan(a)\right)=\rho\left(\frac{\pi}{2}-\arctan(\max\{a,b\})\right),
\end{equation}
where the latter equality uses the hypothesis $a\geq b$. We conclude the proof recalling the definition of $\rho$ in \eqref{eq:x-y_polar}.
\end{proof}
\section{Approximations of the double point locus}
\label{sec:cut_time_general}

Let $S=\partial\Omega\subset\hei$ be a smooth, embedded, oriented surface and let $p\in\Sigma(S)$ be a geometrically stable characteristic point, cf.\ \cref{def:geometrically_stable_char_pts}. In this section, we show how one can describe the double locus of the extended exponential map of $S$ close to $p$, in terms of the homogeneous approximation. As before we assume that $p=(0,0,0)$ and we describe $S$ locally as in \eqref{eq:surface_S1}, namely as the graph of a function $u\in C^\infty(\R^2)$ given by \cref{lem:nice_f}. The geometric stability of $p$ is encoded in the conditions 
\begin{equation}\label{eq:geom_stab}
    ab+1\neq 0,\qquad a\neq b.
\end{equation}
As before, without loss of generality, we may additionally assume $a>b$.

 \subsection{Localization of the double points of the extended exponential map}

Let $\mathscr F=(E,\ell)$ be the extended exponential map of $S\subset\hei$, expressed in coordinates $(s,\rho,\phi)$ of \cref{ss:change}. Observe that in these polar coordinates $\mathscr F(s,\rho,\phi)$ is a smooth map, including at $\rho=0$, indeed
\begin{equation}
    \mathscr F(s,\rho,\phi)=(E(s,\rho,\phi),\ell(s,\rho,\phi)),
\end{equation}
and $E(s,\rho,\phi)$ extends smoothly to $\rho=0$ by \cref{lem:smooth_E}, while $\ell(s,\rho,\phi)$ extends smoothly by \cref{rmk:pi}.
We have already shown that the map $E$ can be approximated by the corresponding map $E_{a,b}$, cf.\ \cref{rmk:E_approx}. The same result holds for the length function. Hence, after applying the change of coordinates of \cref{ss:change}, it holds: 
\begin{equation}\label{eq:fundamental_approx}
	E(s,\rho,\phi)=\delta_{\ve}\left(E_{a,b}(s,\rho/\ve,\phi)+\ve R_E(s,\rho/\ve,\rho,\phi)\right),\qquad\forall\,(s,\rho,\phi)\in\mathbb R^+\times \mathbb R^+\times\mathbb S^1,
\end{equation}
where $R_E:(\R^+)^3\times\mathbb S^1\to\mathbb R^3$ is a smooth function and $\delta_\ve(x,y,z)=(\ve x,\ve y,\ve^2 z)$ as in \eqref{eq:dilations}. Moreover, the formula for the length obtained in \cref{rmk:pi}, together with \cref{rmk:diver_h0} gives 
\begin{equation}
\label{eq:approx_length}
    \ell(s,\rho,\phi)=\frac{s}{h_0(\rho,\phi-\tfrac s2)}=\frac{s\rho}{2}+\rho^2R_\ell(s,\rho,\phi).
\end{equation}
In particular, the latter two equations show that $\mathscr F$ is a smooth perturbation of $\mathscr F_{a,b}=(E_{a,b},\ell_{a,b})$. For the next proposition, recall the quantity $s_a=\pi-2\arctan(a)$, defined in \cref{r:identities}. Moreover, in coordinates $(\rho,\phi)$, for $\rho_0>0$, we set $$B_{\rho_0}^\times:=\{q=(\rho,\phi)\in S: 0<\rho\leq\rho_0\}$$ 
to be a punctured ball around the characteristic point $p=(0,0,0)$.

\begin{prop}\label{lem:singular_pos}
	For every $\delta>0$ sufficiently small, there exists $\rho_0>0$ such that
	\begin{equation}\label{eq:double_dovesei}
		\mathrm{Doub}(\mathscr F,[0,s_a+\delta]\times B_{\rho_0}^\times)=\mathrm{Doub}(\mathscr F,[s_a-\delta,s_a+\delta]\times B_{\rho_0}^\times).
	\end{equation}
\end{prop}
\begin{proof}
Fix $\delta>0$. We must show that there exists $\rho_0>0$ such that, for every $(s,w)\in \mathrm{Doub}(\mathscr F,[0,s_a+\delta]\times B_{\rho_0}^\times)$, it holds $s\geq s_a-\delta$. Assume, by contradiction, that this is not the case. Hence, we find a collection of distinct pairs 
\begin{equation}
    (s_n,w_n):=(s_n,\rho_n,\phi_n),\quad  (t_n,\widetilde w_n):=(t_n,\widetilde\rho_n,\widetilde\phi_n), \qquad n\in\N,
\end{equation}
such that 
\begin{equation} 
\label{eq:equality_of_images}
0\le s_n\le s_a-\delta,\qquad 0\le t_n\le s_a+\delta, \qquad
\mathscr F(s_n,w_n)=\mathscr F(t_n,\widetilde w_n),
\end{equation}
and it holds $\rho_n,\widetilde\rho_n\rightarrow0$ as $n\to\infty$, with $\rho_n,\widetilde \rho_n\ne0$.  
For every $n\in\N$, setting $r_n:=\sqrt{\rho_n^2+ \widetilde\rho_n^2}$, we normalize both $w_n,\widetilde w_n$ as
\begin{equation}
\zeta_n:=\left(\frac{\rho_n}{r_n},\phi_n\right),\qquad
\widetilde\zeta_n:=\left(\frac{\widetilde\rho_n}{r_n},\widetilde\phi_n\right),
\end{equation}
so that $|\zeta_n|^2+|\widetilde\zeta_n|^2=1$, where the moduli stand for the radial coordinates. Given $\ve> 0$, define 
\begin{equation}
    \mathscr F_\ve(s,\rho,\phi)=(E_\ve(s,\rho,\phi),\ell_\ve(s,\rho,\phi)):=(\delta_\ve^{-1} E(s,\ve \rho,\phi),\ve^{-1} \ell(s,\ve \rho,\phi)),
\end{equation}
for every $(s,\rho,\phi)\in\R^+\times\R^+\times\mathbb S^1$. Observe that, for all $\ve>0$, the homogeneity property of \cref{lem:homogeneity_normal_exp} rewrites for the map $E_{a,b}$ in coordinates $(s,\rho,\phi)$ as 
	\begin{equation}
		\delta_\ve E_{a,b}(s,\rho,\phi)=E_{a,b}(s,\ve\rho, \phi),\qquad \forall\,(s,\rho,\phi)\in\R^+\times\R^+\times\mathbb S^1.
	\end{equation}
Thus, by \eqref{eq:fundamental_approx}, as $\ve\to 0$, $\mathscr F_\ve$ $C^1$-converges to $\mathscr F_{a,b}$, uniformly on compact sets of $\R^+\times \R^+\times\mathbb S^1$. The equality in \eqref{eq:equality_of_images} may be written as
\begin{equation}\label{eq:F_r_n}
\mathscr F_{r_n}(s_n,\zeta_n)
=\mathscr F_{r_n}(t_n,\widetilde\zeta_n),\qquad\forall\,n\in\N.
\end{equation}
By compactness, up to extracting a subsequence, there exist $(s,\zeta), (t,\widetilde\zeta)\in \R^+\times\R^+\times\mathbb S^1$ such that $(s_n,\zeta_n)\to (s,\zeta)$ and $(t_n,\widetilde\zeta_n)\to (t,\widetilde\zeta)$. Uniform convergence of the rescaled maps then implies that
\begin{equation}\label{eq:unified_limit_pair}
\mathscr F_{a,b}(s,\zeta)=\mathscr F_{a,b}(t,\widetilde\zeta),\qquad\text{and}\qquad
|\zeta|^2+|\widetilde\zeta|^2=1.
\end{equation}
We claim that neither $|\zeta|$ nor $|\widetilde\zeta|$ vanishes. Indeed, if by contradiction one of the two vanishes, say $|\zeta|=0$, then, by \eqref{eq:unified_limit_pair}, $\widetilde\zeta\neq 0$ and $\mathscr F_{a,b}(t,\widetilde\zeta)=\mathscr F_{a,b}(s,0)=0$. Now, if $t>0$, the length coordinate is $\ell_{a,b}(t,\widetilde\zeta) = t|\widetilde\zeta|/2$, meaning that $|\widetilde\zeta|=0$. If $t=0$, the first two components of $\mathscr F_{a,b}(t,\widetilde\zeta)$ equal $\widetilde\zeta$, meaning that the latter vanishes. In both cases, we get a contradiction. Now \cref{p:seipuntotre} gives $s=t$.
Moreover, since $s\le s_a-\delta<s_a$, \cref{p:seipuntotre} also implies that $\zeta=\widetilde\zeta$. Thus, both sequences converge to the same limit 
\begin{equation}
    z=(s,\zeta):=\lim_{n\to\infty}(s_n,\zeta_n)=\lim_{n\to\infty}(t_n,\widetilde\zeta_n),\qquad\text{with } |\zeta|=\frac{1}{\sqrt 2} \text{ and } 0\leq s\leq s_{a}-\delta.
\end{equation}
On the other hand, by \cref{c:umaiuscolo}, $E_{a,b}|_{ [0,s_a)\times\mathbb R^2\setminus\{0\}}$ is a diffeomorphism onto its image, hence, $d_{z}E_{a,b}$ is invertible. Thus, we can find a small compact ball $V$ centered at $z$ and a constant $C>0$ such that
\begin{equation}
|E_{a,b}(z_1)-E_{a,b}(z_2)|\ge C|z_1-z_2|,
\qquad \forall\, z_1,z_2\in V.
\end{equation}
Since, as $\ve\to 0$, $E_\ve\to E_{a,b}$ in $C^1(V)$, for every $z_1,z_2\in V$ we find
\begin{equation}
\begin{aligned}
|(E_{\ve}-E_{a,b})(z_1)-(E_{\ve}-E_{a,b})(z_2)|\le
\sup_V\|d E_{\ve}-dE_{a,b}\|\,|z_1-z_2|
=o(1)|z_1-z_2|.
\end{aligned}
\end{equation}
Consequently, for every $z_1,z_2\in V$, we can estimate
\begin{equation}
\begin{aligned}
|E_{\ve}(z_1)- E_{\ve}(z_2)|
&\geq \Big||E_{a,b}(z_1)-E_{a,b}(z_2)|-|(E_{\ve}-E_{a,b})(z_1)-(E_{\ve}-E_{a,b})(z_2)|\Big|\\
&\geq \Big||E_{a,b}(z_1)-E_{a,b}(z_2)|-o(1)|z_1-z_2|\Big| \\
&\ge (C-o(1))|z_1-z_2|\geq \frac{C}{2}|z_1-z_2|,
\end{aligned}
\end{equation}
for sufficiently small ${\ve}$. Thus $E_{\ve}$ is injective
on the same neighborhood $V$ for all such ${\ve}$. This contradicts the existence of the two distinct sequences $(s_n,\zeta_n)\neq (t_n,\widetilde\zeta_n)$ approaching $z$, and satisfying \eqref{eq:F_r_n}, for sufficiently large $n\in\N$.
\end{proof}

\begin{remark}
    Since $\mathscr F$ is a smooth perturbation of $\mathscr F_{a,b}$, using \cref{lem:singular_pos}, one may be tempted to conclude that the closure of the double point locus of $\mathscr F$ is a smooth perturbation of the corresponding set of the homogeneous model $\mathscr F_{a,b}$. However, the latter contains critical points of $\mathscr F_{a,b}$ which are not stable under arbitrary smooth perturbations. Indeed, due to the validity of \eqref{eq:lege}, $\mathscr F$ is a \emph{Legendrian front} (cf.\,\cref{app:Legendrian}), and thus its critical points are unstable in the class of all smooth maps (see \cite{Ishikawa2005}). Thus it is natural to study the stability of $\mathscr F$ in the class of Legendrian fronts. 
\end{remark}
\subsection{Reduction to a family of Legendrian fronts}\label{ssec:red_front}
We now describe the Legendrian structure of $\mathscr F$ (see Appendix~\ref{app:Legendrian}). Using the coordinates $(s,\rho,\phi)$ of \cref{ss:change}
and writing $\mathscr F=(X,Y,Z,\ell)$, by \cref{rmk:leges}, for $\rho>0$ we have
\begin{equation}\label{eq:ci2}
    d\ell=P_X\,dX+P_Y\,dY+P_Z\,dZ,
\end{equation}
where $(P_X,P_Y,P_Z)$ are the components of the normal lift in the target coordinates $(X,Y,Z)$.
Although $\mathscr F$ extends smoothly to $\rho=0$, its Legendrian lift in these coordinates does not. More precisely one can show that $P_Z(s,\rho,\phi)\sim\rho^{-1}$ as $\rho\to 0^+$. This suggests restricting $\mathscr F$ to level sets of $Z$. On each such level, the pullback of $dZ$ vanishes, so the singular
component $P_Z$ disappears from the contact identity \eqref{eq:ci2}. Moreover, this reduction preserves all double point information: if two distinct source points have the same image under $\mathscr F$, then they necessarily have the same value of $Z$.

In the next lemma we formalize this construction by describing the level sets $Z=\ve^2$.
 Given $\delta,\rho_0>0$ determined by \cref{lem:singular_pos}, we study the following restriction of the extended exponential map to the set $[s_a-\delta,s_a+\delta]\times B_{\rho_0}^\times$. To lighten the notation we set 
\begin{equation}
    I:=[s_a-\delta,s_a+\delta].
\end{equation}
We identify $B_{\rho_0}^\times$ with $(0,\rho_0]\times \mathbb S^1$, using the polar coordinates $(\rho,\phi)\in (0,\rho_0]\times \mathbb S^1$.
\begin{lemma}\label{lem:Z} 
For every $\delta>0$ small enough there exists $\rho_0>0$ such that the following map
	\begin{equation}\label{eq:Psi_Z}
		\begin{aligned}
			\Psi_Z:I\times [0,\rho_0]\times \mathbb S^1\to I\times \R^+\times \mathbb S^1,\qquad
			 \Psi_Z(s,\rho,\phi)=\left(s,\sqrt{Z(s,\rho,\phi)},\phi\right),
		\end{aligned}
	\end{equation}
	is a diffeomorphism onto its image. In particular its image contains $I\times [0,\ve_0]\times \mathbb S^1$ for $\ve_0>0$ small enough, and on the latter domain its inverse satisfies 
	\begin{equation}\label{eq:Psi_Z-1}
		\begin{aligned}
			\Psi_Z^{-1}(s,\ve,\phi)=\left(s,\frac{\ve}{\sqrt{q_a(s)\cos^2\phi+q_b(s)\sin^2\phi}}+o(\ve),\phi\right), \qquad\text{as }\ve\to 0^+.
		\end{aligned}
	\end{equation}
\end{lemma}
\begin{proof}
	 We first show that $Z(s,\rho,\phi)$ is non-negative in $I\times [0,\rho_0]\times \mathbb S^1$. Combining \cref{l:belle_coords} with \eqref{eq:fundamental_approx}, applied with $\ve=\rho$, we have 
	\begin{equation}\label{eq:Z_is_pos}
		Z(s,\rho,\phi)=\rho^2(q_a(s)\cos^2\phi+q_b(s)\sin^2\phi+\rho R_z(s,\rho,\phi)),
	\end{equation}
	where $R_z(s,\rho,\phi)$ is a smooth function, including at $\rho=0$. Using the definitions of $q_a,q_b$ given in \cref{l:belle_coords}, a direct evaluation gives 
	\begin{equation}\label{eq:q>0}
		q_a(s_a)=s_a>0,\qquad q_b(s_a)=s_a+\frac{2(a-b)}{1+a^2}>0,
	\end{equation}
	where for the second inequality we have used the assumption $a\geq b$. Thus, for $\delta>0$ small enough, $q_a,q_b$ remain strictly positive on the interval $I$. Consequently, for $\rho_0>0$ small enough, $Z$ is non-negative on $I\times [0,\rho_0]\times \mathbb S^1$. Moreover, combining \eqref{eq:Z_is_pos} and \eqref{eq:q>0}, together with the computation 
    \[
    \partial_\rho\sqrt{Z}\big|_{\rho=0} = \sqrt{q_a(s)\cos^2\phi+q_b(s)\sin^2\phi},
    \]
    we deduce that, for $\rho_0>0$ sufficiently small, $\partial_\rho \sqrt{Z}>0$ on $I\times [0,\rho_0]\times \mathbb S^1$. Therefore, \eqref{eq:Psi_Z} is an injective immersion on a compact manifold and hence a diffeomorphism onto its image. Applying $\Psi_Z^{-1}$ from the right to both sides of \eqref{eq:Z_is_pos} yields 
	\begin{equation}
		\ve^2=\rho^2\circ \Psi_Z^{-1}(s,\ve,\phi)(q_a(s)\cos^2\phi+q_b(s)\sin^2\phi)+O(\rho^3\circ \Psi_Z^{-1}(s,\ve,\phi)), \qquad \text{as }\ve\to0^+.
	\end{equation}
	From the latter equation we deduce that $\rho\circ \Psi_Z^{-1}(s,\ve,\phi)=O(\ve)$, uniformly with respect to $s,\phi$. Therefore, 
	\begin{equation}
			\rho^2\circ \Psi_Z^{-1}(s,\ve,\phi)(q_a(s)\cos^2\phi+q_b(s)\sin^2\phi)=\ve^2+O(\ve^3), \qquad \text{as }\ve\to0^+.
	\end{equation}
Inverting the above equality for $\rho^2\circ \Psi_Z^{-1}(s,\ve,\phi)$ yields \eqref{eq:Psi_Z-1}.
\end{proof}
 Using the previous lemma, we restrict $\mathscr F$ to a fixed level set $\{Z=\ve^2\}$. When restricted, the remaining components $(X,Y,\ell)$ of $\mathscr F$ are of order $\ve$. This justifies the following definition. For every $\ve\in (0,\ve_0]$, we define
	\begin{equation}\label{eq:Fve}
		\begin{aligned}
		F^\ve:[s_a-\delta,s_a+\delta]\times \mathbb S^1\to \mathbb R^2\times \R,\qquad
			F^\ve(s,\phi)= \frac{1}{\ve}(X,Y,\ell)\circ\Psi_Z^{-1}(s,\ve,\phi).
		\end{aligned}
	\end{equation}
The double point locus of $\mathscr F$ may be traced back from the one of the family of maps $\{F^\ve\}_{\ve\in(0,\ve_0]}$.
\begin{lemma}\label{l:union}
	For every $\ve_0>0$ small enough, setting $V:=\Psi_Z^{-1}(I\times(0,\ve_0]\times \mathbb S^1)$, we have 
	\begin{equation}
		\mathrm{Doub}(\mathscr F,V)=\{\Psi_Z^{-1}(s,\ve,\phi)\,:\,(s,\phi)\in\mathrm{Doub}(F^\ve,I\times \mathbb S^1),\,\ve\in(0,\ve_0]\}.
	\end{equation}
\end{lemma}
\begin{proof}
Fix $\ve>0$. The third component of
$\mathscr F(\Psi_Z^{-1}(s,\ve,\phi))$ is $\ve^2$. Deleting from $\mathscr F\circ \Psi_Z^{-1}$ this constant component and dividing the remaining components by $\ve$ gives $F^\ve$. Consequently,
\begin{equation}
\begin{aligned}
\mathscr F\bigl(\Psi_Z^{-1}(s,\ve,\phi)\bigr)
 =\mathscr F\bigl(\Psi_Z^{-1}(s',\ve,\phi')\bigr)\iff 
F^\ve(s,\phi)=F^\ve(s',\phi').
\end{aligned}
\end{equation}
Since $\Psi_Z^{-1}$ is injective by \cref{lem:Z}, the two points $\Psi_Z^{-1}(s,\ve,\phi)$ and $\Psi_Z^{-1}(s',\ve,\phi')$ are distinct precisely when $(s,\phi)\ne(s',\phi')$. This proves the claimed description of $\mathrm{Doub}(\mathscr F,V)$.
\end{proof}
\subsection{The limiting front and its Legendrian lift}

We show that the maps $F^\ve$ are actually Legendrian fronts, including at $\ve=0$. In the following, we have denoted 
\begin{equation}
    \mathcal N: I\times\mathbb S^1\to \R^+,\qquad \mathcal N(s,\phi)=\sqrt{q_a(s)\cos^2\phi+q_b(s)\sin^2\phi}.
\end{equation}
\begin{lemma}\label{lem:front_family}
There exists a sufficiently small $\ve_0>0$ such that the maps $F^\ve$ form a smooth family of
two-dimensional Legendrian fronts for $0\le\ve\le\ve_0$, with
\begin{equation}\label{eq:front_limit}
F^0(s,\phi)=\frac1{\mathcal N(s,\phi)}
\left(p_a(s)\cos\phi,p_b(s)\sin\phi,\frac{s}{2}\right), \qquad \forall\,(s,\phi)\in I\times\mathbb S^1.
\end{equation}
Their Legendrian lifts may be chosen smoothly in $(s,\phi,\ve)$,
and the limiting lift is
\begin{equation}\label{eq:front_lift_limit}
\widetilde F^0(s,\phi)=
\bigl(F^0(s,\phi),-\sin(\tfrac{s}2)\cos\phi,-\sin(\tfrac{s}2)\sin\phi\bigr).
\end{equation}
\end{lemma}
\begin{proof}
    Equation \eqref{eq:fundamental_approx} with $\ve=\rho$ and \eqref{eq:approx_length} imply that 
    \begin{equation}\label{eq:appross}
        (X(s,\rho,\phi),Y(s,\rho,\phi),\ell(s,\rho,\phi))=\rho(p_a(s)\cos\phi,p_b(s)\sin\phi,s/2)+\rho^2R(s,\rho,\phi),
    \end{equation}
    where $R:I\times\R^+\times\mathbb S^1\to \R^3$ is smooth. \cref{lem:Z}, and in particular \eqref{eq:Psi_Z-1}, shows that, in coordinates $(s,\ve,\phi)$, it holds
    \begin{equation}\label{eq:rho_ve}
        \rho(s,\ve,\phi)=\frac{\ve}{\sqrt{q_a(s)\cos^2\phi+q_b(s)\sin^2\phi}}+o(\ve)= \frac{\ve}{\mathcal N(s,\phi)}+o(\ve),\qquad\text{as }\ve\to0^+.
    \end{equation}
    Substituting the latter equation into \eqref{eq:appross} we deduce that $F^\ve$ extends smoothly to $\ve=0$,  indeed:
    \[
    F^\ve(s,\phi) = \frac{1}{\mathcal N(s,\phi)}(p_a(s)\cos\phi,p_b(s)\sin\phi,s/2)+O(\ve),\qquad\text{as }\ve\to0^+.
    \]
    Moreover, from the latter equation, we see that $F^0$ is given by \eqref{eq:front_limit}. We now show that these maps are fronts by computing their Legendrian lifts. By \cref{rmk:leges}, the extended exponential map $\mathscr F=(X,Y,Z,\ell)$ satisfies
    \begin{equation}\label{eq:leg12}
        d\ell-P_XdX-P_YdY-P_ZdZ=0,
    \end{equation}
    where $(P_X,P_Y,P_Z)$ are the momentum components of the normal lift $$\lambda(s,\rho,\phi)=e^{(s/h_0(q))\vec H}(\lambda^\perp(q)),\qquad q=\pi(s,\rho,\phi),$$i.e., the evolution of $\lambda^\perp$ along the Hamiltonian flow, in the coordinates of \cref{ss:change}. We now explicitly compute these momentum components. Let $(h_1,h_2,h_0)$ be the components of $\lambda^\perp$ with respect to the coframe \eqref{eq:the_coframe} and recall that the evolution of the momentum of $\lambda^\perp$, with respect to the coframe \eqref{eq:the_coframe}, has first two components equal to $\mathfrak R(-s)(h_1,h_2)^T$, while its third component is constantly equal to $h_0$. In the coordinates of \cref{ss:change}, and in particular using \eqref{eq:psi}, we have $(X,Y,Z)=(ax+y,by-x,8(z-(ax^2+by^2)/4))$, and thus
    \begin{equation}
        \nu_0=\frac18\,dZ+\frac12(X\nu_1+Y\nu_2),
        \qquad
        \begin{pmatrix}\nu_1\\\nu_2\end{pmatrix}
        = \begin{pmatrix}a&1\\-1&b\end{pmatrix}^{-1}         \begin{pmatrix}dX\\dY\end{pmatrix}.
    \end{equation}
    A standard computation then shows that 
    \begin{equation}\label{eq:adapted_momenta}
        \begin{pmatrix}P_X\\P_Y\end{pmatrix}
        =
        \frac{1}{1+ab}
        \begin{pmatrix}b&1\\-1&a\end{pmatrix}
        \left[
            \mathfrak R(-s)
            \begin{pmatrix}h_1\\h_2\end{pmatrix}
            +
            \frac{h_0}{2}
            \begin{pmatrix}X\\Y\end{pmatrix}
        \right],
        \qquad
        P_Z=\frac{h_0}{8}.
    \end{equation}
    Thanks to \cref{rmk:diver_h0}, we may express the components $(h_1,h_2,h_0)$ of the initial covector $\lambda^\perp$ in terms of the source coordinates $(s,\rho,\phi)$, finding
    \begin{equation}
        h_0=\frac{2}{\rho}+H_0(s,\rho,\phi),
        \qquad
        \begin{pmatrix}h_1\\h_2\end{pmatrix}
        =
        -\mathfrak R(s/2)
        \begin{pmatrix}\cos\phi\\\sin\phi\end{pmatrix}
        +\rho \begin{pmatrix}
            H_1(s,\rho,\phi)\\ 
            H_2(s,\rho,\phi)
        \end{pmatrix},
    \end{equation}
     where the remainders $H_0,H_1,H_2:I\times \R^+\times \mathbb S^1\to \R$ are smooth. Substituting the latter expansions and \eqref{eq:appross} into \eqref{eq:adapted_momenta}, we find that the components of the normal lift $\lambda$ satisfy
    \begin{equation}\label{eq:boundary_momenta}
        \begin{pmatrix}P_X\\P_Y\end{pmatrix}
        =-\sin(s/2)
        \begin{pmatrix}\cos\phi\\\sin\phi\end{pmatrix}+\rho P(s,\rho,\phi),
    \end{equation}
    where $P:I\times\R^+\times \mathbb S^1\to \R^2$ is smooth. We now pass to the source coordinates $(s,\ve,\phi)$ given by \cref{lem:Z}. For every $(s,\ve,\phi)\in I\times[0,\ve_0]\times\mathbb S^1$, where $\ve_0>0$ is identified by \cref{lem:Z}, we define
    \begin{equation}
        \begin{aligned}
            \iota_\ve(s,\phi):=\Psi_Z^{-1}(s,\varepsilon,\phi)
            =\bigl(s,\rho(s,\varepsilon,\phi),\phi\bigr),
        \end{aligned}
    \end{equation}
    where $\rho(s,\ve,\phi)$ is given in \eqref{eq:rho_ve}. Recall that, by construction, it holds $Z\circ\iota_\varepsilon=\varepsilon^2$ and the family of maps $F^\ve$ satisfies 
    \begin{equation}
        F^\ve(s,\phi)=
        \left(
            \frac{X}{\varepsilon},
            \frac{Y}{\varepsilon},
            \frac{\ell}{\varepsilon}
        \right)\circ\iota_\varepsilon(s,\phi),\qquad\forall\, (s,\phi)\in I\times \mathbb S^1,\ \ve\in (0,\ve_0].
    \end{equation}
    Then, define the family  of lifts $\widetilde F^\ve$ by
    \begin{equation}
        \widetilde F^\varepsilon(s,\phi)
        =
        \left(
            F^\ve,
            P_{X}\circ\iota_\varepsilon,P_{Y}\circ\iota_\varepsilon
        \right)(s,\phi),\qquad\forall\, (s,\phi)\in I\times \mathbb S^1,\ \ve\in (0,\ve_0].
    \end{equation}
    Substituting \eqref{eq:rho_ve} into \eqref{eq:boundary_momenta} shows that the family of lifts also extends
    smoothly to $\varepsilon=0$, with
    \begin{equation}
    \label{eq:widetilde_F0}
        \widetilde F^0(s,\phi)
        :=
        \left(
            F^0(s,\phi),
            -\sin(s/2)\cos\phi,
            -\sin(s/2)\sin\phi
        \right),\qquad\forall\, (s,\phi)\in I\times \mathbb S^1.
    \end{equation}
    We now show that $\widetilde F^\ve$ is a Legendrian lift of $F^\ve$, for every $\ve\in[0,\ve_0]$. We fix $\varepsilon>0$ and we compute: 
    \begin{equation}
        \begin{aligned}
            (\widetilde F^\varepsilon)^*
            \bigl(d\ell-P_{X}dX-P_{Y}dY\bigr)
            &=
            \frac1\varepsilon\,
            \iota_\varepsilon^*
            \bigl(d\ell(s,\rho,\phi)-P_{X}(s,\rho,\phi)dX(s,\rho,\phi)-P_{Y}(s,\rho,\phi)\,dY(s,\rho,\phi)\bigr)\\
            &=
            \frac1\varepsilon\,
            \iota_\varepsilon^*(P_{Z}(s,\rho,\phi)\,dZ(s,\rho,\phi))
            =0,
        \end{aligned}
    \end{equation}
    where we have used \eqref{eq:leg12} and the fact that $\iota_\varepsilon^*dZ(s,\rho,\phi)=d(Z\circ \iota_\ve)(s,\phi)=d(\varepsilon^2)=0$ since $\ve$ is fixed. By smoothness of $\widetilde F^\ve$ at $\ve=0$, the same contact identity holds at $\varepsilon=0$. Finally, to conclude the proof, we need to show that $\widetilde F^\ve$ is an immersion for all sufficiently small $\ve$. Since $\{\widetilde F^\ve\}_{\ve\in[0,\ve_0]}$ is a smooth family of maps defined on a compact domain, it is sufficient to show that $\widetilde F^0$ is an immersion. Write $P_{X}^0,P_{Y}^0$ for the momentum components of $\widetilde F^0$, i.e.\ the last two components in \eqref{eq:widetilde_F0}. Then, we have
    \begin{equation}
    \label{eq:det_momenta}
        \det\left(
        \frac{\partial(P_{X}^0,P_{Y}^0)}
             {\partial(s,\phi)}\right)
        =
        \frac12\sin(s/2)\cos(s/2),\qquad\forall\, (s,\phi)\in I\times \mathbb S^1.
    \end{equation}
    After possibly shrinking $\delta$, the interval $[s_a-\delta,s_a+\delta]$ lies in $(0,2\pi)$, so the determinant in \eqref{eq:det_momenta} is nonzero except possibly at $s=\pi$.
    At $s=\pi$, an explicit computation shows that, for every $\phi\in \mathbb S^1$, 
    \[
        \partial_s\begin{pmatrix}
            P_X^0\\ P_Y^0
        \end{pmatrix}\bigg|_{s=\pi} = 0,\qquad \partial_\phi\begin{pmatrix}
            P_X^0\\ P_Y^0
        \end{pmatrix}\bigg|_{s=\pi} = \begin{pmatrix}
            \sin\phi\\ -\cos\phi
        \end{pmatrix}.
    \]
    On the other hand, an elementary computation shows that $q_c'(\pi)=0$, which implies
    $\partial_s\mathcal N(\pi,\phi)=0$, and hence
    \begin{equation}
        \left.
        \partial_s\left(\frac{s}{2\mathcal N(s,\phi)}\right)
        \right|_{s=\pi}
        =
        \frac1{2\mathcal N(\pi,\phi)}>0.
    \end{equation}
    This implies that $\widetilde F^0$ is an immersion and concludes the proof.
\end{proof}
The crucial fact which allows us to compute the cut time for a general surface, is that the limit front $F^0$ is stable in the class of Legendrian fronts. 

\begin{theorem}\label{thm:F0_stab}
	Let $N=I\times\mathbb S^1$, let $C=\{s_a\}\times\mathbb S^1$ and let $p_\pm=(s_a,\pm\tfrac{\pi}{2})$. Then, it holds
    \begin{equation}\label{eq:doub_F0_statement}
        \mathrm{Doub}(F^0,N)=C\setminus\{p_+,p_-\},
    \end{equation}
    where the points $p_{\pm}$ are swallowtails. Moreover, every point in $F^0(C\setminus\{p_+,p_-\})$ has exactly two distinct preimages in $N$, which are regular points whose image tangent spaces are transverse, i.e., for every $q\in F^0(C\setminus \{p_+,p_-\})$, there exist $z_1,z_2\in C$ such that
    \begin{equation}
         (F^0)^{-1}(q)=\{z_1,z_2\},\qquad\text{and}\qquad d_{z_1}F^0\bigl(T_{z_1}N\bigr)+ d_{z_2}F^0\bigl(T_{z_2}N\bigr)=T_q\mathbb R^3. 
    \end{equation}
\end{theorem}
\begin{proof}
    We begin by determining the type of critical points of $F^0$, proving in particular that the points $p_\pm=(s_a,\pm\tfrac{\pi}2)$ are swallowtails. Throughout the proof, we fix the following notation 
    \begin{equation}\label{eq:contact_condition_F0}
       F^0=(x,y,L),\qquad  \mathscr F_{a,b}=(E_{a,b},\ell_{a,b}).
    \end{equation}
    First, we note that, by \cref{lem:front_family}, $F^0$ is a Legendrian front. In particular, the momentum components of its lift, computed in \eqref{eq:front_lift_limit}, and the corresponding contact identity are
	\begin{equation}
    \label{eq:real_contact_condition_F0}
		(P_X^0,P_Y^0)=-\sin(s/2)(\cos\phi,\sin\phi),
		\qquad
		dL=P_X^0\,dx+P_Y^0\,dy.
	\end{equation}
    It follows that $\operatorname{rank}(d_{(s,\phi)}F^0)=\operatorname{rank}(d_{(s,\phi)}(x,y))$, and thus the critical points of $F^0$ are given by 
    \begin{equation}
    \label{eq:crit_F0}
        \mathrm{Crit}(F^0)=\left\{(s,\phi)\in I\times\mathbb S^1\,:\, \det\left(\frac{\partial(x,y)}{\partial(s,\phi)}\right)=0\right\}.
    \end{equation}
    To compute such determinant, note that, by \cref{l:belle_coords} and \eqref{eq:front_limit}, in coordinates $(s,\rho,\phi)$, we have the equality 
	\begin{equation}\label{eq:F0_scaling_relation}
		\mathscr F_{a,b}\left(s,\frac{\rho}{\mathcal N(s,\phi)},\phi\right)
		=\bigl(\rho x(s,\phi),\rho y(s,\phi),\rho^2,\rho L(s,\phi)\bigr),
		\qquad \forall\,\rho>0,\ (s,\phi)\in I\times\mathbb S^1.
	\end{equation}
	In particular, defining $\mathcal C(s,\rho,\phi):=(s,\rho/\mathcal N(s,\phi),\phi)$, from the above equation we obtain
    \begin{equation}\label{eq:pre_fact}
        \left.\det\left(\frac{\partial (E_{a,b}\circ \mathcal C)\left(s,\rho,\phi\right)}{\partial(s,\rho,\phi)}\right)\right|_{\rho=1}=\det\begin{pmatrix}
			\partial_sx(s,\phi) & x(s,\phi) & \partial_\phi x(s,\phi)\\
			\partial_sy(s,\phi) & y(s,\phi) & \partial_\phi y(s,\phi)\\
			0   & 2 & 0
		\end{pmatrix}=-2\det\left(\frac{\partial(x,y)}{\partial(s,\phi)}\right).
    \end{equation}
    On the other hand, the Jacobian determinant of the coordinate change $\mathcal C(s,\rho,\phi)$ is $1/\mathcal N(s,\phi)$. Therefore, using the chain rule to factorize the left hand side of \eqref{eq:pre_fact} we deduce 
    \begin{equation}\label{eq:det_xy_1}
        \det\left(\frac{\partial(x,y)}{\partial(s,\phi)}\right)=-\frac{1}{2\mathcal N(s,\phi)}\left.\det\left(\frac{\partial E_{a,b}\left(s,\rho,\phi\right)}{\partial(s,\rho,\phi)}\right)\right|_{\rho=\tfrac{1}{\mathcal N(s,\phi)}}.
    \end{equation}
    The Jacobian determinant of $E_{a,b}$ in coordinates $(s,\rho,\phi)$ is given in \cref{l:det_E_ab}:
	\begin{equation}\label{eq:F0_polar_jacobian}
		\det\left(\frac{\partial E_{a,b}(s,\rho,\phi)}{\partial(s,\rho,\phi)}\right)
		=-2\rho^3\Lambda(s,\phi),\qquad \Lambda(s,\phi)=p_b(s)m_a(s)\cos^2\phi+p_a(s)m_b(s)\sin^2\phi,
	\end{equation}
    where $m_c(s)=sp_c'(s)-p_c(s)$, for $c\in \{a,b\}$, cf.\ \cref{l:belle_coords} for the definition of the functions $p_c,q_c$. Combining the previous identities, we finally obtain
    \begin{equation}
    \label{eq:formula_det_dF0}
        \det\left(\frac{\partial(x,y)}{\partial(s,\phi)}\right)=  \frac{\Lambda(s,\phi)}{\mathcal N(s,\phi)^4},
    \end{equation}
    which, according to \eqref{eq:crit_F0}, implies that
	\begin{equation}
		\mathrm{Crit}(F^0)
		=\left\{(s,\phi)\in I\times\mathbb S^1:
		\Lambda(s,\phi)=0\right\}.
	\end{equation}
    We now study the equation $\Lambda(s,\phi)=0$. A direct evaluation gives
    \begin{equation}
    \label{eq:p_c(s_a)_and_some}
    p_a(s_a) = 0,\qquad p_b(s_a) = \frac{a-b}{\sqrt{1+a^2}},\qquad p_a'(s_a)= -\frac{\sqrt{1+a^2}}{2},\qquad p_b'(s_a)= -\frac{1+ab}{2\sqrt{1+a^2}}.
    \end{equation}
    Therefore, we obtain:    
	\begin{equation}\label{eq:Lambda_core}
		\Lambda(s_a,\phi)
		=-\frac{(a-b)s_a}{2}\cos^2\phi=0\qquad\iff\qquad\phi=\pm\frac{\pi}{2},
	\end{equation}
    since $a>b$ by the geometric stability assumption. This shows that $\Lambda(p_\pm)=0$. We want to apply the implicit function theorem at such points. A standard computation gives
    \begin{equation}\label{eq:Lambda_derivatives}
		\begin{aligned}
			\partial_s\Lambda(p_\pm)
			=p_a'(s_a)m_b(s_a)
            ,\qquad 
			\partial_\phi\Lambda(p_\pm)&=0,
			\qquad
			\partial_{\phi}^2\Lambda(p_\pm)=-(a-b)s_a.
		\end{aligned}
	\end{equation}
 We claim that $p_a'(s_a)m_b(s_a)>0$. Indeed, since $a>b$, $p_b(s)>0$ for $s\in [0,s_a]$ (cf.\,\cref{r:identities}), while
	\begin{equation}
		m_b(0)=-1,\qquad
		m_b'(s)=sp_b''(s)=-\frac{s}{4}p_b(s).
	\end{equation}
	Thus, $m_b'(s)<0$ and consequently $m_b(s_a)<0$. The assertion follows as $p_a'(s_a)<0$ by \eqref{eq:p_c(s_a)_and_some}. The implicit function theorem therefore shows that the critical set of $F^0$ near $p_\pm$ is parametrized by smooth curves $\gamma_\pm:[-t_0,t_0]\to I\times\mathbb S^1$, with $t_0>0$, defined by 
	\begin{equation}
		\gamma_\pm(t)
		:=\bigl(\sigma_\pm(t),\pm\pi/2+t\bigr),
		\qquad \sigma_\pm(0)=s_a.
	\end{equation}
    Moreover, by compactness of $\mathbb S^1$ and \eqref{eq:Lambda_core}, up to shrinking $\delta$, we may assume that $\Lambda(p)\neq 0$, for every $p\in N\setminus \mathrm{Im}(\gamma_\pm)$.
	Differentiating the defining identity $\Lambda(\sigma_\pm(t),\pm\pi/2+t)=0$ and using \eqref{eq:Lambda_derivatives}, we obtain
	\begin{equation}
		\sigma_\pm'(0)=0,\qquad
		\sigma_\pm''(0)
		=-\frac{\partial^2_{\phi}\Lambda(p_\pm)}
		{\partial_s\Lambda(p_\pm)}
		=\frac{(a-b)s_a}{p_a'(s_a)m_b(s_a)}.
	\end{equation}
	Consequently, the parametrization of the critical set of $F^0$ on $I\times \mathbb S^1$ is 
	\begin{equation}\label{eq:critical_curves_F0}
		\gamma_\pm(t)
		=\left(s_a+\kappa t^2+O(t^3),\pm\frac{\pi}{2}+t\right),
		\qquad
		\kappa:=\frac{(a-b)s_a}{2p_a'(s_a)m_b(s_a)}>0,
	\end{equation}
    for every $t\in [-t_0,t_0]$. Differentiating \eqref{eq:formula_det_dF0} with respect to $s$, and recalling that $\Lambda(p_\pm)=0$, we obtain 
    \begin{equation}
    \left.
    \partial_s\det\left(\frac{\partial(x,y)}{\partial(s,\phi)}\right)
    \right|_{p_\pm}
    =
    \frac{\partial_s\Lambda(p_\pm)}
         {\mathcal N(p_\pm)^4}
    \neq0.
\end{equation}
Since $\partial_s\Lambda(p_\pm)\neq 0$, $p_\pm$ are non-degenerate singular points in light of \cref{rmk:non_deg_sing}. In particular the kernel of $dF^0$ is necessarily one-dimensional along $\gamma_\pm(t)$ (see \cref{lem:crit_recon}), and thus spanned by a non-vanishing vector field $\eta_\pm(t)$. We now compute these vector fields. Recall that $L(s,\phi)=s/(2\mathcal N(s,\phi))>0$. Then, the quotient $y/L$ and its differential are
	\begin{equation}
		\frac{y}{L}=\frac{2p_b(s)\sin\phi}{s},
		\qquad
		d\left(\frac{y}{L}\right)
		=\frac{2}{s^2}
		\bigl(m_b(s)\sin\phi\,ds
		+sp_b(s)\cos\phi\,d\phi\bigr).
	\end{equation}
	Since $m_b(s_a)\sin(\pm\tfrac{\pi}2)=\pm m_b(s_a)\neq0$, cf.\ \eqref{eq:p_c(s_a)_and_some}, the latter differential is nonzero near
	$p_\pm$, meaning that $\ker d(y/L)$ is one-dimensional along the critical curves. Moreover, it can also be written as a linear combination of $dy$ and
	$dL$, so that
	$\ker dF^0\subset\ker d(y/L)$.
	Since, along the critical curves, both kernels are one-dimensional,
	they must coincide. It follows that 
	\begin{equation}\label{eq:null_direction_F0}
		\eta_\pm(t)
		=\left.\left(-\frac{sp_b(s)}{m_b(s)}\cot\phi\,\partial_s
		+\partial_\phi\right)\right|_{\gamma_\pm(t)}.
	\end{equation}
    Now according to \cref{lem:crit_recon}, in order to understand the type of critical points composing the critical set, it is sufficient to compute the determinant
\begin{equation}
    \Delta_\pm(t):=\det\left(\dot\gamma_\pm(t),\eta_\pm(t)\right).
\end{equation}
Using equation \eqref{eq:critical_curves_F0} and observing that 
	\begin{equation}
		\frac{s_ap_b(s_a)}{m_b(s_a)}=-\kappa,
		\qquad
		\cot(\pm\pi/2+t)=-t+O(t^3),
	\end{equation}
	we may compute a Taylor expansion of $\eta_\pm(t)$ near $t=0$ finding 
	\begin{equation}
		\eta_{\pm}(t)
		=\bigl(-\kappa t+O(t^2)\bigr)\partial_s+\partial_\phi,\qquad \dot\gamma_\pm(t)
		=\bigl(2\kappa t+O(t^2)\bigr)\partial_s+\partial_\phi.
	\end{equation}
Computing the Taylor first-order expansion of $\Delta_\pm(t)$ yields 
\begin{equation}
    \Delta_\pm(t)=3\kappa t+O(t^2).
\end{equation}
Since $\kappa\neq 0$, \cref{lem:crit_recon} implies that the points $\gamma_{\pm}(t)$ are cuspidal edges for $t\neq 0$ and swallowtails for $t=0$. In particular the points $p_\pm$ are swallowtails.

We now prove the assertion on double points. Suppose there are distinct pairs $(s,\phi),(t,\psi)\in I\times\mathbb S^1$ such that $F^0(s,\phi)=F^0(t,\psi)$. Equation \eqref{eq:F0_scaling_relation} implies that 
\[ 
    \mathscr F_{a,b}\left(s,\frac{1}{\mathcal N(s,\phi)},\phi\right)=\mathscr F_{a,b}\left(t,\frac{1}{\mathcal N(t,\psi)},\psi\right)
\]
and then \cref{p:seipuntotre} implies $s=t$. Equality of the length components, i.e., $L(s,\phi)=L(t,\psi)$, then gives
	$\mathcal N(s,\phi)=\mathcal N(s,\psi)$, and equality of the remaining two
	components $(x(s,\phi),y(s,\phi))=(x(t,\psi),y(t,\psi))$ yields
	\begin{equation}
		p_a(s)(\cos\phi-\cos\psi)=0,
		\qquad
		p_b(s)(\sin\phi-\sin\psi)=0.
	\end{equation}
	Thus, since $\phi\neq\psi$ and, for $s\in I$, $p_b(s)>0$ while $p_a(s)=0$ if and only if $s=s_a$, we must have $s=s_a$ and
	$\psi=\pi-\phi$ modulo $2\pi$. Conversely, such pairs have
	the same image under $F^0$, and they are distinct exactly when
	$\phi\neq\pm\pi/2$. This proves equation \eqref{eq:doub_F0_statement} and that every point in $F^0(\mathrm{Doub}(F^0,N))$ has exactly two preimages, in particular it shows that, for $\phi\neq\psi$, one has
    \begin{equation}
        F^0(s_a,\phi)=F^0(s_a,\psi)\qquad \iff\qquad\psi=\pi-\phi.
    \end{equation}
    Note that these distinct preimages are regular points by the first part of the proof (cf.\ \eqref{eq:Lambda_core}), because they are distinct if and only if $\phi\neq \pm\tfrac{\pi}{2}$. It remains to show 
    \begin{equation}
        d_{(s_a,\phi)}F^0\bigl(T_{(s_a,\phi)}N\bigr)+ d_{(s_a,\pi-\phi)}F^0\bigl(T_{(s_a,\pi-\phi)}N\bigr)=T_q\mathbb R^3,
    \end{equation}
    where $q=F^0(s_a,\phi)=F^0(s_a,\pi-\phi)$ and $\phi\neq\pm\tfrac\pi 2$. We are going to show that the (Euclidean) normals of the two planes are linearly independent. By \eqref{eq:real_contact_condition_F0}, for every $(s,\phi)\in N$, the covector $(\sin(s/2)\cos\phi,\sin(s/2)\sin\phi,1\bigr)\in T^*\R^3$ annihilates the tangent plane $d_{(s,\phi)}F^0(T_{(s,\phi)}N)\subset T\R^3$. Thus the normal direction in $\mathbb R^3$ to the latter plane is spanned by  
    \begin{equation}\label{eq:F0_normal_from_lift}
	\nu_{(s,\phi)}
	=\bigl(\sin(s/2)\cos\phi,\sin(s/2)\sin\phi,1\bigr).
\end{equation}
Evaluating the normal at $(s_a,\phi)$ and $(s_a,\pi-\phi)$ we find 
	\begin{equation}
		\nu_{(s_a,\phi)}
		=\left(\frac{\cos\phi}{\sqrt{1+a^2}},
		\frac{\sin\phi}{\sqrt{1+a^2}},1\right),\qquad
		\nu_{(s_a,\pi-\phi)}
		=\left(-\frac{\cos\phi}{\sqrt{1+a^2}},
		\frac{\sin\phi}{\sqrt{1+a^2}},1\right),
	\end{equation}
    which are linearly independent for every $\phi\neq \pm\tfrac{\pi}{2}$, as claimed.
\end{proof}

We conclude the section by showing how the latter theorem, combined with classical stability theory, allows to compute the double point locus of $F^\ve$, for $\ve$ sufficiently small, as a smooth perturbation of the double locus for the model $F^0$. We then use \cref{l:union} to obtain the self intersections of $\mathscr F$. 
\begin{corollary}\label{cor:F_ve_doub}
There exist $\ve_0>0$ and a smooth
function $f :[0,\ve_0]\times\mathbb S^1\to(s_a-\delta,s_a+\delta)$,
with $f(0,\phi)=s_a$, such that, setting
$V=\Psi_Z^{-1}(I\times(0,\ve_0]\times\mathbb S^1)$, we have
\begin{equation}\label{eq:doubF}
\overline{\mathrm{Doub}(\mathscr F,V)}
=\left\{\Psi_Z^{-1}(f(\ve,\phi),\ve,\phi):
0<\ve\le\ve_0,\ \phi\in\mathbb S^1\right\}.
\end{equation}
\end{corollary}
\begin{proof}
By \cref{thm:F0_stab}, the assumptions of \cref{lem:stability} are verified for the Legendrian front $F^0$. Thus, applying \cref{lem:stability} and \cref{cor:stability} to the smooth family of Legendrian fronts $F^\ve$, we find $\ve_0>0$, a smooth map $f:[0,\ve_0]\times \mathbb S^1\to I$ and two smooth curves $\phi_\pm:[0,\ve_0]\to \mathbb S^1$ such that 
\begin{equation}\label{eq:conclusion_front_graph}
\mathrm{Doub}(F^\ve,I\times\mathbb S^1)
=\{(f(\ve,\phi),\phi):\phi\in\mathbb S^1\setminus{\phi_{\pm}(\ve)}\},\qquad
f(0,\phi)=s_a,\qquad \phi_\pm(0)=\pm\frac{\pi}{2}.
\end{equation}
Combining \cref{l:union} with \eqref{eq:conclusion_front_graph} and taking the closure in $V$ yield the result.
\end{proof}

\section{Computation of the cut time. Proof of Theorem~\ref{thm:intro-cut-time-stability}}
\label{sec:proof_of_cut_time_general}
Now that we have an explicit description of $\overline{\mathrm{Doub}(\mathscr F,V)}$ in terms of the double locus of the perturbation $F^\ve$ of the reduced front $F^0$, we exploit \cref{thm:t_cut_front} to compute the cut time: we intersect $\overline{\mathrm{Doub}(\mathscr F,V)}$ with the fibers of the projection map $\pi$, and compute the infimum of the length function along this intersection.
\begin{lemma}\label{lem:tau}
There exists $R>0$ such that $I\times B_R^\times\subset V$. Defining
\begin{equation}\label{eq:tau_cut_lambda_proof}
T_V(q):=\inf\{\ell(z):z\in\overline{\mathrm{Doub}(\mathscr F,V)},\
                       \pi(z)=q\},\qquad \forall q\in B_{R}^\times,
\end{equation}
then, the infimum is attained at a unique parameter $s=\sigma(\rho,\theta)$, where
$$\sigma:[0,R]\times\mathbb S^1\to(s_a-\delta,s_a+\delta)$$
is a smooth function satisfying $\sigma(0,\vartheta)=s_a$. Moreover,
\begin{equation}\label{eq:tau_expansion}
T_V(\rho,\vartheta)=\frac{\sigma(\rho,\vartheta)}{h_0(\rho,\vartheta)}
=\frac{s_a}{2}\rho+O(\rho^2),\qquad\text{as }\rho\to0^+.
\end{equation}
\end{lemma}
\begin{proof}
By \cref{lem:Z}, the function
$\ve(s,\rho,\phi):=\sqrt{Z(s,\rho,\phi)}$ is smooth at
$\rho=0$, where it vanishes. By compactness of $I\times\mathbb S^1$, we may choose $R>0$ so small that $I\times B_R^\times\subset V$. Since $\Psi_Z(s,\rho,\phi)=(s,\ve(s,\rho,\phi),\phi)$,
\cref{cor:F_ve_doub} implies that a point $(s,\rho,\phi)\in V$ belongs to $\overline{\mathrm{Doub}(\mathscr F,V)}$ if and only if $s=f(\ve(s,\rho,\phi),\phi)$. Thus,
\begin{equation}\label{eq:conclusion_implicit_graph}
	\overline{\mathrm{Doub}(\mathscr F,V)}
	=
	\left\{(s,\rho,\phi)\in V:
	s-f\bigl(\ve(s,\rho,\phi),\phi\bigr)=0
	\right\}.
\end{equation}
We define the function
\begin{equation}
    \mathcal L:V\to\R,\qquad \mathcal L(s,\rho,\vartheta)
    :=s-f\bigl(\ve(s,\rho,\vartheta+s/2),\vartheta+s/2\bigr).
\end{equation}
Then, since $\pi(s,\rho,\phi)=(\rho,\phi-\tfrac s2)$ by \cref{rmk:pi}, the intersection $\overline{\mathrm{Doub}(\mathscr F,V)}\cap\pi^{-1}(\rho,\vartheta)$ is characterized by the equation $\mathcal L(s,\rho,\vartheta)=0$. Since $\ve(s,0,\phi)=0$ and $f(0,\phi)=s_a$, we have
\begin{equation}
	\mathcal L(s,0,\vartheta)=s-s_a,
	\qquad
	\partial_s\mathcal L(s,0,\vartheta)=1.
\end{equation}
It follows that, there exists $R>0$ sufficiently small such that $\partial_s\mathcal L>0$ on
$I\times[0,R]\times\mathbb S^1$, and $\mathcal L(\cdot,\rho,\vartheta)$ has opposite signs at the endpoints of $I$, for every $(\rho,\vartheta)\in[0,R]\times\mathbb S^1$.
Thus, for every $(\rho,\vartheta)\in [0,R]\times\mathbb S^1$, $\mathcal L(s,\rho,\vartheta)=0$ has a unique solution $s=\sigma(\rho,\vartheta)$.
The non-vanishing $s$-derivative allows us to apply the implicit function theorem, deducing smoothness of $\sigma$,
including at $\rho=0$, where $\sigma(0,\vartheta)=s_a$. Therefore the intersection $\overline{\mathrm{Doub}(\mathscr F,V)}\cap \pi^{-1}(\rho,\vartheta)$ consists of a single point, i.e.,
\begin{equation}
    \overline{\mathrm{Doub}(\mathscr F,V)}\cap\pi^{-1}(\rho,\vartheta)=\left\{\left(
    \sigma(\rho,\vartheta),\rho,
    \vartheta+\frac{\sigma(\rho,\vartheta)}{2}
    \right)\right\}.
\end{equation}
Hence, evaluating the length function $\ell$ at such point, using \cref{rmk:pi}, gives
\begin{equation}
    T_V(\rho,\vartheta)=\frac{\sigma(\rho,\vartheta)}{h_0(\rho,\vartheta)}.
\end{equation}
The Taylor expansion \eqref{eq:tau_expansion} follows from expansions: 
\[ \sigma(\rho,\vartheta)=s_a+O(\rho)\qquad\text{and}\qquad \frac 1{h_0(\rho,\vartheta)}=\frac{\rho}{2}+O(\rho^2),
\]
as $\rho\to 0^+$. This concludes the proof.
\end{proof}
We can now prove the main result showing that $T_V$ coincides with the cut time, sufficiently close to a characteristic point.

\begin{proof}[Proof of \cref{thm:intro-cut-time-stability}]
By \cref{thm:t_cut_front}, and since $\mathrm{Doub}(\mathscr F,V)\subset\mathrm{Doub}(\mathscr F)$, we have
\begin{equation}\label{eq:t_cut_lambda_proof}
\begin{aligned}
t_\cut(q)=\inf\{\ell(z):z\in\overline{\mathrm{Doub}(\mathscr F)},\ \pi(z)=q\}\leq T_V(q),\qquad q\in B_R^\times.
\end{aligned}
\end{equation}
We prove that equality holds in a smaller punctured ball $B_r^\times$. First, there exists $0<r<R$ such that
\begin{equation}\label{eq:conclusion_distance_gap}
2\sup_{q\in B_r}T_V(q)
<d_{SR}(B_r,S\setminus B_R).
\end{equation}
Indeed, the left-hand side tends to zero as $r\to0^+$ by \eqref{eq:tau_expansion}, while the right-hand side stays bounded below by a positive constant. Second, fix $q\in B_r^\times$ and set $t_0:=t_\cut(q)$. By \cref{rmk:foc_after_max} there exist minimizing normal geodesics from $S$, based at distinct points $q_n, q_n'\in S$, with the same endpoint and length $t_n$, such that
\begin{equation}
(t_n,q_n)\to(t_0,q),\qquad t_n=t_\cut(q_n)=t_\cut(q_n').
\end{equation}
Setting $s_n=t_n h_0(q_n)$ and $s_n'=t_nh_0(q_n')$, we claim that, eventually
\begin{equation}\label{eq:in_V}
	\{(s_n,q_n)\},\{(s_n',q_n')\}\subset V.
\end{equation}
Indeed, the triangle inequality, together with the inequality \eqref{eq:t_cut_lambda_proof}, gives
\begin{equation}
\label{eq:interchangeable}
d_{SR}(q,q_n')\le d_{SR}(q,q_n)+2t_n
\xrightarrow{n\to\infty} 2t_0\le2T_V(q).
\end{equation}
Then, by \eqref{eq:conclusion_distance_gap}, for $n$ large, $d_{SR}(q,q_n')<d_{SR}(B_r,S\setminus B_R)$ and therefore $q_n'\in B_R^\times$. Also, $q_n\in B_R^\times$ because $q_n\to q$ by construction. Moreover, combining \eqref{eq:t_cut_lambda_proof} with \eqref{eq:tau_expansion},we obtain
\begin{equation}
\begin{aligned}
s_n=h_0(q_n)t_n=h_0(q_n)t_\cut(q_n)
 \leq h_0(q_n)T_V(q_n)=\sigma(q_n)<s_a+\delta,
\end{aligned}
\end{equation}
and the same exact inequality holds for $s_n'$. Thus $(s_n,q_n),(s_n',q_n')\in[0,s_a+\delta]\times B_R^\times$.
\cref{lem:singular_pos} finally implies that $s_n,s_n'\in I$, therefore 
\begin{equation}
		\{(s_n,q_n)\},\{(s_n',q_n')\}\subset I\times B_R^\times.
\end{equation}
By \cref{lem:tau}, $R$ can be chosen sufficiently small so that $I\times B_R^\times\subset V$, hence \eqref{eq:in_V} holds for large $n$. Since, $(s_n,q_n)$ and $(s_n',q_n')$ are double points, for large $n$ it holds
\[
(s_n,q_n),(s_n',q_n')\in\mathrm{Doub}(\mathscr F,V).
\]
The definition of $T_V$ then gives
\begin{equation}\label{eq:other_ineq}
T_V(q)\leq \lim_{n} \ell(s_n,q_n)=\lim_{n}\ell(t_nh_0(q_n),q_n)=\lim_{n}t_n=t_\cut(q).
\end{equation}
Since $q$ was arbitrarily chosen in $B_r^\times\subset B_R^\times$, the two inequalities \eqref{eq:t_cut_lambda_proof} and \eqref{eq:other_ineq} combined give 
\begin{equation}
	t_\cut(q)=T_V(q),\qquad \forall q\in B_r^\times.
\end{equation}
Expressing the result in polar coordinates $(\rho,\vartheta)$ we find that that the cut time restricted to the punctured ball $B_r^\times \simeq (0,r]\times \mathbb S^1$ satisfies 
\begin{equation}
	t_\cut(\rho,\vartheta)=T_V(\rho,\vartheta),\qquad \forall \,(\rho,\vartheta)\in (0,r]\times \mathbb S^1.
\end{equation}
In particular by \cref{lem:tau} the cut time is smooth on such domain and smoothly extends to $\rho=0$. Moreover \cref{lem:tau} also gives 
\begin{equation}
	t_\cut(\rho,\vartheta)=\frac{\rho s_a}{2}+\rho^2 R_\cut(\rho,\vartheta),
\end{equation}
where $s_a=\pi-2\arctan(a)$ and $R_\cut \in C^\infty(\R^+ \times \mathbb S^1)$. This concludes the proof, observing that, under the assumption $a>b$, $c_{a,b}=s_a/2$. 
\end{proof}
\part{Appendix}
\label{part:appendix}

\appendix

\section{Horizontal plane and homogeneous surfaces} \label{a:f}  
\subsection{Direct computations for the horizontal plane} We derive \eqref{eq:intro-piano} by means of a direct computation. First we combine \eqref{eq:intro-tube-integral} with our knowledge of the cut time and Jacobian \eqref{eq:intro-Jac-plane} to obtain 
\begin{equation}
	\begin{aligned}
		\mathrm{vol}(\mathcal T_{B_R}(\ve))
		&=\int_0^{2\pi}\int_0^R
		\int_0^{\min\{\ve,t_\cut(r,\theta)\}}
		J_{S_0}(t,r,\theta)\,dt\,dr\,d\theta\\
		&=\frac{\pi}{2}\int_0^R
		\int_0^{\min\{\ve,\pi r/2\}}
		\left[
		r^2\left(1+\cos\left(\frac{2t}{r}\right)\right)
		+tr\sin\left(\frac{2t}{r}\right)
		\right]dt\,dr.
	\end{aligned}
\end{equation}
For $0<\ve<\pi R/2$, we may exchange the order of integration in $t$ and $r$ finding
\begin{equation}
	\mathrm{vol}(\mathcal T_{B_R}(\ve))=\frac{\pi}{2}\int_0^\ve I(t)\,dt,
	\qquad
	I(t):=\int_{2t/\pi}^{R}
	\left[
	r^2\left(1+\cos\left(\frac{2t}{r}\right)\right)
	+tr\sin\left(\frac{2t}{r}\right)
	\right]dr.
\end{equation}
For $t>0$, two integrations by parts in $r$ yield
\begin{equation}
	\begin{aligned}
		I(t)=
		\frac{R^3}{3}
		\left(1+\cos\left(\frac{2t}{R}\right)\right)
		+\frac{tR^2}{6}\sin\left(\frac{2t}{R}\right)
		+\frac{t^2}{3}\int_{2t/\pi}^{R}
		\cos\left(\frac{2t}{r}\right)\,dr.
	\end{aligned}
\end{equation}
To evaluate the remaining integral, we define
$\operatorname{Si}(x):=\int_0^x(\sin u)/u\,du$ and we observe that
\begin{equation}
	\frac{d}{dr}\left[
	r\cos\left(\frac{2t}{r}\right)
	+2t\operatorname{Si}\left(\frac{2t}{r}\right)
	\right]
	=\cos\left(\frac{2t}{r}\right).
\end{equation}
Consequently,
\begin{equation}
	\int_{2t/\pi}^{R}\cos\left(\frac{2t}{r}\right)\,dr
	=
	R\cos\left(\frac{2t}{R}\right)
	+2t\operatorname{Si}\left(\frac{2t}{R}\right)
	+2t\left(\frac1\pi-\operatorname{Si}(\pi)\right).
\end{equation}
Substituting into the expression for $I(t)$ and expanding at
$t=0$, using $\operatorname{Si}(x)=x+O(x^3)$, gives
\begin{equation}
	\begin{aligned}
		I(t)=
		\frac{2R^3}{3}
		+\frac{2t^3}{3}
		\left(\frac1\pi-\operatorname{Si}(\pi)\right)
		+O\left(\frac{t^4}{R}\right).
	\end{aligned}
\end{equation}
Finally, integrating with respect to $t$, we obtain
\begin{equation}
	\mathrm{vol}(\mathcal T_{B_R}(\ve))
	=
	\frac{\pi R^3}{3}\ve
	+\frac{1-\pi\operatorname{Si}(\pi)}{12}\ve^4
	+O\left(\frac{\ve^5}{R}\right),
	\qquad \ve\to0^+.
\end{equation}
\subsection{The curvature forms of homogeneous surfaces}\label{ssec:expl_curve_forms}
Here we explicitly compute the curvature forms introduced in \cref{defi:curvature_forms} for homogeneous surfaces $S_{a,b}$, $1+ab\neq 0$, in the coordinates $(\rho,\vartheta)$ of \eqref{eq:x-y_polar}. First observe that, as a consequence of \cref{cor:curv_forms1} it holds 
\begin{equation}\label{eq:hom_mathcal_J}
    \sigma_{S_{a,b}}^{(n)}=\rho^{2-n}\partial_s^{n}J_{a,b}(s,1,\vartheta)|_{s=0}d\rho\wedge d\vartheta.
\end{equation}
\begin{lemma}\label{l:det_E_ab}
    In the coordinates of \cref{rmk:pi} it holds 
    \begin{equation}\label{eq:det_E_ab_Lambda}
        \det\left(\frac{\partial E_{a,b}(s,\rho,\phi)}{\partial(s,\rho,\phi)}\right)=-2\rho^3\Lambda(s,\phi),\qquad \Lambda(s,\phi):=p_b(s)m_a(s)\cos^2\phi+p_a(s)m_b(s)\sin^2\phi,
	\end{equation}
    where $m_c(s):=sp_c'(s)-p_c(s)$, for $c\in \{a,b\}$, and $p_c(s)$ is defined in \cref{l:belle_coords}.
\end{lemma}
\begin{proof}
	By the first identity of \eqref{eq:identities_remark} and standard trigonometric identities, we have
	\[
	q_c(s)=s+2\sin(s/2)p_c(s),\qquad
	\cos(s/2)p_c(s)-2\sin(s/2)p_c'(s)=1.
	\]
	Consequently, a direct computation yields 
	\[
	p_c(s) q_c'(s)-2p_c'(s)q_c(s)=2p_c(s)-2s p_c'(s)=-2m_c(s).
	\]
	Therefore, differentiating \eqref{eq:E_coords} and computing the Jacobian determinant gives
	\begin{equation}
		\begin{aligned}
			\det\left(\frac{\partial E_{a,b}(s,u,v)}{\partial(s,u,v)}\right)
			&=p_b(s)\bigl(p_a(s)q_a'(s)-2p_a'(s)q_a(s)\bigr)u^2
			+p_a(s)\bigl(p_b(s)q_b'(s)-2p_b'(s)q_b(s)\bigr)v^2\\
            &=-2\bigl(p_b(s)m_a(s) u^2+p_a(s)m_b(s) v^2\bigr).
		\end{aligned}
	\end{equation}
	Finally, evaluating the latter equation at $(u,v)=(\rho\cos\phi,\rho\sin\phi)$, and using the fact that the change of variables $(u,v)=(\rho\cos\phi,\rho\sin\phi)$ has
	Jacobian determinant $\rho$, we prove \eqref{eq:det_E_ab_Lambda}.
\end{proof}
\begin{lemma}\label{l:det_exp_perp}
	The normalized Jacobian of $S_{a,b}$ is
	\begin{equation}\label{eq:hom_mathcal_J_explicit}
		\begin{aligned}
			J_{a,b}(s,1,\vartheta)
			=\frac{1}{4(1+ab)}\Bigl[\,&(1+ab)
			+(1-ab)\bigl(\cos(2s)+s\sin(2s)\bigr)\\
			&+(a+b)\bigl(s\cos(2s)-\sin(2s)\bigr)
			+s(a-b)\cos(2\vartheta+2s)\Bigr].
		\end{aligned}
	\end{equation}
\end{lemma}

\begin{proof}
	By \cref{ss:change,rmk:pi}, we have the identity
	\[
	\exp^\perp_{S_{a,b}}(t,\rho,\vartheta)
	=\psi^{-1}\circ E_{a,b}
	\left(\frac{2t}{\rho},\rho,\vartheta+\frac{t}{\rho}\right).
	\]
	The map $\psi$ has Jacobian determinant $8(1+ab)$,
	whereas the map $(t,\rho,\vartheta)\mapsto(\tfrac{2t}{\rho},\rho,\vartheta+\tfrac{t}{\rho})$
	has Jacobian determinant $2/\rho$.
	Thus \cref{l:det_E_ab} and the chain rule give
	\[
	J_{S_{a,b}}(t,\rho,\vartheta)
	=-\frac{\rho^2}{2(1+ab)}
	\Lambda\left(\frac{2t}{\rho},\vartheta+\frac{t}{\rho}\right).
	\]
	Therefore, equation \eqref{eq:hom_mathcal_J} implies that 
	\begin{equation}\label{eq:mathcal_J_Lambda}
		J_{a,b}(s,1,\vartheta)
		=-\frac{\Lambda(2s,\vartheta+s)}{2(1+ab)}.
	\end{equation}
	Substituting $m_c(s)=s p_c'(s)-p_c(s)$, and the expression of
	$p_c(s)$ from \eqref{eq:def_p-q_c}, into \eqref{eq:det_E_ab_Lambda} yields
	\eqref{eq:hom_mathcal_J_explicit}.
\end{proof}
\begin{prop}\label{prop:hom_curv_form}
    Let $(\rho,\vartheta)$ be the coordinates on $S_{a,b}$ introduced in \eqref{eq:x-y_polar}. Then, for every $n\in\mathbb N$, the $n$-th curvature form of the homogeneous surface $S_{a,b}$ has the expression 
\begin{equation}
    \sigma^{(n)}_{S_{a,b}}
    =\rho^{2-n}\Big(
    A_n\cos(2\vartheta)
    +B_n\sin(2\vartheta)
    +C_n\Big)\,d\rho\wedge d\vartheta,
\end{equation}
where, for $n\geq1$, we have defined 
\begin{equation}\label{eq:A_n_B_n}
    (A_n,B_n,C_n)
    =\frac{(-1)^{\lfloor n/2\rfloor}2^{n-3}}{1+ab}
    \begin{cases}
        \bigl(0,n(a-b),(2-n)(1-ab)\bigr),
        & n\ \text{even},\\[2mm]
        \bigl(n(a-b),0,(n-2)(a+b)\bigr),
        & n\ \text{odd}.
    \end{cases}
\end{equation}
For $n=0$, the coefficients are $(A_0,B_0,C_0)=\left(0,0,\tfrac{1}{2(1+ab)}\right)$.
\end{prop}
\begin{proof} Substituting the identity $\cos(2\vartheta+2s)=\cos(2\vartheta)\cos(2s)-\sin(2\vartheta)\sin(2s)$ into \eqref{eq:hom_mathcal_J_explicit}, we see that the Jacobian $J_{a,b}(s,1,\vartheta)$ splits as a sum of three functions 
\begin{equation}\label{eq:A,B,C_dec}
    J_{a,b}(s,1,\vartheta)
			=\frac{1}{4(1+ab)}\bigl(A(s)\cos(2\vartheta)+B(s)\sin(2\vartheta)+C(s)\bigr),
\end{equation}
where the functions $A,B,C$ are given by 
\begin{equation}
\begin{aligned}
     A(s)=(a-b)s\cos(2s),\qquad B(s)=-(a-b)s\sin(2s),\qquad C(s)=C^{even}(s)+C^{odd}(s),
\end{aligned}
\end{equation}
where $C^{even},C^{odd}$ are nothing but the even and odd parts of $C(s)$, i.e. 
\begin{equation}
    C^{even}(s)=(1+ab)+(1-ab)\bigl(\cos(2s)+s\sin(2s)\bigr),\qquad C^{odd}(s)=(a+b)\bigl(s\cos(2s)-\sin(2s)\bigr).
\end{equation}
Thus, by \eqref{eq:hom_mathcal_J} and direct differentiation of \eqref{eq:A,B,C_dec}, the statements of the propositions hold with 
\begin{equation}
    A_n:=\frac{A^{(n)}(0)}{4(1+ab)},\qquad B_n:=\frac{B^{(n)}(0)}{4(1+ab)},\qquad C_n:=\frac{C^{(n)}(0)}{4(1+ab)}.
\end{equation}
The function $A$ is odd, while $B$ is even. Then, for $n$ even (resp. odd) only the $n$-th derivative of $B$ and $C^{even}$ (resp. $A$ and $C^{odd}$) are non vanishing. Taking full Taylor expansions of the elementary functions $A,B,C^{even},C^{odd}$ and extracting the $n$-th derivatives yields the claimed expressions of $A_n,B_n,C_n$.
\end{proof}
\begin{example}\label{ex:hom_curv_forms}
    The first five curvature forms of $S_{a,b}$ are
    \begin{equation}
        \begin{aligned}
            \sigma^{(0)}_{S_{a,b}}
            &=\frac{\rho^2}{2(1+ab)}
            \,d\rho\wedge d\vartheta,\\
            \sigma^{(1)}_{S_{a,b}}
            &=\frac{\rho}{4(1+ab)}
            \bigl((a-b)\cos(2\vartheta)-(a+b)\bigr)
            \,d\rho\wedge d\vartheta,\\
            \sigma^{(2)}_{S_{a,b}}
            &=-\frac{a-b}{1+ab}\sin(2\vartheta)
            \,d\rho\wedge d\vartheta,\\
            \sigma^{(3)}_{S_{a,b}}
            &=-\frac{3(a-b)\cos(2\vartheta)+(a+b)}
            {(1+ab)\rho}
            \,d\rho\wedge d\vartheta,\\
            \sigma^{(4)}_{S_{a,b}}
            &=\frac{4\bigl(2(a-b)\sin(2\vartheta)+(ab-1)\bigr)}
            {(1+ab)\rho^2}
            \,d\rho\wedge d\vartheta.
        \end{aligned}
    \end{equation}
\end{example}
\section{On logarithmic expansions. Proof of Lemma~\ref{l:integrato}} 
\label{app:noname}

The goal of this section is to prove \cref{l:integrato}. We first observe that whenever a function has an expansion involving logarithmic coefficients, these are uniquely determined. 

 \begin{lemma}\label{lem:log_expansion_intrinsic}
    	Let $F:(0,1)\to\mathbb R$ be a function. Assume that there exist $k_0\in\mathbb N$ and coefficients $\{\alpha_n,\beta_n\}_{n\geq -k_0}\subset\R$ such that 
    	\begin{equation}\label{eq:log_laurent_expansion}
    		F(\ve)=\sum_{n=-k_0}^{N}
    		\ve^n(\alpha_n+\beta_n\log\ve)+o(\ve^N),
    		\qquad \ve\to0^+,
    	\end{equation}
    	for every integer $N\geq-k_0$. Then the coefficients
    	$\alpha_n,\beta_n$ are uniquely determined by $F$.
    \end{lemma}
    \begin{proof}
    	Suppose two families of coefficients $\{\alpha_n,\beta_n\}$ and $\{\widetilde\alpha_n,\widetilde\beta_n\}$  satisfy \eqref{eq:log_laurent_expansion}, and let $m$ be the smallest index such that $(\alpha_m,\beta_m)\neq (\widetilde\alpha_m,\widetilde\beta_m) $. Then, the function $F$ satisfies \eqref{eq:log_laurent_expansion} 
    	with $N=m$, for both families of coefficients, hence taking the difference of the two expansions and dividing by $\ve^m$ gives
    	\begin{equation}
    	    0=\frac{1}{\ve^m}\left(\sum_{n=-k_0}^m\ve^n\bigl((\alpha_n-\widetilde \alpha_n)+(\beta_n-\widetilde \beta_n)\log\ve\bigr)+o(\ve^m)\right)=(\alpha_m-\widetilde\alpha_m)
    	+(\beta_m-\widetilde\beta_m)\log\ve+\frac{o(\ve^m)}{\ve^m},
    	\end{equation}
        where the last equality follows by our choice of $m$. It follows that 
        $$
    	(\alpha_m-\widetilde\alpha_m)
    	+(\beta_m-\widetilde\beta_m)\log\ve=o(1).$$
    	Since $\log\ve\to-\infty$, it must be
    	$\beta_m=\widetilde\beta_m$, and then
    	$\alpha_m=\widetilde\alpha_m$, contradicting the choice of $m$.
    \end{proof}

The previous lemma shows that the statement \cref{l:integrato} is well-posed. Secondly, we prove an auxiliary result that is needed in the proof of \cref{l:integrato}. 

\begin{lemma}\label{l:euge}
        Let $F:\mathbb R^2\to\mathbb R$ be a smooth function. Then, there exist $\{f_k,g_k\}_{k=0}^\infty\subset C^{\infty}(\mathbb R)$ such that, for every $N\in\N$, there exists 
        $R_{N+1}\in C^\infty(\mathbb R^2)$
       such that \begin{equation}\label{eq:euge}
            F(x,y)=\sum_{k=0}^N (xy)^k(f_k(y)+xg_k(x))+ (xy)^{N+1}R_{N+1}(x,y).
        \end{equation}
    \end{lemma}
    
\begin{proof}
        We proceed by induction, first proving the statement for $N=0$. Since $F$ is smooth, both $F$ and its partial derivatives are absolutely continuous and we can write
        \begin{equation}
        \begin{aligned}
            F(x,y)-F(0,y)-F(x,0)+F(0,0)
            &=\int_0^x\partial_xF(t,y)dt-\int_0^x\partial_xF(t,0)dt \\
            &=\int_0^x\int_0^y\partial_y\partial_x F(t,s)dsdt=xy\underbrace{\int_0^1\int_0^1\partial_y\partial_x F(xt,ys)dsdt}_{R_1(x,y)}.
        \end{aligned}
        \end{equation}
        Setting $f_0(y)=F(0,y)$ and $g_0(x)=\frac{1}{x}(F(x,0)-F(0,0))$, which are smooth (extending $g_{0}$ by continuity setting $g_{0}(0)=\partial_{x}F(0,0)$), we can write 
        \begin{equation}
        \label{eq:expansion_N=0}
            F(x,y)=f_0(y)+xg_0(x)+(xy)R_1(x,y).
        \end{equation}
         We now prove \eqref{eq:euge} for $N+1$, assuming it holds for $N$. Since $R_{N+1}$ is smooth, using \eqref{eq:expansion_N=0}, we deduce the existence of two smooth functions $f_{N+1},g_{N+1}\in C^\infty(\R)$ such that
        \[
        R_{N+1}(x,y)= f_{N+1}(y)+xg_{N+1}(x) + xy R_{N+2}(x,y),
        \]
        for some smooth remainder $R_{N+2}$. Substituting the latter expression into \eqref{eq:euge}, we conclude.
    \end{proof}
 \begin{proof}[Proof of \cref{l:integrato}]
By \cref{l:euge}, for every $N\geq0$ we can write
\begin{equation}
G(s,r)
=\sum_{k=0}^N(sr)^k\bigl(f_k(r)+s g_k(s)\bigr)
+(sr)^{N+1}R_{N+1}(s,r),
\end{equation}
for suitable smooth functions $f_k,g_k,R_{N+1}$. We now set 
\begin{equation}
    G_{1,k}(s,r):=(sr)^kf_k(r),\qquad  G_{2,k}(s,r):=s^{k+1}r^k g_k(s),\qquad G_{3,N}(s,r):=(sr)^{N+1}R_{N+1}(s,r).
\end{equation}
Let $\mathcal G_{i,k}(\ve)$ be the integral \eqref{eq:mathcal_G} computed with $G_{i,k}$ in place of $G$. It is sufficient to prove the lemma for each and every one of these functions $\mathcal G_{i,k}(\ve)$. Indeed, if the lemma holds for $\mathcal G_{1,k}(\ve)$, $\mathcal G_{2,k}(\ve)$, $\mathcal G_{i,N}(\ve)$, for $k=0,\dots,N$, with coefficients $\gamma_n(G_{i,k}),\beta_n(G_{i,k})$ computed as in \eqref{eq:abstract_coefficients} with $G_{i,k}$ in place of $G$, then by linearity of the mappings $G\mapsto \gamma_n(G)$ and $G\mapsto \beta_n(G)$ we have  
\begin{equation}
\begin{aligned}
    \mathcal G(\ve)&=\sum_{k=0}^{N}\mathcal (G_{1,k}(\ve)+\mathcal G_{2,k}(\ve))+\mathcal G_{3,N}(\ve)=\sum_{n=0}^N\ve^n
   		\sum_{k=0}^N\sum_{i=1}^3\bigl(\gamma_n(G_{i,k})+\beta_n(G_{i,k})\log\ve\bigr)+o(\ve^N),\\
        &=\sum_{n=0}^N\ve^n
   		\bigl(\gamma_n(G)+\beta_n(G)\log\ve\bigr)+o(\ve^N).
\end{aligned}
\end{equation}
We first focus on $G_{1,k}(s,r)=(sr)^kf_k(r)$:
\begin{equation}
\mathcal G_{1,k}(\ve):=\int_\ve^{r_0}G_{1,k}(\ve/r,r)\,dr
=\ve^k\int_\ve^{r_0}f_k(r)\,dr
=\ve^k\int_0^{r_0}f_k(r)\,dr
-\ve^k\int_0^\ve f_k(r)\,dr.
\end{equation}
Then, a Taylor expansion of $f_k$ at zero gives an expansion with no logarithms, whose non-zero coefficients are
\begin{equation}
\gamma_k(G_{1,k})=\int_0^{r_0}f_k(r)\,dr,
\qquad
\gamma_{k+j+1}(G_{1,k})=-\frac{f_k^{(j)}(0)}{(j+1)!},
\qquad j\geq0.
\end{equation}
These agree with \eqref{eq:abstract_coefficients}.
Indeed, for $n\geq1$, the term 
\begin{equation}\label{eq:first_term}
	\Fp_{\delta\to0^+}\left(
	\frac1{n!}\int_\delta^{r_0}
	\frac{\partial_s^nG_{1,k}(0,r)}{r^n}\,dr\right)
\end{equation}
is non-zero
only when $n=k$, in which case it is
$\int_0^{r_0}f_n(r)\,dr$.
The second term, which is 
\begin{equation}\label{eq:secon_term}
	\Fp_{\delta\to0^+}\left(
	\frac1{(n-1)!}\int_\delta^1
	\frac{\partial_r^{n-1}G_{1,k}(s,0)}{s^{n+1}}\,ds
	\right)
\end{equation}
is non-zero only when
$n=k+j+1$, for some $j\geq0$, and then equals
\begin{equation}
\frac{f_k^{(j)}(0)}{j!}
\Fp_{\delta\to0^+}\int_\delta^1s^{-j-2}\,ds
=-\frac{f_k^{(j)}(0)}{(j+1)!}.
\end{equation}
The mixed derivative defining $\beta_n$ vanishes for
every $n$. Next, we focus on $G_{2,k}(s,r)=(sr)^k s g_k(s)$.
The change of variables $s=\ve/r$ gives
\begin{equation}
\mathcal G_{2,k}(\ve)
=\ve^{k+1}\int_{\ve/r_0}^1\frac{g_k(s)}s\,ds.
\end{equation}
Adding and subtracting $g_k(0)$ from the numerator of $\frac{g_k(s)}s$ we obtain the identity 
\begin{equation}
\begin{aligned}
\mathcal G_{2,k}(\ve)
=\ve^{k+1}
\left(\int_0^1\frac{g_k(s)-g_k(0)}s\,ds+g_k(0)\log r_0-g_k(0)\log\ve\right)-\ve^{k+1}\int_0^{\ve/r_0}\frac{g_k(s)-g_k(0)}s\,ds
\end{aligned}
\end{equation}
The integrand in the last term is smooth. A Taylor expansion therefore gives, for $k\geq0$ and $j\geq1$
\begin{equation}
\begin{aligned}
\gamma_{k+1}(G_{2,k})&=\int_0^1\frac{g_k(s)-g_k(0)}s\,ds+g_k(0)\log r_0,
\qquad \gamma_{k+j+1}(G_{2,k})
=-\frac{g_k^{(j)}(0)}{j!\,j\,r_0^j},\\[0.5em]
\beta_{k+1}(G_{2,k})&=-g_k(0),\qquad 
\end{aligned}
\end{equation}
with all remaining coefficients equal to zero. These are precisely the coefficients obtained by applying 
\eqref{eq:abstract_coefficients} to $G_{2,k}(s,r)=(sr)^k s g_k(s)$. Indeed, the term \eqref{eq:secon_term} is non-zero only for $n=k+1$,
when it equals $\int_0^1\frac{g_k(s)-g_k(0)}s\,ds$, while for $n=k+j+1$, $j\geq0$,
the term \eqref{eq:first_term} is
\begin{equation}
\frac{g_k^{(j)}(0)}{j!}
\Fp_{\delta\to0^+}
\int_\delta^{r_0}r^{-j-1}\,dr
=
\begin{cases}
g_k(0)\log r_0,&j=0,\\[4pt]
\displaystyle
-\frac{g_k^{(j)}(0)}{j!\,j\,r_0^j},&j\geq1.
\end{cases}
\end{equation}
The mixed derivative defining $\beta_n(G_{2,k})$ in \eqref{eq:abstract_coefficients} is non-zero
only for $n=k+1$, when its value is $-g_k(0)$. Finally for the function $G_{3,N}(s,r)=(sr)^{N+1}R_{N+1}(s,r)$ we have 
\begin{equation}
\mathcal G_{3,N}(\ve)=\ve^{N+1}\int_\ve^{r_0}R_{N+1}(\ve/r,r)\,dr
=O(\ve^{N+1}).
\end{equation}
Thus $\gamma_n(G_{3,N})=\beta_n(G_{3,N})=0$ for $n=0,\dots, N$. These coefficients therefore coincide with those obtained by applying \eqref{eq:abstract_coefficients} to $G_{3,N}$, since all derivatives appearing therein vanish when
vanish when applied to $(sr)^{N+1}R_{N+1}$.
\end{proof}
\section{Distributions on manifolds} 
\label{app:distributions}

We review some basic facts about distributions on manifolds, following the monograph \cite{Grosser}. Let $M$ be a smooth oriented manifold of $\dim M=n$. Fix $k\in\N\cup\{\infty\}$ and an atlas $\{(V_\alpha,\psi_\alpha)\}_{\alpha\in A}$ for $M$. Then, given $\alpha\in A$, we introduce a locally convex topology on $C^k(V_\alpha)$ via the following notion of convergence: a sequence $\{u_l\}_{l\in\N}\subset C^k(V_\alpha)$ converges to zero if and only if, for every $K\subset\psi_\alpha(V_\alpha)$ compact and for every $\nu\in\N^n$, $|\nu|\leq k$, 
\begin{equation}
\partial^\nu( u_l\circ\psi_\alpha^{-1})\xrightarrow{l\to\infty}0
\qquad\text{uniformly on }K.
\end{equation}
Such a topology exists and is generated by the following family of semi-norms on $C^k(V_\alpha)$: denote by $\{K_m^\alpha\}_{m\in\N}$ an exhausting, increasing sequence of compact subsets of $\psi_\alpha(V_\alpha)$, and define
\begin{equation}
\label{eq:chart_semi-norm}
p_{s,m,\alpha}(u):=\sup_{x\in K_m^\alpha,\,|\nu|\leq s}
\left|\partial^\nu\left(u\circ\psi_\alpha^{-1}(x)\right)\right|,\qquad\forall\, u\in C^k(V_\alpha),
\end{equation}
where $s\in\N$, $0\leq s\leq k$, and $m\in\N$. In turn, this allows to define a topology on the space $C^k(M)$: observe that a sequence $\{u_l\}_{l\in\N}\subset C^k(M)$ converges to zero if and only if $u_l\to0$ in $C^k(V_\alpha)$, for every $\alpha\in A$. Equipped with this topology, $C^k(M)$ is a separable Fr\'echet space and, for every $K\subset M$ compact, the set $C^k_{c,K}(M):=\{u\in C^k(M):\supp(u)\subset K\}\subset C^k(M)$ is a Fr\'echet subspace. Hence, letting $\{K_i\}_{i\in\N}$ be an exhausting, strictly increasing sequence of compact subsets of $M$, we have  
\[
C^k_c(M)=\bigcup_{i\in \N}C^k_{c,K_i}(M),
\]
and we endow $C^k_c(M)$ with the \emph{inductive limit topology} $\tau$ with respect to the countable collection of Fr\'echet subspaces $\{C^k_{c,K_i}(M)\}_{i\in\N}$. 
Finally, we can give the following definition of a distribution on $M$.

\begin{defi}
    A distribution (respectively of order $k$) on $M$ is a continuous linear functional on the locally convex space $C_c^\infty(M)$ (respectively $C_c^k(M)$). We denote the space of distributions by $\D'(M)$ (respectively $\D'_k(M)$). Given $T\in\D'(M)$, we denote its action on a function $u$ by $\langle T,u\rangle$ or $T(u)$.
\end{defi}

The following criterion for a linear functional on $C^k_c(M)$ to be a distribution is immediate from the definition of the topology $\tau$.
\begin{lemma}
\label{lem:continuity}
Let $k\in\N$ and let $T:C_c^\infty(M)\to\mathbb R$ be a linear functional.
Then the following conditions are equivalent:
\begin{enumerate}[label=(\roman*),leftmargin=3em]
\item $T$ is a distribution of order at most $k$;
\item For every compact $K\subset M$, the restriction $T|_{C^\infty_{c,K}(M)}$ is continuous with respect to the topology induced by $C^k_{c,K}(M)$;
\item For every compact $K\subset M$, there exist $s,\ell,m\in\mathbb N$, with $s\leq k$, a constant $c_K>0$, and indices $\alpha_1,\ldots,\alpha_\ell$ such that
\[
|T(u)|\leq c_K\max_{1\leq i\leq\ell}
p_{s,m,\alpha_i}(u),
\qquad
\forall u\in C^\infty_{c,K}(M);
\]
\item $T$ admits a unique continuous linear extension $\widetilde T:C_c^k(M)\to\mathbb R$.
\end{enumerate}
\end{lemma}

For every open set $\Omega\subset M$, denote by $\D'(\Omega)$ the space of distributions on $\Omega$, namely the continuous linear functionals on $C_c^\infty(\Omega)$. Given $\widetilde\Omega\subset\Omega$, denote by $T|_{\widetilde\Omega}:=T|_{C^\infty_c(\widetilde\Omega)}$ the restriction of $T$ to $\widetilde\Omega$. 

\begin{prop}\label{prop:Dprime_sheaf}
	Let $\{\Omega_i\}_{i\in I}$ be an open cover of a smooth
	manifold $M$, and let $T:C_c^\infty(M)\to\mathbb R$
	be a linear functional. Fix an integer $k\geq0$.
	If $T|_{\Omega_i}\in \D'(\Omega_i)$ of order at most $k$, for every $i\in I$, then $T$ is a distribution on $M$ of order at most $k$.
\end{prop}

\begin{proof}
    For a given $K\subset M$ compact, there is a finite sub-cover $\{\Omega_{i_j}\}_{j=1}^N$ for $K$ and a partition of unity $\{\chi_j\}_{j=1}^N$ subordinated to it. Then, setting $T_i:=T|_{\Omega_i}\in\D'(\Omega_i)$, we write
    \begin{equation}
    \label{eq:partition_of_unity}
        T(u) = T\left(\sum_{j=1}^N\chi_ju\right) = \sum_{j=1}^N T(\chi_ju) = \sum_{j=1}^N T_{i_j}(\chi_ju),\qquad\forall\,u\in C^\infty_{c,K}(M).
    \end{equation} 
    We claim that, for every $j=1,\ldots,N$, $C^\infty_{c,K}(M)\ni u\mapsto T_{i_j}(\chi_j u)$ admits a unique continuous extension to the Fr\'echet space $C^k_{c,K}(M)$. By assumption, $T_{i_j}\in\D'(\Omega_{i_j})$ is a distribution of order at most $k$. Hence, by \cref{lem:continuity}, it admits a continuous extension 
    \[
    \widetilde T_{i_j}:C_c^k(\Omega_{i_j})\to\mathbb R.
    \]
    In particular, setting $K_j:=K\cap\supp\chi_j$, the restriction 
    \[
    \widetilde T_{i_j}|_{C^k_{c,K_j}(M)}
    \]
    is continuous. We now observe that the multiplication map by $\chi_j$, defined as
    \[
    M_j:C^k_{c,K}(M)\to C^k_{c,K_j}(M),\qquad u\to \chi_ju,
    \]
    is continuous in the suitable Fr\'echet topologies, as a consequence of Leibniz' rule. Hence, we deduce that $\widetilde T_{i_j}\circ M_j$ is the (unique) continuous extension of $C^\infty_{c,K}(M)\ni u\mapsto T_{i_j}(\chi_j u)$ to $C^k_{c,K}(M)$, which is then a distribution of order at most $k$ by \cref{lem:continuity}, as claimed. Since the sum of distributions is a distribution as well, we infer that $T|_{C^\infty_{c,K}(M)}$ is continuous with respect to the topology induced by $C^k_{c,K}(M)$, as a consequence of \eqref{eq:partition_of_unity}. By the arbitrariness of $K$ and using \cref{lem:continuity}, we conclude that $T$ is a distribution of order at most $k$. 
\end{proof}

\subsection{The support of a distribution} The next result is a consequence of \cite[Thm.\ 3.1.7]{Grosser}. 

\begin{lemma}
    Let $M$ be a smooth manifold, let $\Omega\subset M$ be an open set and let $\{\Omega_i\}_{i\in I}$ be an open cover for $\Omega$. Assume that there is a family $\{T_i\}_{i\in I}$ such that, for all $i\in I$, $T_i\in \D'(\Omega_i)$ and  
    \begin{equation}
        T_i|_{\Omega_i\cap\Omega_j}= T_j|_{\Omega_i\cap\Omega_j},\qquad\forall\, i,j\in I.        
    \end{equation}
    Then, there exists a unique distribution $T\in \D'(\Omega)$ such that $T|_{\Omega_i}=T_i$, for all $i\in I$.    
\end{lemma}
By the previous lemma, it follows that, given $T \in \D'(\Omega)$, there is a largest open $\Omega_T \subset \Omega$ on which $T$ vanishes (i.e.\ $T|_{\Omega_T}=0$).

\begin{defi}[Support of a distribution]
    Let $M$ be a smooth manifold, let $\Omega\subset M$ be an open set and let $T \in \D'(\Omega)$. Then, the \emph{support} of $T$ is the closed set given by
    \begin{equation}
        \supp(T) := \Omega\setminus\Omega_T 
        =\Omega\setminus \{x \in\Omega: T|_{V_x}=0, \text{ for some neighborhood } V_x \subset \Omega \text{ of } x\}.
    \end{equation}
    We say that $T$ is compactly supported if $\supp(T)$ is compact. 
\end{defi}

\subsection{A distribution on characteristic surfaces} 
We analyze a prototypical functional appearing when dealing with volumes of half-tubes around characteristic surfaces in the Heisenberg group. 

\begin{prop}
\label{lem:distribution_T}
    Let $S=\partial\Omega \subset\hei$ be a smooth, oriented, embedded surface. Assume that $\Sigma(S)$ consists only of geometrically stable characteristic points and let $B\subset S$ be one of the balls of \cref{lem:rho_0} with polar coordinates $(r,\theta)$. Let $T:C^\infty(S)\to\R$ be the linear functional given, in coordinates $(r,\theta)$, by
    \begin{equation}
    \label{eq:assumption_T}
        T(\varphi) = \int_0^{2\pi} \partial_r^k\varphi(0,\theta)\phi(0,\theta)d\theta, \qquad\forall\,\varphi\in C^\infty(S),
    \end{equation}
    for some $k\in\N$ and $\phi\in C^\infty(\R^+\times \mathbb S^1)$. Then, $T\in \D'(S)$, it has at most order $k$ and $\supp(T)\subset \{p\}$.  
\end{prop}

\begin{proof}
We prove that $T\in\D'(S)$. This is straightforward if $k=0$, hence we assume that $k\geq 1$. Firstly, observe that if $\supp(\varphi)\cap B=\emptyset$, then $T(\varphi)=0$, hence we check the continuity of $T$ on every $K\subset B$ compact. Let $\psi:V\subset S\to \R^2$ be a coordinate chart such that $B\subset V$ and on $\psi(B)$, we use the (smooth) polar coordinates $\xi:\R^+_r\times\mathbb S^1_\theta\to \R^2$ on $\psi(B)$, centered at $(0,0)$. By \cref{lem:rho_0}, it holds $t_\cut(\xi(r,\theta))=r$. Secondly, as a consequence of Fa\`a di Bruno's formula, there are $c_{k,\alpha}\in\N$ such that   
     \begin{equation}
     \label{eq:faa_di_bruno}
         \partial_r^k (\varphi\circ \xi)(0,\theta) = \sum_{j=1}^k \sum_{\substack{\alpha\in\N_0^j\\|\alpha|=k}} c_{k,\alpha}D^j\varphi(0)\left[\partial_r^{\alpha_1}\xi(0,\theta),\ldots,\partial_r^{\alpha_j}\xi(0,\theta)\right],
    \end{equation}
    where $D^j\varphi(0):(\R^2)^j\to\R$ is the $j$-th Fr\'echet derivative of $\varphi$ at $0\in \psi(B)$, viewed as a symmetric multilinear map and $\N_0^j$ is the set of $j$-tuples with entries $\geq 1$. Thus, recalling that $\xi,\phi$ are smooth functions on $\R^+\times\mathbb S^1$, using  \eqref{eq:faa_di_bruno}, we estimate, for every $\varphi\in C^k_{c,K}(S)$,  
    \begin{equation}
         |T(\varphi)|\leq 2\pi\sum_{j=1}^k \sum_{\substack{\alpha\in\N_0^j\\|\alpha|=k}} c_{k,\alpha} |D^j\varphi(p)| \bigg(\sup_{\substack{\ell=1,\ldots,k\\\theta\in \mathbb S^1}} |\partial_r^\ell\xi(0,\theta)|+1\bigg)^j\|\phi(0,\cdot)\|_{C(\mathbb S^1)}\leq C_B \, p_{k,\psi}(\varphi), 
     \end{equation}
     where $p_{k,\psi}$ is the semi-norm \eqref{eq:chart_semi-norm} (associated to the chart $\psi$ and to the compact set $\psi(B)$). Since $K$ was arbitrary, by \cref{lem:continuity} we infer that $T$ is a distribution of order at most $k$, as claimed. Finally, \eqref{eq:faa_di_bruno} implies that $\supp(T)\subset\{p\}$, since, for every $\varphi\in C^\infty(S)$, with $p\notin\supp(\varphi)$, we have $T(\varphi)=0$. 
\end{proof}
\subsection{Swapping finite parts and integrals} We prove the technical step needed in \cref{lem:string_theory} to prove that $T_n$ is a distribution, namely that integrals can be exchanged with finite parts.
\begin{lemma}\label{lem:finite-part-swap}
    Let $r_0>0$, let $n\geq0$ be an integer, and let
    $f\in C^\infty([0,r_0]\times\mathbb S^1)$.
    Then the following finite parts exist and satisfy
    \begin{equation}
        \Fp_{\delta\to0^+}\left(
        \int_\delta^{r_0}\int_{\mathbb S^1}
        \frac{f(s,\theta)}{s^{n+1}}\,d\theta\,ds\right)
        =
        \int_{\mathbb S^1}
        \Fp_{\delta\to0^+}
        \left(\int_\delta^{r_0}
        \frac{f(s,\theta)}{s^{n+1}}\,ds\right)
        d\theta.
    \end{equation}
\end{lemma}
\begin{proof}
    Computing a Taylor expansion of $f$ gives
    \begin{equation}
        f(s,\theta)
        =\sum_{k=0}^n f_k(\theta)s^k
        +s^{n+1}R(s,\theta),
        \qquad
        f_k(\theta):=\frac{\partial_s^kf(0,\theta)}{k!},
    \end{equation}
    where $R$ is smooth. Integrating term by term yields
    \begin{equation}
        \int_\delta^{r_0}
        \frac{f(s,\theta)}{s^{n+1}}\,ds
        =
        \sum_{k=0}^{n-1}
        \frac{f_k(\theta)}{n-k}\delta^{k-n}
        -f_n(\theta)\log\delta
        +C(\theta)+O(\delta),
    \end{equation}
    where $C$ is the following smooth function 
    \begin{equation}
    \begin{aligned}
        C(\theta)
        =\int_0^{r_0}
        \frac{f(s,\theta)-\sum_{k=0}^n f_k(\theta)s^k}
        {s^{n+1}}\,ds-\sum_{k=0}^{n-1}
        \frac{f_k(\theta)}{n-k}r_0^{k-n}
        +f_n(\theta)\log r_0,
    \end{aligned}
\end{equation}
    and the remainder $O(\delta)$ is uniform in
    $\theta$, by compactness of $\mathbb S^1$.
    Thus the finite part for each fixed $\theta$ is $C(\theta)$.
    Integrating the expansion in $\theta$, its constant term
    is $\int_{\mathbb S^1}C(\theta)\,d\theta$, proving the identity.
\end{proof}
\section{Legendrian fronts}\label{app:Legendrian}
In this section, we give some details around Legendrian fronts. Let $q,p\in \mathbb R^n$, $\ell\in\mathbb R$, and, in coordinates $(q,\ell,p)\in \mathbb R^{2n+1}$, consider the smooth $1$-form 
\begin{equation}
    \alpha:=d\ell-p\cdot dq.
\end{equation}
We define the \emph{front projection} as the map $\pi_f:\R^{2n+1}\to\R^{n+1}$, $\pi_f(q,\ell,p)=(q,\ell).$
\begin{defi}\label{def:front}
    Let $M$ be a smooth manifold of $\dim(M)=n$, possibly with boundary. A smooth map $F:M\to \mathbb R^{n+1}$ is called a \emph{Legendrian front} if there exists a smooth immersion $\widetilde{F}:M\to \mathbb R^{2n+1}$ such that 
    \begin{equation}
        \widetilde{F}^*\alpha=0,\qquad F=\pi_f\circ \widetilde{F}.
    \end{equation}
    The map $\widetilde{F}$ is called a \emph{Legendrian lift} of $F$.
\end{defi}

Since, in the main body of the paper, we mostly encounter Legendrian fronts defined on $2$-dimensio\-nal manifolds, we restrict our analysis to this situation. Let $F:M\to\mathbb R^3$ be a smooth Legendrian front. We denote by $\operatorname{Crit}(F):=\{q\in M: \ker(d_qF)\neq \{0\}\}$ the set of critical points. Then, following the notations of \cite{SingularFronts}, a critical point $q\in M$ of $F$ is called \emph{non-degenerate} if, in local coordinates
$(u,v)$ near $q$, the map
\begin{equation}
	G=\partial_uF\times\partial_vF:M\to\mathbb R^3
\end{equation}
satisfies $\operatorname{rank}(d_qG)=1$.
This condition is independent of the choice of coordinates.

\begin{remark}\label{rmk:non_deg_sing}
	Write $F=(x,y,\ell)\in\R^3$ and let $(p_x,p_y)\in\R^2$ be the components of the momenta of $\widetilde F$. Since, $\widetilde F$ is a Legendrian lift of $F$, it holds 
	\begin{equation}\label{eq:contact_id}
	    d\ell=p_x\,dx+p_y\,dy.
	\end{equation}
    Set $W:=\det\bigl(\partial(x,y)/\partial(u,v)\bigr)$.
	Then the contact identity \eqref{eq:contact_id} gives $G=W(-p_x,-p_y,1)$.
	At a singular point $q\in M$, we have $W(q)=0$ and hence
	\begin{equation}
		d_qG=(-p_x(q),-p_y(q),1)\,d_qW.
	\end{equation}
	Thus $q$ is non-degenerate if and only if $d_qW\neq0$.
\end{remark}
Let us recall the models of non-degenerate singularities for $2$-dimensional Legendrian fronts.

\begin{example}[Cuspidal edge]\label{ex:cusp_edge} The smooth map $F_{ce}:\mathbb R^2\to\mathbb R^3$ defined as
\begin{equation}
    F_{ce}(u,v)=(u,\;v^2,\;v^3),
\end{equation}
is a Legendrian front. The critical set and double points locus are 
\begin{equation}\label{eq:crit_ce}
    \mathrm{Crit}(F_{ce})=\{(u,0)\in\mathbb R^2\},\qquad \mathrm{Doub}(F_{ce})=\emptyset.
\end{equation}
\end{example}
\begin{example}[Swallowtail]\label{ex:swallow_tail} The smooth map $F_{st}:\mathbb R^2\to\mathbb R^3$ defined as
\begin{equation}
    F_{st}(u,v)=\left(4u^3+2uv,\;v,\;3u^4+u^2v\right)
\end{equation}
is a Legendrian front. The critical set and double points locus are 
\begin{equation}\label{eq:crit_st}
    \mathrm{Crit}(F_{st})=\{(u,-6u^2)\in\mathbb R^2\,:\,u\in\mathbb R\},\qquad \mathrm{Doub}(F_{st})=\{(u,-2u^2)\in\mathbb R^2\,:\,u\neq 0\}.
\end{equation}
\end{example}
A point $q$ in the domain of a Legendrian front  $F$ is called a \emph{cuspidal edge} (resp.\,\emph{swallowtail}) if there exists a chart centered at $q$ and target coordinates in which $F$ has the normal form \cref{ex:cusp_edge} (resp.\,\cref{ex:swallow_tail}).
\begin{lemma}[{\cite[Prop.~1.3]{SingularFronts}}]
	\label{lem:crit_recon}
	Let $q\in M$ be a non-degenerate singular point of a smooth
	Legendrian front $F:M^2\to\mathbb R^3$.
	Then $\mathrm{Crit}(F)$ is a smooth curve near $q$,
	and $dF$ has rank one along it.
	Let $\gamma:(-t_0,t_0)\to M$, $t_0>0$, be a regular
	parametrization of this curve with $\gamma(0)=q$, and let
	$\eta(t)$ be a smooth non-vanishing vector field along
	$\gamma$ spanning $\ker d_{{\gamma(t)}}F$. Set
	\begin{equation}
		\Delta(t)=\det\bigl(\dot\gamma(t),\eta(t)\bigr),
	\end{equation}
	where the determinant is computed in local coordinates
	on $M$. Then:
	\begin{enumerate}
		\item $q$ is a cuspidal edge if and only if
		$\Delta(0)\neq0$,
		\item $q$ is a swallowtail if and only if
		$\Delta(0)=0$ and $\Delta'(0)\neq0$.
	\end{enumerate}
\end{lemma}
\subsection{Stability lemmas}
We now prove a stability lemma for families of Legendrian fronts. We say that a family of maps $F^\ve:M\to\mathbb R^3$ is a \emph{smooth family of Legendrian fronts}, if it admits a smooth family of Legendrian lifts $\widetilde F^\ve:M\to\mathbb R^5$ (cf. \cref{def:front}). 

\noindent The following lemma is a well known consequence of the theory of Legendrian singularities (see \cite[$\S$ 20 and $\S$ 21]{arnold}). We include a proof for completeness. 

\begin{lemma}\label{lem:stability}
	Let $F:M\to \mathbb R^3$ be a Legendrian front and let $\mathscr D:=\mathrm{Doub}(F)$. Assume that
    \begin{enumerate}[i)]
        \item every point in $F(\mathscr D)$ has exactly two distinct preimages in $M$, which are regular points whose image tangent spaces are transverse,
        \item $\overline{\mathscr D}$ is a compact subset of $M\setminus \partial M$ and every point in $\overline {\mathscr D}\setminus \mathscr D$ is a swallowtail.
    \end{enumerate}
    Then, for every $F^\ve:M\to \mathbb R^3$ smooth family of fronts such that $F^0=F$, there exists $\ve_0>0$, and smooth families of diffeomorphisms
	$$
	h_\ve:M\to M,
	\qquad
	k_\ve:\mathbb R^3\to\mathbb R^3,
	\qquad
	h_0=\mathrm{id}_{M},
	\quad k_0=\mathrm{id}_{\mathbb R^3},
	$$
	and an open neighborhood $U$ of $\overline{\mathscr D}$ such that
	\begin{equation}\label{eq:front_global_trivialization}
		F^\ve\circ h_\ve=k_\ve\circ F^0
		\qquad\text{on }U,
		\qquad 0\leq\ve\leq\ve_0.
	\end{equation}
\end{lemma}
\begin{proof}
	Up to replacing the maps $F^\ve$ by their restriction to a suitable compact neighborhood of $\overline{\mathscr D}$, we may assume without loss of generality that $M$ is compact. The assumptions on $\mathscr D$ imply $F^{-1}(F(\overline {\mathscr D}))=\overline{\mathscr D}$. Indeed, a point outside $\overline{\mathscr D}$ with the same image as a point of $\overline {\mathscr D}$ is itself a double point, and thus belongs to $\mathscr D$.
	Moreover, for every $y\in F(\overline {\mathscr D}\setminus \mathscr D)$, the fiber $F^{-1}(y)$ is a singleton, since
	$(\overline {\mathscr D}\setminus \mathscr D)\cap \mathscr D=\emptyset $. Therefore, for every $y\in F(\overline {\mathscr D})$, the fiber $F^{-1}(y)$ lies in $\overline {\mathscr D}$ and consists
	either of one swallowtail or of two regular points whose image tangent spaces intersect transversely. We now claim that for any $y\in F(\overline {\mathscr D})$ and $\ve_0>0$ sufficiently small, there exist two open neighborhoods $V_y,V_{y,\ve}\subset\R^3$ of $y$, and two open neighborhoods $U_y, U_{y,\ve}\subset M\setminus \partial M$ of $F^{-1}(y)$, and diffeomorphisms 
    \begin{equation}\label{eq:the_diffeos}
        \varphi_{y,\ve}:U_y\to U_{y,\ve},\qquad \psi_{y,\ve}:V_y\to V_{y,\ve},\qquad \varphi_{y,0}=\mathrm{Id}_{U_y}, \qquad \psi_{y,0}=\mathrm{Id}_{V_y}\qquad \ve\in[0,\ve_0],
    \end{equation}
    depending smoothly on $\ve$ such that the following equality holds
    \begin{equation}\label{eq:local_identity}
	F^\ve\circ \varphi_{y,\ve}=\psi_{y,\ve}\circ F^0,\qquad \text{on $U_{y}$}.
	\end{equation}
   If $F^{-1}(y)$ is a pair of distinct regular double points whose image tangent spaces are transverse, the claim follows by standard transversality theory. If $F^{-1}(y)$ is a single swallowtail see \cite[$\S$ 21]{arnold} for existence for every fixed $\ve$, and \cite[Thm.\,4.1]{Ishikawa2005} applied to the lifts $\widetilde F^\ve$ for smoothness in $\ve$. A standard partition of unity argument can be used to patch together the diffeomorphisms of \eqref{eq:local_identity}. First, we define the vector fields
	\begin{equation}\label{eq:i_campi_ve}
	    X_y^\ve
	=-(\partial_\ve \varphi_{y,\ve})\circ \varphi_{y,\ve}^{-1}\in\Gamma(TU_{y,\ve}),
	\qquad
	Y_y^\ve
	=(\partial_\ve \psi_{y,\ve})\circ \psi_{y,\ve}^{-1}\in\Gamma(TV_{y,\ve}).
	\end{equation}
	Differentiating the local identity \eqref{eq:local_identity} and composing with
	$\varphi_{y,\ve}^{-1}$ yields
	\begin{equation}\label{eq:local_front_velocity}
		\partial_\ve F^\ve=dF^\ve X_y^\ve+Y_y^\ve\circ F^\ve,\qquad \text{on $U_{y,\ve}$}.
	\end{equation}
    By compactness of the parameter interval $[0,\ve_0]$, and by continuity of the diffeomorphisms \eqref{eq:the_diffeos} with respect to $\ve$, it is possible to construct independent of $\ve$ refinements $\mathcal U_y\subset U_{y,\ve}$ and $\mathcal V_y\subset V_{y,\ve}$ over which the fields $X_{y}^\ve$ and $Y_{y}^\ve$ are well defined and \eqref{eq:local_front_velocity} holds. 
    Since $\mathcal U_y$ is a neighborhood of the unique $F^0$-preimages of $y$, up to shrinking $\mathcal V_y$ we have $(F^0)^{-1}(\mathcal V_y)\subset \mathcal U_y$. By compactness of $M$ and continuity of $F^\ve$ with respect to $\ve$, up to shrinking $\mathcal V_y$ we may assume that $(F^\ve)^{-1}(\mathcal V_y)\subset \mathcal U_y$ for all $y\in F^0(\overline{\mathscr D})$. Summarizing we have 
    \begin{equation}\label{eq:inclusion_F_ve}
        \mathcal U_y\subset U_{y,\ve},\qquad\mathcal V_y\subset V_{y,\ve},\qquad (F^{\ve})^{-1}(\mathcal V_y)\subset\mathcal U_y,\qquad\forall\, y\in F^0(\overline{\mathscr D}),\qquad\,\ve\in [0,\ve_0].
    \end{equation}
    By compactness of $\overline{\mathscr D}$, we may take finitely many $\{\mathcal V_i:=\mathcal V_{y_i}\}_{1}^n$ covering $F^0(\overline{\mathscr D})$ and we denote
    \begin{equation}
      X_i^\ve=X_{y_i}^\ve, \qquad  Y_i^\ve=Y_{y_i}^\ve.
    \end{equation}
    We now choose nonnegative functions $\rho_i\in C_c^\infty(\mathcal V_i)$
	such that $\sum_i\rho_i=1$ on a neighborhood of $F(\overline{\mathscr D})$ and we define
	$$
	X^\ve=\sum_{i=1}^n(\rho_i\circ F^\ve)X_i^\ve,\qquad Y^\ve=\sum_{i=1}^n\rho_iY_i^\ve,
	$$
	Note that, since $\rho_i\in C_c^\infty(\mathcal V_i)$ and by \eqref{eq:inclusion_F_ve}, $X^\ve$ is a smooth vector field defined globally on $M$, which is zero outside of $\bigcup_i\mathcal U_i$, while $Y^\ve$ is a smooth vector field on $\R^3$ which is zero outside of $\bigcup_i \mathcal V_i$. In particular $X^\ve$ and $Y^\ve$ are non-autonomous vector fields with uniformly compact support for $\ve\in [0,\ve_0]$. Let $\mathcal V\subset \bigcup_i \mathcal V_i $ be a neighborhood of $F(\overline{\mathscr D})$ over which $\sum_i\rho_i=1$. Since $\overline{\mathscr D}\subset(F^{\ve})^{-1}(\mathcal V)$ for all $\ve\in[0,\ve_0]$, by compactness of the parameter interval $[0,\ve_0]$, and by continuity of $F^\ve$ with respect to $\ve$, we can find $\mathcal U\subset \bigcup_i\mathcal U_i$ open neighborhood of $\overline{\mathscr D}$ such that $\mathcal U\subset (F^{\ve})^{-1}(\mathcal V)$ for all $\ve\in [0,\ve_0]$. Multiplying the local equations \eqref{eq:local_front_velocity} by $\rho_i\circ F^\ve$ and summing over $i$ we obtain
	\begin{equation}\label{eq:global_front_velocity}
		\partial_\ve F^\ve=dF^\ve X^\ve+Y^\ve\circ F^\ve,\qquad \text{on $\mathcal U$}.
	\end{equation}
	Let $h_\ve$ and $k_\ve$ be the time-dependent flows defined by
	$$
	\partial_\ve h_\ve=-X^\ve\circ h_\ve,
	\qquad
	\partial_\ve k_\ve=Y^\ve\circ k_\ve,
	\qquad
	h_0=\mathrm{id}_M,\quad k_0=\mathrm{id}_{\mathbb R^3}.
	$$
	The uniformly compact supports of $X_\ve$ and $Y_\ve$ ensure that the flows are
	defined, on the entire $M$ and $\R^3$, throughout the whole
	parameter interval $[0,\ve_0]$. They are smoothly varying diffeomorphisms and they are the identity outside fixed compact sets. Let $U\subset \mathcal U$ be an open neighborhood of $\overline{\mathscr D}$ with the property that for any $q\in U$, $h_\ve(q)$ remains in $\mathcal U$ for every $\ve\in[0,\ve_0]$. Then, equation \eqref{eq:global_front_velocity} holds evaluated at $h_\ve(q)$ for every $\ve\in[0,\ve_0]$. Consequently 
	$$
	\frac{d}{d\ve}\bigl(F^\ve\circ h_\ve(q)\bigr)
	=\bigl(\partial_\ve F^\ve-dF^\ve X^\ve\bigr)\circ h_\ve(q)
	=Y^\ve\circ (F^\ve\circ h_\ve(q)).
	$$
	Thus, both $F^\ve\circ h_\ve(q)$ and $k_\ve\circ F^0(q)$ solve
	the same ODE with initial value $F^0(q)$. In other words $F^\ve\circ h_\ve(q)=k_\ve\circ F^0(q)$ for all $q\in U$.
\end{proof}
\begin{corollary}\label{cor:stability}
    Let $M=I\times\mathbb S^1$, where $I\subset \R$ is an interval. Let $F:M\to \R^3$ be a Legendrian front satisfying the hypotheses of \cref{lem:stability} with 
    \begin{equation}
        \mathscr D=\{s_0\}\times(\mathbb S^1\setminus\{\phi_1^0,\dots,\phi_k^0\}),
    \end{equation}
    where $s_0\in\mathrm{int}\, I$, and $(s_0,\phi_i^0)\in\{s_0\}\times\mathbb S^1$ are the swallowtails of $F$. Let $F^\ve:M\to \R^3$ be a smooth family of Legendrian fronts, with $F^0=F$. Then, up to shrinking the interval $I$, there exists a smooth function $f:[0,\ve_0]\times\mathbb S^1\to I$ such that 
    \begin{equation}
        \mathrm{Doub}(F^\ve,I\times\mathbb S^1)=\bigl\{(f(\ve,\phi),\phi)\,:\,\phi\in\mathbb S^1\setminus\{\phi_1^\ve,\dots,\phi_k^\ve\}\bigr\},\qquad \ve\in [0,\ve_0],
    \end{equation}
    where $\ve \mapsto\phi_i^\ve\in\mathbb S^1$ is smooth and $f(0,\phi)\equiv s_0$.
\end{corollary}
\begin{proof}
    Let $U_1$ be a neighborhood of $\overline{\mathscr D}$ whose closure is contained in the neighborhood $U$ of validity of
	\eqref{eq:front_global_trivialization}. By reducing $I$ and $\ve_0$, we may arrange that
	$$
	h_\ve(\overline{\mathscr D})\subset I\times\mathbb S^1
	\subset h_\ve(U_1)
	\qquad0\leq\ve\leq\ve_0.
	$$
	Then, injectivity of $h_\ve$ and $k_\ve$ together with
	\eqref{eq:front_global_trivialization}, gives
	$$
	\mathrm{Doub}(F^\ve,I\times\mathbb S^1)
	=h_\ve(\mathscr D).
	$$
	Indeed, every double pair of $F^\ve$ in $I\times\mathbb S^1$ pulls back through $h_\ve$ to a double pair
	of $F^0$ in $U_1$, and both members of every double pair on $\mathscr D$
	are mapped into $I\times\mathbb S^1$. We write $h_\ve(s_0,\phi)=(r(\ve,\phi),g(\ve,\phi))$ for the components of $h_\ve$ restricted to $s=s_0$. Since $g(0,\phi)=\phi$, the maps $g_\ve=g(\ve,\cdot)$
	are diffeomorphisms of $\mathbb S^1$ for small $\ve$, with
	smoothly varying inverses. Thus the map $f$ defined by
	$$
	f(\ve,\phi):=r(\ve,g_\ve^{-1}(\phi))
	$$
	is smooth and satisfies
	\begin{equation}\label{eq:conclusion_front_graph1}
		h_\ve(\mathscr D)=\{(f(\ve,\phi),\phi):\phi\in\mathbb S^1\}\setminus\{h_\ve(s_0,\phi_1^0),\dots,h_\ve(s_0,\phi_k^0)\},
		\qquad f(0,\phi)=s_0.
	\end{equation}
    Setting $\phi_i^\ve:=g(\ve,\phi_i^0)$ we conclude the proof.
\end{proof}
\bibliographystyle{alphaabbr}

    \bibliography{bibliography11}

\newcommand{\etalchar}[1]{$^{#1}$}
\begin{thebibliography}{GGPT03}

\bibitem[ABB20]{Agrachev}
A.~Agrachev, D.~Barilari, and U.~Boscain.
\newblock {\em A comprehensive introduction to sub-{R}iemannian geometry}, volume 181 of {\em Cambridge Studies in Advanced Mathematics}.
\newblock Cambridge University Press, Cambridge, 2020.
\newblock From the Hamiltonian viewpoint, With an appendix by Igor Zelenko.

\bibitem[ABBR25]{ABBR2025}
A.~A. Agrachev, S.~Baranzini, E.~Bellini, and L.~Rizzi.
\newblock Quantitative tightness for three-dimensional contact manifolds: a sub-{R}iemannian approach.
\newblock {\em Nonlinearity}, 38(11):Paper No. 115011, 62, 2025.

\bibitem[ABPT26]{Lucia}
R.~Adami, U.~Boscain, D.~Prandi, and L.~Tessarolo.
\newblock Schr\"{o}dinger evolution on surfaces in 3{D} contact sub-{R}iemannian manifolds.
\newblock {\em J. Differential Equations}, 453(part 5):Paper No. 113915, 29, 2026.

\bibitem[Agr96]{agrexp}
A.~A. Agrachev.
\newblock Exponential mappings for contact sub-{R}iemannian structures.
\newblock {\em J. Dynam. Control Systems}, 2(3):321--358, 1996.

\bibitem[AGZV12]{arnold}
V.~I. Arnold, S.~M. Gusein-Zade, and A.~N. Varchenko.
\newblock {\em Singularities of differentiable maps. {V}olume 1}.
\newblock Modern Birkh\"{a}user Classics. Birkh\"{a}user/Springer, New York, 1985 edition, 2012.
\newblock Classification of critical points, caustics and wave fronts, Translated from the Russian by Ian Porteous based on a previous translation by Mark Reynolds.

\bibitem[AS10]{book_control}
A.~A. Agrachev and Y.~L. Sachkov.
\newblock {\em Control Theory from the Geometric Viewpoint}.
\newblock Springer Berlin, Heidelberg, 2010.

\bibitem[Bal03]{Bal03}
Z.~M. Balogh.
\newblock Size of characteristic sets and functions with prescribed gradient.
\newblock {\em Journal f\"{u}r die Reine und Angewandte Mathematik. [Crelle's Journal]}, 564:63--83, 2003.

\bibitem[BB23]{BellBosc}
E.~Bellini and U.~Boscain.
\newblock Surfaces of genus $g\geq 1$ in 3{D}-contact sub-{R}iemannian manifolds.
\newblock {\em ESAIM: COCV}, 29(79), 2023.

\bibitem[BB24]{BarBoss2024}
D.~Barilari and T.~Bossio.
\newblock Steiner and tube formulae in 3{D} contact sub-{R}iemannian geometry.
\newblock {\em Commun. Contemp. Math.}, 26(7):Paper No. 2350034, 40, 2024.

\bibitem[BBC22]{Daniele}
D.~Barilari, U.~Boscain, and D.~Cannarsa.
\newblock On the induced geometry on surfaces in 3{D} contact sub-{R}iemannian manifolds.
\newblock {\em Control Optim. Calc. Var. 28}, 109, 2022.

\bibitem[BBP25]{BBP25}
D.~Barilari, E.~Bellini, and A.~Pinamonti.
\newblock Curvature measures and the sub-{R}iemannian {G}auss-{B}onnet theorem.
\newblock {\em ArXiv Preprint 2509.26460}, 2025.

\bibitem[Bel26]{Bellini2026}
E.~Bellini.
\newblock On the tightness of left-invariant contact structures.
\newblock {\em Bull. Lond. Math. Soc.}, 58(6):Paper No. e70402, 12, 2026.

\bibitem[BFF{\etalchar{+}}15]{BaloghSteiner}
Z.~M. Balogh, F.~Ferrari, B.~Franchi, E.~Vecchi, and K.~Wildrick.
\newblock Steiner's formula in the {H}eisenberg group.
\newblock {\em Nonlinear Anal.}, 126:201--217, 2015.

\bibitem[BR19]{BarRiz2019}
D.~Barilari and L.~Rizzi.
\newblock Sub-{R}iemannian interpolation inequalities.
\newblock {\em Invent. Math.}, 215(3):977--1038, 2019.

\bibitem[BRR26]{BossioRizziRossi2024}
T.~Bossio, L.~Rizzi, and T.~Rossi.
\newblock Tubes in sub-{R}iemannian geometry and a {W}eyl's invariance result for curves in the {H}eisenberg groups.
\newblock {\em Annali della Scuola Normale Superiore di Pisa, Classe di Scienze}, 2026.
\newblock To appear.

\bibitem[dC92]{DoCarmo}
M.~P. do~Carmo.
\newblock {\em Riemannian Geometry}.
\newblock Mathematics: Theory and Applications. Birkhäuser Boston, MA, 1992.

\bibitem[Fed59]{FedererCM}
H.~Federer.
\newblock Curvature measures.
\newblock {\em Trans. Amer. Math. Soc.}, 93:418--491, 1959.

\bibitem[GGPT03]{GPT2003}
R.~Giamb\`o, F.~Giannoni, P.~Piccione, and D.~V. Tausk.
\newblock Morse theory for normal geodesics in sub-{R}iemannian manifolds with codimension one distributions.
\newblock {\em Topol. Methods Nonlinear Anal.}, 21(2):273--291, 2003.

\bibitem[GKOS01]{Grosser}
M.~Grosser, M.~Kunzinger, M.~Oberguggenberger, and R.~Steinbauer.
\newblock {\em Geometric theory of generalized functions with applications to general relativity}, volume 537 of {\em Mathematics and its Applications}.
\newblock Kluwer Academic Publishers, Dordrecht, 2001.

\bibitem[Gra04]{GrayTubes}
A.~Gray.
\newblock {\em Tubes}, volume 221 of {\em Progress in Mathematics}.
\newblock Birkh\"{a}user Verlag, Basel, second edition, 2004.
\newblock With a preface by Vicente Miquel.

\bibitem[Ish05]{Ishikawa2005}
G.~Ishikawa.
\newblock Infinitesimal deformations and stability of singular {Legendre} submanifolds.
\newblock {\em Asian Journal of Mathematics}, 9(1):133--166, 2005.

\bibitem[KRS{\etalchar{+}}05]{SingularFronts}
M.~Kokubu, W.~Rossman, K.~Saji, M.~Umehara, and K.~Yamada.
\newblock Singularities of flat fronts in hyperbolic space.
\newblock {\em Pacific J. Math.}, 221(2):303--351, 2005.

\bibitem[Rit21]{RitoreTubes}
M.~Ritor\'{e}.
\newblock Tubular neighborhoods in the sub-{R}iemannian {H}eisenberg groups.
\newblock {\em Adv. Calc. Var.}, 14(1):1--36, 2021.

\bibitem[Ros23]{RossiIntegrability}
T.~Rossi.
\newblock Integrability of the sub-{R}iemannian mean curvature at degenerate characteristic points in the {H}eisenberg group.
\newblock {\em Adv. Calc. Var.}, 16(1):99--110, 2023.

\bibitem[RR21]{RizziRossi}
L.~Rizzi and T.~Rossi.
\newblock Heat content asymptotics for sub-riemannian manifolds.
\newblock {\em Journal de Mathématiques Pures et Appliquées}, 148:267--307, 2021.

\bibitem[RS17]{RizziSerres}
L.~Rizzi and U.~Serres.
\newblock On the cut locus of free, step two {C}arnot groups.
\newblock {\em Proc. Amer. Math. Soc.}, 145(12):5341--5357, 2017.

\bibitem[Ste40]{Steiner}
J.~Steiner.
\newblock \"uber parallele fl\"achen.
\newblock {\em Monatsber. Preu. Akad. Wiss}, page 114–118, 1840.

\bibitem[Wey39]{WeylTubes}
H.~Weyl.
\newblock On the {V}olume of {T}ubes.
\newblock {\em Amer. J. Math.}, 61(2):461--472, 1939.

\end{thebibliography}

\end{document}